\documentclass{amsart}

\usepackage{etex}
\usepackage{bm}
\usepackage{amssymb}
\usepackage[stretch=10,shrink=10]{microtype}
\usepackage{tikz}
\usetikzlibrary{hobby}
\usepackage{enumitem}
\usepackage{theoremref}
\usepackage{comment}
\usepackage[showframe=false, margin = 1.5 in]{geometry}
\usepackage{pgfplots}
\usepackage{hyperref}
\usepackage{pdfpages}
\usepackage{graphicx}
\usepackage{array} 

\newcommand{\good}{\textit{Good}}

\newcommand{\eps}{\varepsilon}

\newcommand{\E}{\mathbf{E}}
\renewcommand{\P}{\mathbf{P}}
\newcommand{\prob}{\mathbf{P}}

\newcommand{\dir}{\textup{Dir}}
\newcommand{\arc}{\textup{Arc}}
\newcommand{\Z}{\mathbb Z}
\newcommand{\N}{\mathbb N}
\newcommand{\R}{\mathbb R}

\usepackage{color}
\definecolor{dark_green}{RGB}{1, 180, 1}

\usepackage{microtype}
\usepackage{mathrsfs}
\tikzset{vert/.style={circle,fill,inner sep=0,
    minimum size=0.15cm,draw}, nerve/.style={circle,inner sep=0,
    minimum size=0.15cm,draw},
    rightnerve/.style={circle,inner sep=0,
    minimum size=0.3cm,draw, fill}}

\newcommand{\ball}{\textup{Ball}}

\newcommand{\Ec}{\mathcal E}
\newcommand{\geo}{{\textup{Geo} }}

\newcommand{\be}{\mathbf{e}}

\theoremstyle{plain}
\newtheorem{thm}{Theorem}[section]
\newtheorem{theorem}[thm]{Theorem}
\newtheorem{lemma}[thm]{Lemma}
\newtheorem{prop}[thm]{Proposition}
\newtheorem{cor}[thm]{Corollary}

\newtheorem{definition}[thm]{Definition}

\theoremstyle{remark}

\tikzset{every tree node/.style={align=center}}
\usepackage{url}

\def\tree{{\mathscr T}}
\def\path{{\mathscr P}}

\def\functionals{{\mathscr C}}

\def\edges{{\mathcal E}}
\def\dedges{\overline{\mathcal E}}

\newcounter{saveenum}

\begin{document}

\title{The number of geodesics in planar first-passage percolation grows sublinearly}

\author{Daniel Ahlberg, Jack Hanson and Christopher Hoffman}

\thanks{D.A.~was in part supported by the Swedish Research Council (VR) through grant 2021-03964 and the Ruth and Nils-Erik Stenb\"ack Foundation. J.H.~was supported by the NSF grants DMS-1954257 and DMS-2002388 and by a CUNY-JFRASE grant (through the Sloan Foundation).}

\begin{abstract}
We study a random perturbation of the Euclidean plane, and show that it is unlikely that the distance-minimizing path between the two points can be extended into an infinite distance-minimizing path. More precisely, we study a large class of planar first-passage percolation models and show that the probability that a given site is visited by an infinite geodesic starting at the origin tends to zero uniformly with the distance.
In particular, this show that the collection of infinite geodesics starting at the origin covers a negligible fraction of the plane.
This provides the first progress on the `highways and byways' problem, posed by Hammersley and Welsh in the 1960s.
\end{abstract}

\maketitle

\section{Introduction}\label{sec:intro}

In the Euclidean plane, the distance-minimizing curve between two points is given by a line segment, and each line segment can be extended into a two-sided distance minimizing curve -- the straight line. The analogous statement remains true in hyperbolic geometry. Having zero curvature, Euclidean geometry lies at the rim between elliptic and hyperbolic geometry. Euclidean and hyperbolic geometry are distinguished by the postulate that in Euclidean geometry, given a straight line and a point not on that line, there is one and only one line through that point which is parallel to the first.
Euclidean geometry is critical in this regard, and it seems plausible that `parallel' distance-minimizing curves should collapse into one in a random perturbation of the Euclidean metric, and two-sided distance-minimizing curves should no longer exist.
Various random metric spaces considered in the literature are globally described by a deterministic norm $\mu:\R^2\to[0,\infty)$, equivalent to Euclidean distance on $\R^2$, in the sense that the distance between the origin and $x\in\R^2$ is $\mu(x)+o(|x|)$ with probability tending to one. Models of this type may be thought of as such random perturbations, and first-passage percolation on $\Z^2$ may be the most well-known such model.

Random perturbations of the Euclidean metric such as first-passage percolation are believed to exhibit asymptotic scaling as predicted by Kardar, Parisi and Zhang~\cite{karparzha86}.
Charles Newman has elaborated a convincing heuristic argument for why two-sided distance-minimizing paths (also known as bi-infinite geodesics or bigeodesics) should not exist in models that exhibit KPZ-scaling; see~\cite{aufdamhan17}. Newman's argument has recently been adopted to rule out the existence of bigeodesics in related so-called `integrable' models of last-passage percolation, where KPZ-scaling has been established, by Basu, Hoffman and Sly~\cite{bashofsly}; see also the recent work of Bal\'azs, Busani and Sepp\"al\"ainen~\cite{balbussep20}.
(In contrast, Benjamini and Tessera~\cite{bentes17} have established existence of bigeodesics in an hyperbolic setting.)
While the description of these models is impressively precise, the connections their analysis rests upon are very fragile. Apart from the handful integrable models amenable to explicit computations, we completely lack techniques that appropriately capture the behaviour of the vast majority of models thought to belong to the KPZ universe.

In this paper we address the related question, whether in a random perturbation of the Euclidean metric the distance-minimizing curve between two given points may be extended into a one-sided distance-minimizing curve. Except for degenerate cases, this probability is non-zero, and the question is whether the probability may remain bounded away from zero as the distance between the two points grow. Coupier~\cite{coupier18} showed for certain geometric random trees, including isotropic and/or integrable versions of first- and last-passage percolation, that this is not the case, and that the probability decays to zero with the distance. In these integrable and isotropic models, the shape function $\mu:\R^2\to[0,\infty)$ is either a multiple of the Euclidean norm or known to be strictly convex/concave and differentiable. However, in absence of assumptions of strict convexity and differentiability of the shape functions, this problem remains open \footnote{Coupier claims in the statement of his Theorem 7 that he can show the analogue of our Theorem 1 for general LPP models whose limiting shape is strictly concave. His argument, however, appears to use more unproven assumptions than this. His Proposition 8, cited to [11], does not appear to be established under only an assumption of strict concavity; in the case that the asymptotic shape is not differentiable, one could expect the analogue of Proposition 8 to fail. Models of FPP with ergodic weights, for which the analogue of Proposition 8 fails, have been constructed; see \cite{aleber18}}. 

This is not only a technical point; showing differentiability and strict convexity of $\mu$ for typical `non-solvable' models remains an important open problem, and it is known that one can construct translation-ergodic measures for which these properties do not hold. The objective of this paper is to show that Coupier's result give evidence for a much more general phenomenon, by establishing the analogous result in the context of first-passage percolation, where very little is known about $\mu$ beyond convexity and compactness of its unit ball. Our approach requires no additional assumption on the shape, and extends to the ergodic setting.

Our main result is related to, but different from, the so-called `midpoint problem' posed to Benjamini, Kalai and Schramm~\cite{benkalsch03}. The midpoint problem can be thought of as a preliminary step towards the question of bigeodesics, and was resolved by Ahlberg and Hoffman~\cite{ahlhof}. An alternative solution was provided earlier by Damron and Hanson~\cite{damhan17}, assuming differentiability of the asypmtotic shape. Recent work of Dembin, Elboim and Peled~\cite{demelbpel} provides a quantitative solution to the same problem under the assumption that the shape is not a polygon with too few sides.

\subsection{Model description and result}

First-passage percolation was first introduced in the 1960s, in work of Eden~\cite{eden61} and Hammersley and Welsh~\cite{hamwel65}, and has since become an archetype in the realm of spatial growth. It is non-integrable in nature, and decades of study has made clear that it is a very challenging model to understand. However, the study of first-passage percolation has also led to the development of powerful mathematical techniques, such as an ergodic theory for subadditive sequences~\cite{kingman68}, applicable in a wide range of contexts.

Much past work on first-passage percolation has considered the model on the two (or higher) dimensional integer lattice. This is the graph $(\Z^2,\Ec^2)$, where $\Ec^2$ consists of edges of the form $\{x, x+ \be_i\}$ for $i = 1, 2$, where $\be_i$ is the $i$th standard basis vector. A sequence $(v_0,e_1,v_1,\ldots,e_n,v_n)$, alternating between vertices and edges, and where each edge $e_i$ has the form $e_i=\{v_{i-1},v_i\}$, will be referred to a {\bf lattice} or {\bf nearest-neighbor path}, or simply as a {\bf path}. This definition is extended to the case of singly infinite paths (indexed by $\N$) and doubly infinite paths (indexed by $\Z$) in the obvious way. We shall often abuse notation and identify a path with either its set of vertices or its set of edges, depending on the context. Moreover, we shall write
$$
\|x\|_1:=|x_1|+|x_2|,\quad|x|:=(x_1^2+x_2^2)^{1/2}\quad\text{and}\quad\|x\|_\infty:=\max(|x_1|,|x_2|)
$$
for the usual $\ell_1$-, $\ell_2$- and $\ell_\infty$-norms on $\R^2$.

To construct the model, the edges (or sites) of the integer lattice are equipped with nonnegative random weights $(\omega_e)_{e\in\Ec^2}$. These weights give an elegant formulation of the model as a random (pseudo-)metric space $(\Z^2,T)$, in which the distance between two points $x,y\in\Z^2$ is given by the minimal weight-sum among all nearest-neighbour paths connecting $x$ to $y$. That is, for any path $\pi$ and $x,y\in\Z^2$, we define 
\begin{equation}\label{eq:T}
\textstyle{T(\pi):=\sum_{e\in\pi}\omega_e\quad\text{and}\quad T(x,y):=\inf\big\{T(\pi):\pi\text{ connecting $x$ to }y\big\}.}
\end{equation}
In the above context we may interpret the weights as the `passage times' of a growing entity. From this perspective, balls $\ball(t):=\{z\in\Z^2:T(0,z)\le t\}$ in the metric space have the meaning of `propagation over time' of the same entity. The asymptotic behaviour of distances, balls and geodesics, i.e.\ distance-minimizing connections between points in the metric space, are the primary focus of study.

On a formal note, the edge-weight configuration $(\omega_e)_{e\in\Ec^2}$ is drawn from the sample space $\Omega_1=[0,\infty)^{\Ec^2}$ according to $\P$, which equipped with the Borel sigma algebra forms a probability space $(\Omega_1,\mathcal{F},\P)$. We shall in this paper work with a large class of edge-weight distributions which includes those with independent edge weights from a common continuous distribution with finite mean. (We shall be more precise on the exact conditions in Section~\ref{sec:geodesics}.)
In this setting, the infimum in~\eqref{eq:T} is known to be uniquely attained, almost surely, and we denote by $\geo(x,y)$ the path $\pi$ for which $T(\pi)=T(x,y)$.
An infinite path $(v_0,v_1,\ldots)$ with the property that $(v_0,v_1,\ldots,v_n)$ coincides with $\geo(v_0,v_n)$ for all $n\ge1$ will be referred to as an {\bf infinite} or {\bf singly infinite geodesic}. For $v\in\Z^2$ we let $\tree_v=\tree_v(\omega)$ be defined as
\begin{equation}\label{eq:tree}
\tree_v:=\{\text{infinite geodesics originating at }v\}
\end{equation}
denote the collection of infinite geodesics $(v_0,v_1,\ldots)$ such that $v_0=v$. We shall often identify $\tree_v$ with the subgraph of $(\Z^2,\Ec^2)$ obtained as the union of its elements. From this perspective, it is straightforward to verify that $\tree_v$ can be described as the graph limit\footnote{A sequence $(G_n)_{n\ge1}$ of subgraphs of $(\Z^2,\Ec^2)$ is said to converge to a subgraph $G$ if for every $m$ there exists $n_0$ so that the restrictions of $G_n$ and $G$ to $[-m,m]^2$ coincide for all $n\ge n_0$.}
$$
\tree_v=\lim_{n\to\infty}\bigcup_{x\in v+\partial B(n)}\geo(v,x),
$$
where $B(n):=\{x\in\Z^2:\|x\|_\infty\le n\}$ and $\partial B(n):=\{x\in\Z^2:\|x\|_\infty=n\}$. Under our standing assumptions, the graph $\tree_v$ is almost surely a tree, and every (finite or infinite) segment of a branch in $\tree_v$ starting at $v$ is a geodesic.

Taking the perspective that $\tree_v$ is a random tree, the symbols `$y\in\tree_v$' will be given the meaning that $y$ is a vertex of this graph, which is equivalent with the statement that $y\in\gamma$ from some infinite geodesic $\gamma\in\tree_v$. Our main theorem states that the density of the tree $\tree_v$ as a subset of the plane is zero in a uniform sense. Due to translation invariance it will suffice to consider the geodesic tree rooted at zero.

\begin{theorem}\label{thm:highways}
For first-passage percolation on $\Z^2$ we have that
$$
\limsup_{|v|\to\infty}\P\big(v\in\tree_0\big)=0.
$$
\end{theorem}

Let us briefly relate the statement of the above theorem to the preceding discussion. For any $v\in\Z^2$ it is known that there exists an almost surely unique finite geodesic $\geo(0,v)$ connecting the origin to $v$. The theorem implies that for every $\eps>0$ there exists an $n_0$ such that the probability that either $v\in\tree_0$ or $0\in\tree_v$, and thus that $\geo(0,v)$ can be extended into a one-sided infinite geodesic, is at most $\eps$ whenever $|v|>n_0$.

The above theorem further provides the first progress on the so-called `highways and byways problem' due to Hammersley and Welsh~\cite{hamwel65}. Hammersley and Welsh referred to an edge $e\in\Ec^2$ as an {\bf highway arc} if $e$ is traversed by $\geo(0,z)$ for infinitely many $z\in\Z^2$, and as a {\bf byway arc} if $e$ is traversed by $\geo(0,z)$ for finitely many $z$. Note that an edge is an highway arc if and only if $e\in g$ for some $g\in\tree_0$. Let $f_r$ denote the number of highway arcs that intersect the circumference of the disc $\{x^2+y^2=r^2\}$. Hammersley and Welsh asked whether $f_r\to\infty$ as $r\to\infty$, and if so, at what rate? The best known general lower bound gives that $f_r\ge 4$ for large $r$, but also that $f_r\to\infty$ as long as the asymptotic shape is not a polygon. It is not known whether there are continuous weight distributions for which the asymptotic shape is non-polygonal. As an immediate consequence of Theorem~\ref{thm:highways}, without assumptions on the shape, we obtain the following.

\begin{cor}\label{cor:highways}
For first-passage percolation on $\Z^2$ we have that
$$
\E[f_r]=o(r).
$$
\end{cor}

This work is a continuation of recent work developing an ergodic theory for the study of infinite geodesics in first-passage percolation. To prove Theorem~\ref{thm:highways} we will apply an approach based on the convergence of geodesic measures, introduced by Damron and Hanson~\cite{damhan14}, together with elements from the ergodic theory for infinite geodesics developed by Ahlberg and Hoffman~\cite{ahlhof}. We mention that a similar approach was used in~\cite{ahlhof} in order to resolve the `midpoint problem' of Benjamini, Kalai and Schramm~\cite{benkalsch03}.

\subsection{The asymptotic shape and the shape theorem}

Central in the study of first-passage percolation is the existence of a `shape theorem', describing the first-order asymptotics of distances and balls in the first-passage metric. A first step towards the shape theorem was made by Kingman~\cite{kingman68}, who from his subadditive ergodic theorem obtained the existence of a norm $\mu:\R^2\to[0,\infty)$, referred to as the {\bf time constant}, such that for every $z\in\Z^2$ we have
\begin{equation}\label{eq:timeconstant}
\lim_{n\to\infty}\frac1nT(0,nz)=\mu(z),\quad\text{almost surely}.
\end{equation}
In the setting of i.i.d.\ exponential edge weights, Richardson~\cite{richardson73} extended the \emph{radial} convergence in~\eqref{eq:timeconstant} to \emph{simultaneous} convergence in all directions,  proving the first version of the shape theorem. Later work of Cox and Durrett~\cite{coxdur81} gave both necessary and sufficient conditions under which its conclusion, that
\begin{equation}\label{eq:shapethm}
\limsup_{|z|\to\infty}\frac{1}{|z|}\big|T(0,z)-\mu(z)\big|=0\quad\text{almost surely},
\end{equation}
holds. In the stationary ergodic setting, a corresponding result was obtained by Boivin~\cite{boivin90}.

The shape theorem is above described in terms of distances, but can equally well be be phrased as an asymptotic result for large balls in the first-passage metric. Indeed, this may well be the more familiar version of the theorem. As such, the theorem relates the ball $\ball(t)=\{z\in\Z^2:T(0,z)\le t\}$ of the random metric to the unit ball in the norm $\mu$ on $\R^2$, referred to as the {\bf asymptotic shape}, as which we shall denote by
\begin{equation}\label{eq:ball}
\ball:=\{x\in\R^2:\mu(x)\le1\}.
\end{equation}
Since $\mu:\R^2\to[0,\infty)$ defines a norm on $\R^2$, it follows from the subadditive property that the asymptotic shape is convex. In the i.i.d.\ setting, $\ball$ obeys the symmetries of the $\Z^2$ lattice fixing the origin, and Kesten~\cite{kesten86} has showed that $\ball$ is compact as long as the edge-weight distribution puts mass strictly less than 1/2 at zero-weight edges. Under the assumptions stipulated in this paper, the shape will enjoy all of these properties.

Beyond the elementary properties of the shape described above, very little general information is available regarding the geometry of $\ball$. The lack of a more qualitative description of the shape has been an impediment for a more detailed understanding of the geometry of geodesics. This led Newman~\cite{newman95}, in his pioneering work on geodesics, to introduce an unproven condition of uniform curvature. The uniform curvature condition implies that the asymptotic shape is both strictly convex and its boundary differentiable. Although widely believed to be true (in the i.i.d.\ setting), both strict convexity and differentiability remain unproven to this day. As we shall see, the main obstacle in the present paper is the possible existence of `corners', i.e.\ points of non-differentiability, of the asymptotic shape, and we shall have to work hard in order to circumvent this difficulty.

\subsection{Outline of the proof}
We outline here our arguments under the additional assumption that the boundary $\partial\ball:=\{x\in\R^2:\mu(x)=1\}$ of the asymptotic shape is differentiable in every direction, meaning that there is a unique supporting line to $\ball$ through every point $x\in\partial\ball$, i.e.\ the tangent line. We emphasize that this assumption allows us to dramatically simplify our arguments --- indeed, the case that $\partial \ball$ is differentiable is handled by a relatively short argument at Lemma~\ref{corners} below, and does not require us to construct geodesic measures. We choose to focus on this simpler case to give the reader some insight without getting into the complications introduced when considering distributions with less regular $\ball$.

We shall argue by contradiction. So, suppose that for some $\delta>0$ and some diverging sequence $(v_k)_{k\ge1}$ of vertices of $\Z^2$ we have uniformly in $k$ that
\begin{equation}\label{eq:contra}
\P\big(v_k\in\tree_0\big)>\delta.
\end{equation}
By restricting to a subsequence we may further assume that $v_k/|v_k|\to v$ for some unit vector $v\in\R^2$. By symmetry and translation invariance, the assumption in~\eqref{eq:contra} is equivalent to
\begin{equation}\label{eq:contra2}
\P\big(0\in\tree_{v_k}\big)>\delta.
\end{equation}

The tangent lines of $\ball$ induce a partition of $\partial\ball$, and by projecting these sets on the unit sphere $S^1:=\{x\in\R^2:|x|=1\}$ we obtain a partition of $S^1$. (The sets in the partition are closed intervals or consist of a single point.) The vector $v$ is contained in precisely one of these sets; denote this set by $A_v$. Under the additional and unproven assumption that the shape is differentiable, the following facts are consequences of the work of Damron and Hanson~\cite{damhan14,damhan17}: for fixed $v$ as above, almost surely
\begin{itemize}
\item there is a unique geodesic $g=(u_1,u_2,\ldots)$ in $\tree_0$ with `asymptotic direction' in $A_v$, meaning that the set of limit points of the sequence $(u_k/|u_k|)_{k\ge1}$ is contained in $A_v$;
\item the geodesic $g$ is not part of a doubly infinite geodesic;
\item there is a sequence $(g_n^-)_{n\ge1}$ of geodesics in $\tree_0$ converging to $g$ in a counter-clockwise motion, and a sequence of geodesics $(g_n^+)_{n\ge1}$ converging to $g$ in a clockwise motion.
\end{itemize}
The convergence alluded to above is in the usual sense of path convergence, where 
\begin{equation}
\label{eq:convdef}\begin{minipage}{0.9\textwidth}
\begin{center}
   a sequence $(\pi_n)_{n\ge1}$ of (finite or infinite) paths is said to converge to a
path $\pi$ if for every $m$ there exists $n_0$ such that the initial segments of length $m$
of $\pi_n$ and $\pi$ coincide for all $n\ge n_0$.
\end{center}
    \end{minipage}
\end{equation}

On the above almost sure event, we may for every $m$ choose $n$ large so that $g_n^-$ and $g_n^+$ coincide with the first $m$ steps of $g$.
For each $n$ the geodesics $g_n^-$ and $g_n^+$ are (asymptotically) directed clockwise and counter-clockwise of $v$, respectively. Consequently, for every $n$ we can find $k_0$ so that for all $k\ge k_0$ the segment $\geo(0,v_k)$ is ``trapped'' in the region between $g_n^-$ and $g_n^+$.
By~\eqref{eq:contra2} and the ergodic theorem, we may therefore, with positive probability, find a sequence $(\gamma_m)_{m\ge1}$ of singly infinite geodesics that visits the origin, and whose first $m$ steps in reverse order from the origin coincide with the first $m$ steps of some (i.e.\ the almost surely unique) geodesic $g\in\tree_0$.
Taking limits as $m\to\infty$, possibly along a subsequence, $(\gamma_m)_{m\ge1}$ will converge to a doubly infinite geodesic through the origin, which in one direction coincides with $g$. This contradicts the fact that $g$ is almost surely {\em not} extendable to a doubly infinite geodesic, and thus proves Theorem~\ref{thm:highways}.

In the argument outlined above we have imposed the additional unproven assumption that the shape is differentiable at every point on the boundary, in which case methods from~\cite{damhan14,damhan17} apply. Using methods from~\cite{ahlhof} we can make the above argument rigorous without assumptions on the shape, as long as $v$ is a direction of differentiability.
The above outline hides a story about coalescence among geodesics with a similar asymptotic behaviour. For a given direction of differentiability, geodesics starting from different point and with this direction will eventually coalesce. Being non-extendable into a doubly infinite geodesic is a consequence of this coalescence property.
In order to prove Theorem~\ref{thm:highways}, without assumptions on the shape, we shall, in essence, rely on the significantly stronger statement claiming that a `typical' geodesic is coalescing. 
The statement can be made precise in terms of shift-invariant measures on families of geodesics; see Theorem~\ref{thm:non-crossing}.
We shall make the argument rigorous by combining the more refined methods from~\cite{ahlhof} with the approach of geodesic measures of~\cite{damhan14} (discussed in more detail in Section~\ref{sec:weakconv}).

The geodesic measures constructed in the present paper are translation-invariant distributions from which we may sample families of infinite geodesics.  Roughly speaking, these infinite geodesics will arise as limits of finite geodesics of the form $\geo(x, x + v_k)$ for appropriate $x$ and $v_k$.
Under the assumption \eqref{eq:contra} above, we will be able to construct our measures in such a way that, with positive probability, the collection of geodesics sampled from them will not coalesce. Indeed, the subgraph of $\mathbb{Z}^2$ induced by this collection will have infinitely many components or ``coalescence classes''. On the other hand, the Busemann functions of many of these geodesics --- relative distances to infinity along these geodesics --- will have the same asymptotics, and this will contradict their non-coalescence, as argued in Theorem~\ref{thm:non-crossing}.



\subsection{Outline of the paper}

The paper is outlined as follows. We shall first, in Section~\ref{sec:geodesics}, review the state-of-the-art in the study of geodesics in first-passage percolation. At the end of this section we describe in detail under which assumptions these results are known, and the precise setting in which we shall work. In Section~\ref{sec:ergodictheory} we proceed with an overview of the methods used in the study of geodesics, which shall also be required for the proof. Along the way we establish some key lemmas that will be used for the proof of our main theorem. We then, finally, prove Theorem~\ref{thm:highways} in Section~\ref{sec:thehighways}.

\section{Geodesics in first-passage percolation} \label{sec:geodesics}

Questions regarding the structure of geodesics in first-passage percolation were raised already in the foundational work of Hammersley and Welsh~\cite{hamwel65}, and our main theorem addresses an open problem proposed there. However, the modern study of infinite geodesics began in the mid 1990s with work of Newman~\cite{newman95}, together with Licea and Piza~\cite{licnewpiz96,licnew96,newpiz95}. These papers began a systematic study of these objects and laid out important conjectures that have motivated much subsequent work. In this section we discuss these and  more recent results, and we provide several lemmas which will be useful in our later arguments.

To simplify the presentation, we defer the statement of the precise assumptions on the distribution of the edge weights $(\omega_e)$ until Section~\ref{sec:assns}. Unless otherwise stated, any past result quoted in this section is valid under either of the two sets {\bf A1} and {\bf A2} of assumptions from that section. For the moment, we remind the reader that these assumptions are enough to establish the shape theorem in the form of \eqref{eq:shapethm} and the almost sure existence and uniqueness of the finite geodesics $\geo(x,y)$ for each pair $x, y \in \Z^2$. For instance, the class of distributions includes every continuous distribution with finite mean.

\subsection{Geodesics according to Newman and collaborators} \label{sec:newmangeo}

In his paper inaugurating the modern study of geodesics, Newman~\cite{newman95} addressed the geodesic structure of the first-passage metric by posing questions regarding the number and straightness of infinite branches in the graph obtained from the union of all finite geodesics with an endpoint at the origin, i.e.\ the set $\tree_0$ of infinite geodesics. In order to describe Newman's results in some detail, we recall the meaning of asymptotic direction and coalescence of paths.

Given an infinite self-avoiding path $\pi = (v_i)_{i=0}^\infty$, we define its {\bf asymptotic direction} $\dir(\pi)$:
\begin{equation}
\label{eq:dirgdef}
\dir(\pi) \text{ is the set of all limit points of the sequence $(v_i / |v_i|)$ in $ S^1$},
\end{equation}
where, as before, $S^1=\{x\in\R^2:|x|=1\}$ is the Euclidean unit circle in $\R^2$. We shall identify $S^1$ with $[0, 2 \pi)$ in the natural way --- for instance, we shall write $\{\theta\} = \dir(\pi)$ when we mean that $v_i / |v_i| \to (\cos\theta,\sin\theta)$.
In case $\dir(\pi) = \{\theta\}$ is a singleton, we say that $\pi$ has {\bf asymptotic direction} $\theta$. 

Two infinite self-avoiding paths $\pi = (v_i)_{i=0}^\infty$ and $\pi' = (v_j')_{j=0}^\infty$ are said to {\bf coalesce} if
\begin{equation}
\label{eq:coaldef}
\text{for some $I, J < \infty$, the terminal segments } (v_i)_{i = I}^\infty,\, (v_j')_{j=J}^\infty \text{ are equal.}
\end{equation} 




Several credible predictions can be drawn from Newman's~\cite{newman95} pioneering work on the geodesic structure of first-passage percolation on $\Z^2$. In the setting of i.i.d.\ edge-weights, drawn from some continuous distribution with finite mean (and possibly satisfying some weak additional regularity condition) the set $\tree_0$ should have the following properties:
\begin{enumerate}[label= (\roman*)]
\item \label{it:dirgeo} Almost surely, each $\gamma \in \tree_0$ has an asymptotic direction --- that is, $\dir(\gamma) = \{\theta\}$ for some $\theta \in S^1$.
\item \label{it:ungeo} For each fixed deterministic $\theta \in S^1$, there almost surely exists a unique geodesic in $\tree_0$ with direction $\theta$;
\item \label{it:coalgeos} For each fixed deterministic $\theta \in S^1$, almost surely, all infinite geodesics having direction $\theta$ coalesce. That is, for fixed $\theta\in S^1$, almost surely, for any $u,v\in\Z^2$ and geodesics $\gamma\in\tree_u$ and $\gamma'\in\tree_v$ with $\dir(\gamma)=\dir(\gamma')=\{\theta\}$ the symmetric difference $|\gamma\triangle\gamma'|$ is finite.
\end{enumerate}

In the papers \cite{licnew96,newman95}, Licea and Newman showed conditional versions of these predictions. The statement~\ref{it:dirgeo} was in~\cite{newman95} shown to hold under the additional unproven assumption that $\partial \ball$ is uniformly curved. This assumption is widely believed to hold for sufficiently nice i.i.d.~weight distributions, but to this day has not been established for a single example of i.i.d.\ weights. Given statement \ref{it:dirgeo}, a partial verification of \ref{it:ungeo}--\ref{it:coalgeos} was accomplished in~\cite{licnew96}. Namely, fixing a direction $\theta$ outside of some (non-explicit) set $D \subseteq S^1$ having zero Lebesgue measure, all $\theta$-directed geodesics from all vertices of $\Z^2$ must almost surely coalesce.

Much recent work on geodesics has been motivated by the technical obstacles encountered by the above approach --- for instance, circumventing the uniform curvature assumption, since establishing uniform curvature seems to be a difficult problem. In fact, uniform curvature and other properties used by Newman and collaborators (such as exponential tail estimates for $T(0,x)$) should not hold for many ergodic models of interest --- it is unclear how to run Newman's argument for directedness of geodesics under the more general hypothesis of ergodic edge-weights. A now-classical paper of H\"aggstr\"om and Meester~\cite{hagmee95} shows that any convex and compact set with the symmetries of $\Z^2$ appears as the asymptotic shape of some ergodic model of first-passage percolation, and the more recent works of Alexander and Berger~\cite{aleber18} and Brito and Hoffman~\cite{brihof21} provide examples of ergodic distributions with polygonal asymptotic shapes such that $\tree_0$ violates some aspect of \ref{it:dirgeo}--\ref{it:coalgeos} above. All of these considerations have motivated the approach we now describe, which studies geodesics via objects known as Busemann functions.

\subsection{Busemann functions}\label{sec:Busef}

If one were trying to prove conjecture \ref{it:dirgeo} from Section \ref{sec:newmangeo}, one might take the following approach: taking a sequence $(x_k)_{k\ge1}$ of points in $\Z^2$ such that $x_k/|x_k|\to\theta$ as $k\to\infty$, consider the limiting behavior of the sequence of finite geodesics $(\geo(0,x_k))_{k\ge1}$. One may expect that $\geo(0, x_k)$ converges to some $\theta$-directed $g \in \tree_0$, where convergence is in the usual sense of \eqref{eq:convdef}.
To show that the limit exists is nontrivial, and in general not known. However, a standard compactness argument shows that sequences of paths with the same initial point always have subsequential limits:
\begin{equation}
\label{eq:compactnessg}
\begin{gathered}
\text{If $(\pi_n)_{n\ge1}$ is a sequence of paths that share the same initial vertex, then}\\
\text{there exists a path $\pi$ and a strictly increasing sequence $(n_k)_{k\ge1}$ such that $\pi_{n_k}\to\pi$.}
\end{gathered}
\end{equation}
Moreover, it follows immediately from definitions that geodesicity is preserved under limits:
\begin{equation}
  \label{eq:geopreserve}
  \text{If $(g_n)_{n\ge1}$ is a sequence of geodesics and $g_n \to g$, then $g$ is a geodesic.}
\end{equation}
It follows that any limit of finite geodesics towards direction $\theta$ is indeed an element of $\tree_0$, but it is a priori not clear that this limit should be $\theta$-directed --- or indeed that the subsequential limits should change as we change $\theta$. Newman~\cite{newman95} proved that the limit exists, and indeed is $\theta$-directed, but his argument is conditional on the unproven curvature hypothesis.

Work to remove this assumption \cite{garmar05,hagpem98,hoffman05} began with a less ambitious goal: to show that the sequences $(\geo(0,n \be_1))_{n \geq 0}$ and $(\geo(0, -n \be_1))_{n \geq 0}$ should have distinct subsequential limits, since an infinite geodesic which is `fast' for travel in one direction should not also be `fast' for travel in the antipodal direction.
We describe here Hoffman's~\cite{hoffman05,hoffman08} method for making such arguments rigorous, which introduced so-called Busemann functions to first-passage percolation. These measure a form of relative passage time to a point at infinity picked out by a particular infinite geodesic. Precisely, if $g = (v_0,v_1,\ldots)$ is an element of $\tree_0$, its {\bf Busemann function} $B_g:\Z^2\times\Z^2\to\R$ is defined by the limit
\begin{equation}
\label{eq:busesdef}
B_g(y,z):=\lim_{k\to\infty}\big[T(y,v_k)-T(z,v_k)\big].
\end{equation}
This limit exists, by a monotonicity argument (see~\cite[Sec.~3]{hoffman05}), simultaneously for every $\omega \in \Omega_1$ and every $g\in\tree_0(\omega)$.

Some elementary but very useful properties of Busemann functions are collected in the following lemma. The lemma follows readily from the definitions and is now standard (see, e.g.~\cite[Sec.~5.1]{aufdamhan17}), so we omit the proof.

\begin{lemma}\label{lem:busemanpropps}
For every $\omega \in \Omega_1$ and every infinite geodesic $g = (v_i)_{i\ge0}$ (with arbitrary initial vertex), the following properties hold:
\begin{enumerate}[label = (\alph*)]
\item $B_g(x,z)=B_g(x,y)+B_g(y,z)$ for all $x,y,z$ in $\Z^2$; \label{it:badd}
\item $|B_g(y,z)|\le T(y,z)$ for all $y,z$ in $\Z^2$; \label{it:bleq}
\item $B_g(v_i,v_j)=T(v_i,v_j)$ for all $i < j$; \label{it:beq}
\item $B_g(x, v_0) \leq T(x, v_n) - T(v_0, v_n)$ for each $n \geq 0$; \label{it:buseprelim}
\item for every geodesic $g'$ such that $g$ and $g'$ coalesce, we have $B_{g'} = B_{g}$. \label{it:terminalseg}
\end{enumerate}  
\end{lemma}

As observed already in~\cite{hoffman05}, additivity of Busemann functions and the fact that the Busemann functions of coalescing geodesics coincide (properties (a) and (e) of the above lemma) are together responsible for another fundamental notion for Busemann functions --- that of asymptotic linearity.
Let $\rho:\R^2\to\R$ be a linear functional. Given a weight configuration $\omega \in \Omega_1$ and a geodesic $g\in\tree_0(\omega)$ with Busemann function $B_g$, we say that $B_g$ is {\bf asymptotically linear} to $\rho$ if
\begin{equation}\label{Busemann shape}
\limsup_{|z|\to\infty}\frac{\big|B_g(0,z)-\rho(z)\big|}{|z|}=0.
\end{equation}
As will be clear from Theorem~\ref{AH1} below, this behavior is in fact typical.
The asymptotic linearity in~\eqref{Busemann shape} is extremely useful for controlling the behavior of infinite geodesics. In particular, if $B_g$ is asymptotically linear, then $\dir(g)$ is constrained by properties of the asymptotic shape $\ball$.

Let us introduce some notation in order to make this idea precise.
In Lemma~\ref{lem:whatrho} below, we show that if $B_g$ is asymptotically linear to $\rho$, then the line\footnote{Here and below we shall identify a linear functional $\rho:\R^2\to\R$ with its gradient.} $\{x\in\R^2:\rho\cdot x=1\}$ is a supporting line to $\ball$. That is, only linear functionals such that $\{x\in\R^2:\rho\cdot x=1\}$ is a supporting line to $\ball$ may arise as a limit as in~\eqref{Busemann shape}. We shall henceforth refer to any linear functionals (not necessarily arising as such a limit) such that $\{x\in\R^2:\rho\cdot x=1\}$ is a supporting line to $\ball$ as {\bf supporting}, and to functionals in the subset of supporting functionals such that $\{x\in\R^2:\rho\cdot x=1\}$ is a tangent line to $\ball$, i.e. the unique supporting line at some point of $\partial\ball$, as {\bf tangent}. It is a standard fact from convex analysis that the set of supporting lines of a compact convex set in $\R^2$ is in 1-1 correspondence with $S^1$.

For each supporting linear functional $\rho: \R^2 \to \R$ we define
\begin{equation}
\label{eq:sectrho}
\arc(\rho) := \{x \in S^1: \, \rho \cdot x = \mu(x)\}.
\end{equation}
That is, the set $\arc(\rho)$ is obtained by taking the set of $y \in \partial \ball$ for which $\{x\in\R^2: \, \rho \cdot x = 1\}$ is a supporting line to $\ball$, and then projecting this set onto $S^1$ in the natural way. Hence, $\arc(\rho)$ is a closed arc of the form $[\theta_1, \theta_2] \subseteq S^1$.

If the Busemann function of $g$ is asymptotically linear to a particular $\rho$, this guarantees that $\dir(g)$ is a subset of $\arc(\rho)$; this is the content of the next lemma. Lemmas similar to it have appeared in~\cite{ahlhof,damhan14,hoffman08}. We include the proof for completeness.
 
\begin{lemma}\label{lem:whatrho}
	The following occurs almost surely: for each $g \in \tree_0$ and each linear functional $\rho$ such that $B_g$ is asymptotically linear to $\rho$ (as in \eqref{Busemann shape}),
	\begin{enumerate}[label = {(\alph*)} ]
	\item \label{item:whatrho1} $\dir(g) \subseteq \arc(\rho)$, and
	\item \label{item:whatrho2} $\{x \in \mathbb{R}^2: \,  \rho \cdot x = 1\}$ is a supporting line of $\ball$ (i.e.\ $\rho$ is a supporting functional).
 \end{enumerate}
\end{lemma}

\begin{proof}
	Consider an outcome $\omega$ in the event
	\begin{equation}\label{eq:shapey}
	\left\{\lim_{|x| \to \infty} \frac{T(0, x)}{\mu(x)} = 1 \right\} \cap \left\{ \geo(x,y) \text{ exists and is unique for each $x,y\in\Z^2$} \right\}\ ,
	\end{equation}
	which has probability one by assumptions {\bf A1} and {\bf A2}, defined below, and by the shape theorem \eqref{eq:shapethm}. In this outcome, fix $g$ and $\rho$ as in the statement of the lemma, and write $g = (0 = x_0, x_1, \ldots)$. By \eqref{eq:shapey}, we have $|x_i| \to \infty$. Let $z \in \dir(g)$, and choose a subsequence $(x_{n_k})$ of vertices of $g$ such that $x_{n_k}/|x_{n_k}| \to z$. Then we have
	\begin{align*}
	\rho \cdot z &= \lim_{k \to \infty} \frac{B_g(0, x_{n_k})}{|x_{n_k}|} \quad \text{(by asymptotic linearity of $B_g$)}\\
	&= \lim_{k \to \infty} \frac{T(0, x_{n_k})}{|x_{n_k}|} \quad \text{(by \ref{it:beq} of Lemma \ref{lem:busemanpropps})}\\
	&= \mu(z) \quad \text{(by occurrence of the event in \eqref{eq:shapey})}.
		\end{align*}
		In other words, $z \in \arc(\rho)$, proving claim \ref{item:whatrho1} of the lemma holds on the event in \eqref{eq:shapey}. In addition, we note that $z /\mu(z)$ is then an element of $\partial \ball$ with $\rho \cdot (z / \mu(z)) = 1$.
		
		On the other hand, let us consider any $y \in \partial \ball$, with $(y_n)$  any sequence of distinct vertices of $\Z^2$ such that $y_n / \mu(y_n) \to y$. For $\omega$, $g$, and $\rho$ as in the previous paragraph,
		\begin{align*}
		\rho \cdot y &= \lim_{k \to \infty} \frac{B_g(0,y_n)}{\mu(y_{n})}\\
		&\leq \limsup_{k \to \infty} \frac{T(0,y_n)}{\mu(y_n)} \quad \text{(by \ref{it:bleq} of Lemma \ref{lem:busemanpropps})}\\
		&= 1 \quad \text{(by occurrence of the event in \eqref{eq:shapey}).}
		\end{align*}
		In particular, $\rho \cdot x \leq 1$ for all $x \in \partial \ball$. But we have already seen that there is some $x \in \partial \ball$ with $\rho \cdot x = 1$; namely, $z / \mu(z)$ (where $z$ is as in the previous paragraph). Thus, the line $\{w \in \mathbb{R}^2: \,\rho \cdot w = 1\}$ is a supporting line to $\ball$ at $z / \mu(z)$. Since $\omega$ was an arbitrary element of the probability one event in \eqref{eq:shapey}, we have shown claim \ref{item:whatrho2} of the lemma.
\end{proof}

Hoffman used the above properties to show by contradiction that, a.s., there must exist at least two non-coalescing geodesics. In outline, if there were always only one infinite geodesic $G$ modulo coalescence, then its Busemann function $B_G$ would be a random variable which is translation-covariant, in the sense that  and \ref{it:badd} of Lemma \ref{lem:busemanpropps} and the ergodic theorem show that $B_G$ is asymptotically linear to some nonrandom $\rho$. But since $B_G$ must be invariant in distribution under symmetries of $\Z^2$ (because each $\tree_x$ is), we see that $\rho$ would be the zero functional, contradicting Lemma \ref{lem:whatrho}.

\subsection{Versions of Newman's conjectures}

In a follow-up work, Hoffman~\cite{hoffman08} associated Busemann functions to sides (tangent functionals) of $\ball$ and showed that $\tree_0$ has cardinality at least four. In pursuit of Newman's conjectures, Hoffman's approach was refined in later work of Damron and Hanson~\cite{damhan14}, who showed existence of geodesics with asymptotically linear Busemann functions, and this result remains the state-of-the-art regarding existence.

\begin{theorem}\label{DH1}
Let $\rho$ be an arbitrary fixed tangent functional. There exists, with probability one, a geodesic in $\tree_0$ whose Busemann function is asymptotically linear to $\rho$.
\end{theorem}

We mention also that Theorem~\ref{DH1} is a simplified synthetization of the results from~\cite{damhan14}, and we shall come back to this in the more detailed discussion in the next section.
Later work of Ahlberg and Hoffman~\cite{ahlhof} establishes the fact that all geodesics simultaneously have asymptotically linear Busemann functions. 

\begin{theorem}\label{AH1}
There is a probability one event on which the following holds: for every geodesic $g\in\tree_0$, there exists a supporting functional $\rho:\R^2\to\R$ such that the Busemann function of $g$ is asymptotically linear to $\rho$.
\end{theorem}

The above results regard linearity of Busemann functions, which we have seen to describe the asymptotic direction of geodesics. Theorems~\ref{DH1} and~\ref{AH1} can thus be thought of as a rigorous version of Newman's conjecture~\ref{it:dirgeo} --- at least if one is looking for a result valid at the level of general ergodic distributions. Indeed, as alluded to in Section~\ref{sec:newmangeo}, Newman's conjectures for i.i.d.~distributions do not hold for general translation-ergodic distributions: \cite{brihof21} provides an example of such a distribution where $|\tree_0| = 4$. In this example, $\ball$ has four sides, and the elements of $\tree_0$ have Busemann functions respectively asymptotically linear to functionals tangent to each of these four sides. 

The next result, from~\cite{ahlhof}, asserts that the linear functionals associated to geodesics in $\tree_0$ are both ergodic and unique, thus proving a version of Newman's conjecture~\ref{it:ungeo}.

\begin{theorem}\label{AH2}
There exists a compact deterministic set $\functionals \subseteq \R^2$ of linear functionals $\rho: \R^2 \to \R$ supporting to $\ball$ such that the following holds with probability one: the (a priori random) set of linear functionals 
\[\{\rho: \, \text{there is a $g \in \tree_0$ with $B_g$ asymptotically linear to $\rho$}\} \]
 equals $\functionals$. Moreover, for every $\rho\in\functionals$ we have
$$
\P\big(\text{there are two distinct geodesics in $\tree_0$ with Busemann function linear to }\rho\big)=0.
$$
\end{theorem}

The final result, also from~\cite{ahlhof}, establishes a version of Newman's conjecture \ref{it:coalgeos}.

\begin{theorem}\label{AH3}
Fix an arbitrary supporting linear functional $\rho\in\functionals$. For each pair $y, z \in \Z^2$,
\begin{equation*}
\P\left(\exists \,g \in \tree_y,\, g' \in \tree_z \text{ with } B_g, \, B_{g'} \text{ asymptotically linear to $\rho$ and } |g \triangle g'| = \infty  \right) = 0\ .
\end{equation*}
In other words, with probability one, all geodesics (from arbitrary starting points) with Busemann function asymptotically linear to $\rho$ coalesce.
\end{theorem}

Together, the stated theorems allow one to prove versions of Newman's conjectures given relatively weak information about the limit shape or set $\functionals$ of attainable limiting linear functionals for Busemann functions. For instance, showing the existence of a $\theta$-directed geodesic requires only the knowledge that $\theta \in \partial \ball$ is an exposed point of differentiability. Moreover, the above results give information in general ergodic models.
We currently cannot determine $\functionals$ for most distributions of interest, but it follows from Theorem~\ref{DH1} that $\functionals$ contains all tangent functionals, and so  $\{\rho / |\rho|: \, \rho \in \functionals\} = S^1$  if the boundary of $\ball$ is differentiable. If $\ball$ is strictly convex, then Theorem~\ref{AH1} implies that every geodesic has an asymptotic direction in the sense of Newman's predicitons.

\subsection{Model assumptions}\label{sec:assns}

We shall, finally, describe the precise assumptions under which our results are valid. We shall work under two separate sets of conditions {\bf A1} and {\bf A2}, corresponding to i.i.d.\ and stationary-ergodic edge-weight distributions, respectively. 
Throughout the paper we shall often consider the probability space  $(\Omega_1,\mathcal{F},\P)$,  where $\Omega_1=[0,\infty)^{\Ec^2}$, $\mathcal{F}$ is the corresponding Borel sigma algebra, and $\P$ satisfies either of the following.
\begin{itemize}
		\item[\bf A1] The variables $(\omega_e)_{e \in \Ec^2}$ are independent and identically distributed under $\P$, and they satisfy the moment condition of \cite{coxdur81}: If $\omega_{e_1}, \, \ldots, \omega_{e_4}$ are the four edges incident to the origin, then
		\[\E\left[\Big(\min_{i = 1, \ldots, 4} \omega_{e_i}\Big)^2\right] < \infty\ . \]
		Moreover, the common distribution of the $\omega_e$'s is continuous.
		
		As shown in \cite{coxdur81}, the shape theorem \eqref{eq:shapethm} holds under {\bf A1}. It trivially implies the almost sure existence of geodesics between vertices; the uniqueness of such geodesics is implied by the continuity of the distribution of each $\omega_e$.
		\item[\bf A2] The joint distribution $\P$ of the variables $(\omega_e)_{e \in \Ec^2}$ obeys the following conditions:
		
		\begin{enumerate}
			\item\label{item:1stshape} It is ergodic with respect to translations of $\Z^2$;
			\item It has all the symmetries of $\Z^2$;
			\item\label{item:lastshape} For some $\delta > 0$, the $2 + \delta$ moment $\E[\omega_e^{2 + \delta}] < \infty$;
			\setcounter{saveenum}{\value{enumi}}
		\end{enumerate}
			Items \eqref{item:1stshape} -- \eqref{item:lastshape} suffice to prove the shape theorem in the form of \eqref{eq:shapethm} for some $\mu$; see \cite{boivin90}. Additional assumptions are required to guarantee that $\mu$ is not identically zero, or equivalently that the asymptotic shape $\ball$ defined at \eqref{eq:ball} above is bounded. We impose the boundedness of $\ball$ as our next assumption, and thereafter describe the remaining conditions on the joint distribution of $(\omega_e)$ which make up {\bf A2}:
		\begin{enumerate}
			\setcounter{enumi}{\value{saveenum}}
			\item \label{item:bddball} The asymptotic shape $\ball$ is bounded;
			\item \label{it:uniquea2} Passage times are a.s.~unique: for any two finite lattice paths $\gamma, \gamma'$ with distinct edge sets,
			\[\P\big(T(\gamma) = T(\gamma')\big) = 0\ ; \]
			\item\label{it:energy} The upward finite energy property holds: given any $\lambda > 0$ such that $\P(\omega_e > \lambda) > 0$, we have $\P(\omega_e > \lambda \mid (\omega_f)_{f \neq e}) > 0$ a.s.;
			\end{enumerate}
			The above assumptions suffice to show a.s.~uniqueness of geodesics between arbitrary pairs of vertices. Recalling the geodesic from $x$ to $y$ is denoted $\geo(x,y)$, we use this notation in the remaining conditions of {\bf A2}:
			\begin{enumerate}[resume]
			\item There exists $L < \infty$ such that
			\[\lim_{n \to \infty} \P\left( \text{for all $x, y \in [-n,n]^2$, $\geo(x,y)$ contains at most $Ln$ edges} \right) = 1\ ; \]
			\item There exist $t, \, \delta,\,\gamma > 0$ such that the following hold:
				\begin{enumerate}
					\item For each edge $e \in \Ec^2$, 
					\[ \P(\omega_e < t - \delta \mid (\omega_f)_{f \neq e}) > \delta \quad \text{a.s., and} \]
					\item $\displaystyle \lim_{n \to \infty} \P\left(\begin{array}{c}\text{for all } x, y \text{ with } \|x\|_\infty = n \text{ and } \|y\|_\infty = 2n, \vspace{0.25em} \\ |\{ e \in \geo(x,y): \omega_e > t\}| > \gamma \|x-y\|_\infty    \end{array}\right) = 1\ . $
					\end{enumerate}
		\end{enumerate}
\end{itemize}

All results and arguments described in this paper are valid under either assumption {\bf A1} or {\bf A2}. Many preliminary results and statements do not require the full range of conditions in assumption {\bf A2} in order to hold. For instance, \eqref{item:1stshape}-\eqref{item:lastshape} of {\bf A2} suffice for the shape theorem (as stated in~\eqref{eq:shapethm}) to hold~\cite{boivin90}, and Theorem~\ref{DH1} is in~\cite{damhan14} proved under conditions~\eqref{item:1stshape}-\eqref{it:energy} of {\bf A2}. However, Theorems~\ref{AH1}-\ref{AH3} are in~\cite{ahlhof} proven under the full set of conditions of {\bf A2}.

\section{Elements of an ergodic theory for infinite geodesics} \label{sec:ergodictheory}

In this section, we introduce some techniques and prove some structural results about geodesics which will be used in our proof of Theorem \ref{thm:highways}.
In Section \ref{sec:weakconv}, we describe the technique of geodesic measures originally introduced in \cite{damhan14}. We will use a different, but similar construction of geodesic measures later in the paper, and so we introduce the main ideas of the measures from \cite{damhan14} here. In Section \ref{sec:rcgprop}, we give the definition and properties of random coalescing geodesics, a tool introduced in \cite{ahlhof}, which allow us to measurably select a family of coalescing geodesics with Busemann function asymptotically linear to a particular functional. Along the way, we introduce a ``good event'' of probability one, denoted $\mathcal{X}$, on which geodesics all simultaneously have natural directedness properties.  In Section \ref{sec:ordering}, we describe the natural counterclockwise ordering on infinite geodesics, along with consequences thereof.   Finally, in Section \ref{sec:measureson}, we prove a useful result about distributions on families of ``non-crossing geodesics'' (geodesics which either coalesce or remain disjoint). This result, Theorem \ref{thm:non-crossing}, will be an important ingredient in the proof of our main Theorem \ref{thm:highways}: we will show in Section \ref{sec:thehighways} that if Theorem \ref{thm:highways} did not hold, one could derive a contradiction to Theorem \ref{thm:non-crossing}.


\subsection{Weak convergence of geodesic measures} \label{sec:weakconv}

In order to remedy the difficulties in Hoffman's subsequential approach for producing infinite geodesics, Damron and Hanson~\cite{damhan14} developed a more systematic method based on weak limits of measures on geodesics and Busemann-like functions. By encoding families of finite geodesics along with associated Busemann differences, they were able to preserve properties of these geodesics while taking limits. We shall outline their method below, as weak convergence of geodesic measures will play an important role in the proof of Theorem~\ref{thm:highways}. We remark that our exposition differs from theirs in some details.

We recall that a linear functional $\rho:\R^2\to\R$ is called a supporting linear functional if $\{x \in \mathbb{R}^2: \, x \cdot \rho = 1\}$ is a supporting line of $\partial \ball$. Let $\rho$ be such a supporting linear functional. We define the FPP model on a particular ``enlarged'' probability space which naturally encodes some additional randomness, allowing us to pick out the Busemann function $B_g$ of a particular geodesic $g \in \tree_0$. 

 We define the new probability spaces
\begin{equation}
\label{eq:allomegas}
\Omega_2:=\{0,1\}^{\dedges^2}, \quad  \text{and}\quad \Omega_3:=\R^{\Z^2\times \Z^2};
\end{equation}
here we have introduced the shorthand $\dedges^2 := \{(x,y): \, \{x, y\} \in \Ec^2 \}$ for the set of directed versions of elements of $\Ec^2$.
An outcome  $\eta \in \Omega_2$ can be thought of as inducing a directed graph whose vertex set is $\Z^2$ and whose edge set is $\{\overline e \in \dedges^2:\, \eta(\overline e) = 1\}$. We often identify $\eta$ with the directed graph it induces. 

We henceforth take the convention
\begin{equation}
\label{eq:samplept3}
(\omega, \eta, \theta) \text{ represents a generic sample point of } \Omega_1 \times \Omega_2 \times \Omega_3\ ,
\end{equation}
with $\eta$ representing a point of $\Omega_2$ and $\theta$ a point of $\Omega_3$.
To avoid cumbersome subscripts, we indicate entries of $\eta$ and 
$\theta$ using arguments rather than subscripts. For $\overline e = (x,y) \in \dedges^2$, we write $\eta(x,y)$ for the $\overline e$ entry of $\eta$, and similarly write $\theta(w,z)$ for arbitrary $w, z \in \Z^2$.

For each supporting functional $\rho$ and real number $\alpha\ge0$, we will define a function $\Psi_{\rho,\alpha}:\Omega_1\to\Omega_1\times\Omega_2\times\Omega_3$ encoding geodesics to distant half-planes as well as an associated Busemann-like function. We fix a realization of edge weights $(\omega_e) \in \Omega_1$. For each $z \in \Z^2$, let $\gamma_\alpha(z)$ denote the geodesic from $z$ to the half-plane $H_\alpha := \{x\in\R^2:\rho(x)\ge\alpha\}$; it is uniquely defined for almost every $\omega \in \Omega_1$.  Considering each $\gamma_\alpha(z)$ as a sequence of directed edges oriented away from its initial point $z$, we define for each $(x,y) \in \dedges^2$
	\[\eta_\alpha (x,y) = \begin{cases} 1, \quad &\text{the directed edge $(x,y)$ is traversed by some $\gamma_\alpha(z)$;}\\
	0, \quad &\text{otherwise};
	\end{cases} \]
	we write $\eta_\alpha = (\eta_\alpha(x,y))_{(x,y) \in \dedges^2} \in \Omega_2$.
	
 Thus $\eta_\alpha$ encodes all geodesics to $H_\alpha$. It induces, as described above, a directed ``geodesic graph'' with vertex set $\Z^2$.  We define a $\theta_\alpha\in\Omega_3$ encoding associated ``Busemann increments'': 
	\[
\theta_\alpha(x,y):=T(x, H_\alpha) - T(y, H_\alpha)\ .
\]
We write $\theta_\alpha = (\theta_\alpha(x,y))_{x,y \in \Z^2}$. Finally, the map $\Psi_{\rho,\alpha}$ is defined using the above as
 \begin{align}
 \Psi_{\rho,\alpha}(\omega) = (\omega, \eta_\alpha, \theta_\alpha) \quad \text{for all $\omega \in \Omega_1$};  \label{eq:psidh}
 \end{align}
 $\Psi_{\rho, \alpha}$ can be seen \cite[Appendix A]{damhan14} to be Borel-measurable. The configuration $(\omega, \eta_\alpha, \theta_\alpha) = \Psi_{\rho, \alpha}(\omega)$  enjoys a few important properties analogous to those of Lemma~\ref{lem:busemanpropps} for a.e.~$\omega$. First, each $x \in \Z^d$ almost surely has out-degree one in $\eta$, and $\eta$ a.s.~has no cycles. Second, $\theta_\alpha$ is additive ($\theta_\alpha(x,z) = \theta_\alpha(x,y) + \theta_\alpha(y,z)$ a.s.) and behaves as $T$ along $\eta$ (if there is a directed path in $\eta_\alpha$ from $x$ to $y$, then $\theta_\alpha(x,y) = T(x,y)$).
 
  We obtain a probability measure $\nu_{\rho,\alpha}$ on $\Omega_1\times\Omega_2\times\Omega_3$ as the push-forward of $\P$ through the map $\Psi_{\rho,\alpha}$.  Using these, we will build a limiting measure on $\Omega_1\times\Omega_2\times\Omega_3$; a point $(\omega, \eta, \theta)$ sampled from this measure will encode geodesics and Busemann functions to a half-plane ``at infinity''. Before taking limits, we take averages as follows:
\begin{equation}\label{DH average}
\nu_{\rho,\alpha}^\ast:=\int_0^\alpha\nu_{\rho,t}\,dt = \E[\nu_{\rho, U_{\alpha}}]\ ,
\end{equation}
where $U_\alpha$ is a uniform random variable on $[0, \alpha]$ independent of $\omega$.
The sequence of measures $(\nu_{\rho,\alpha}^\ast)_{\alpha\ge0}$ is tight due to a statement analogous to property~\ref{it:bleq} of Busemann functions from Lemma~\ref{lem:busemanpropps}. By Prokhorov's theorem, this sequence will have a weakly convergent subsequence. It is straightforward to show that the distribution of $\omega$ under any such subsequential limit $\nu_\rho$ is $\P$ and that $\nu_\rho$ is shift-invariant (due to the averaging in~\eqref{DH average}).
The properties enjoyed by $\nu_{\rho, \alpha}$ yield similar properties for $\nu_\rho$. The following result encapsulating these is a more detailed version of Theorem~\ref{DH1}.

\begin{theorem}\label{DH2}
For every tangent functional $\rho:\R^2\to\R$ the measure $\nu_\rho$ is shift-invariant and satisfies the following properties: For $\nu_\rho$-almost every $(\omega,\eta,\theta)$ we have
\begin{enumerate}[label = {(\alph*)} ]
\item from each site of $\Z^2$, there is a unique infinite directed path in $\eta$, and each such path is a geodesic for $\omega$; \label{it:forgeo}
\item the Busemann function of any infinite directed path is asymptotically linear to $\rho$;
\item any two infinite directed paths coalesce, and $\theta$ encodes their Busemann function.
\end{enumerate}
\end{theorem}

Let us mention that our exposition differs from that of~\cite{damhan14} in that we here encode the entire Busemann differences instead of increments per edge. Our exposition follows that of~\cite{aufdamhan17}, as well as later work~\cite{bridamhan}. The two approaches are analogous to each other.

\subsection{Random coalescing geodesics}\label{sec:rcgprop}

Random coalescing geodesics are a  central tool introduced in \cite{ahlhof} for development of an ``ergodic theory for FPP geodesics.''
The important aspects of their definition relate to their behavior under shifts, and so we introduce notation for shift operators here.
The translation $\sigma_z$ along the vector $z\in\Z^2$ maps $\Omega_1 \times \Omega_2 \times \Omega_3$ to itself via the rule $\sigma_z(\omega, \eta, \theta) = (\sigma_z \omega, \sigma_z \eta, \sigma_z \theta)$, where
\begin{equation}
\label{eq:sigmazdef}
\sigma_z \omega=(\omega_{e-z})_{e\in\edges^2}, \quad \sigma_z \eta = (\eta(\overline e - z))_{\overline e \in \dedges^2}, \quad \sigma_z \theta = (\theta(x-z, y-z))_{x,y \in \Z^2 \times \Z^2}\ .
\end{equation}
Here we have denoted translates of edges in the usual way:  $(x,y) - z = (x-z, y-z)$ (similarly for undirected edges). Of course, each $\sigma_z$ can also be considered as an operator on an individual $\Omega_i$ or a product $\Omega_i \times \Omega_j$ via projection.

Supposing $G:\Omega_1\to\Omega_2$ is a measurable map, we write 
\begin{equation}
\label{eq:rcgsig}
G(z) := \sigma_{-z}\circ G\circ\sigma_z \quad \text{ for each $z \in \Z^2$;}
\end{equation}
 thus $G=G(0)$. The class of such $G$ which are  translation-covariant with respect to the weights $(\omega_e)$ and whose shifts coalesce with one another is exactly the class of random coalescing geodesics:
\begin{definition}\label{defin:rcgd}
A measurable map $G: \Omega_1 \to \Omega_2$ is called a  {\bf random coalescing geodesic} if it has the following almost sure properties: For all $y,z\in\Z^2$
\begin{itemize}
\item almost surely, $G(z)$ is an element of $\tree_z$ --- in other words, $G(z)$ consists of a single infinite directed path beginning at $z$ which is an infinite geodesic, and
\item almost surely, $G(y)$ and $G(z)$ coalesce: $|G(y) \triangle G(z)| < \infty$,
\end{itemize}
where $G(y)$ and $G(z)$ are interpreted as in \eqref{eq:rcgsig}.
  
\end{definition}



The shift-invariance and coalescence assumptions of Definition \ref{defin:rcgd}, in combination with the ergodic theorem, result in well-behaved asymptotic properties. For instance, each random coalescing geodesic $G$ has an asymptotically linear Busemann function. We summarize some of these properties, first described in~\cite[Propositions~4.2--4.5]{ahlhof}, as follows.
\begin{prop}\label{rcg}
Let $G$ be a random coalescing geodesic. The following properties hold:
\begin{enumerate}[label = {(\alph*)}]
\item There exists a (non-random) supporting functional $\rho_G$ such that the Busemann function of $G$ is almost surely linear to $\rho_G$.
\item If $G'$ is a random coalescing geodesic such that $\rho_{G'} = \rho_G$, then $G' = G$ almost surely.
\item \label{it:rcgbf} $G$ is almost surely {\bf backwards finite}, meaning that the set $\{y: \, 0 \in G(y)\}$ is finite.
\end{enumerate}
\end{prop}

We take note of the following immediate consequence of a random coalescing geodesic $G$ being backwards finite: almost surely the path $G(0)$ cannot be extended backwards into a bi-infinite distance-minimising path, i.e.\ a bigeodesic, due to uniqueness of passage times. Related statements are contained in the following lemma, which will be used numerous times in our proof of Theorem \ref{thm:highways}. The lemma rules out ``long backward extensions'' of a similar character, and is derived as a consequence of \ref{it:rcgbf} from Proposition \ref{rcg}.

\begin{lemma}\label{backwards}
	For any random coalescing geodesic $G$, the following occur with probability zero:\footnote{	Here, we consider both $G = G(0)$ and $\geo(0,z)$ as paths beginning at $0$ and interpret convergence as in \eqref{eq:convdef}, and we recall that ``$x \in \tree_y$'' means that the vertex $x$ is in some infinite geodesic beginning at $y$.}
	\begin{enumerate}[label=(\alph*)]
	\item $0\in\tree_v$ for every $v\in G$;
	\item $\exists$ a sequence  $(z_\ell)_{\ell\ge1}$ with $\geo(0, z_\ell) \to G$ and $0 \in \tree_{z_\ell}$ for each $\ell$.
	\end{enumerate}
\end{lemma}

Before we present the proof, we introduce an event $\mathcal{X}$ whose complement is a null set. In many of our arguments, we will need to discard a zero-measure set on which geodesics behave ``unusually''; the common event $\mathcal{X}$ guarantees that many unusual behaviors do not occur. For this reason, it will appear at various points from here on. 

\begin{definition}\label{defin:calx}
	Let $\mathcal{X} \subseteq \Omega_1$ be the event that:
	\begin{enumerate}[label= (\roman*)]
		\item each finite lattice path has a unique passage time: for any finite paths $\gamma, \gamma'$ with at least one edge in their symmetric difference, we have $T(\gamma) \neq T(\gamma')$;\label{it:calx1}
		\item between each pair of points $x, y \in \Z^2$, the geodesic $\geo(x,y)$ exists (by the preceding item, this geodesic is also unique);\label{it:calx2}
		\item the shape theorem holds for passage times from each $x \in \Z^2$, in the sense that $\lim_{|y| \to \infty} [T(x,y) - \mu(y-x)]/|y| = 0$;\label{it:calx3}
		\item every infinite geodesic has linear Busemann function: for each infinite geodesic $g$, there is a linear functional $\rho_g \in \functionals$ such that $\lim_{|y| \to \infty} [B_g(0, y) - \rho_g(y)]/|y| = 0;$\label{it:calx4}
		\item for each infinite geodesic $g$, with $\rho_g$ as in the preceding item, $\dir(g) \subseteq \arc(\rho_g)$.\label{it:calx5}
	\end{enumerate}
	\end{definition}
	
We emphasize that $\P(\mathcal{X}) =1$. Indeed, Item \ref{it:calx1} is almost sure by assumption: the continuity of the distribution from Item {\bf A1}, or the explicit assumption \ref{it:uniquea2} of {\bf A2}. Item \ref{it:calx2} is almost sure under {\bf A1} or {\bf A2}, as noted during the statement of those assumptions. Item \ref{it:calx3} is almost sure by the shape theorem \eqref{eq:shapethm}, Item \ref{it:calx4} by Theorem \ref{AH1}, and Item \ref{it:calx5} by Lemma \ref{lem:whatrho}.

\begin{proof}[Proof of Lemma~\ref{backwards}]
Clearly the occurrence of (a) implies the occurrence of (b); simply take $(z_\ell)_{\ell\ge1}$ to be the enumeration of the vertices on $G$. Moreover, on the event $\mathcal{X}$, the occurrence of (b) implies the occurrence of (a). To see this, suppose that there exists a sequence $(z_\ell)$ as in (b). Since $0\in\tree_{z_\ell}$ we have $0\in\tree_v$ for every $v\in\geo(0,z_\ell)$, due to unique passage times. Since $\geo(0,z_\ell)\to G$, it follows that every $v\in G$ lies on $\geo(0,z_\ell)$ for some $\ell$. Consequently, $0\in\tree_v$ for all $v\in G$.

It will suffice to prove that (a) occurs with probability zero. Suppose the contrary, and let $\omega\in\mathcal{X}$ be a sample point for which (a) occurs. Let $(0,v_1,v_2,\ldots)$ be an enumeration of the vertices in $G$. Let $g_k$ be some geodesic in $\tree_{v_k}$ such that $0\in g_k$; let e.g.\ $g_k$ be the counterclockwise-most geodesic with this property. We may take a subsequential limit of the sequence $(g_k)_{k\ge1}$ to obtain a bi-infinite path $g^*=(\ldots,v_{-1},0,v_1,\ldots)$ that visits the origin, and whose ``positive'' part coincides with $G$. Since $g^*$ is a subsequential limit, we have for every $\ell\ge1$ that $v_{-\ell}\in g_k$ for some $k\ge\ell$. That is, $(v_{-\ell},\ldots,v_\ell)$ is a segment of $g_k$ for some $k$, and hence the geodesic between $v_{-\ell}$ and $v_\ell$. Then $g^*$ is a bigeodesic, and $G$ is backwards infinite with positive probability, contradicting Proposition~\ref{rcg}.
\end{proof}

Proposition \ref{rcg} gives some indication of the utility of Definition \ref{defin:rcgd}. A priori, though, it is not clear that random coalescing geodesics exist. In \cite{ahlhof}, Ahlberg and Hoffman gave the first construction of random coalescing geodesics by showing that one could measurably extract geodesics with the required properties from the geodesic measures of Damron and Hanson.  As a consequence, they obtained that for every tangent functional $\rho$ there exists a random coalescing geodesic with Busemann function asymptotically linear to $\rho$. In fact, most of the work leading up to Theorems~\ref{AH1}-\ref{AH3} amounted to showing that although (perhaps) not all geodesics in $\tree_0$ are (the image of) random coalescing geodesics, the set of random coalescing geodesics is \emph{dense} in $\tree_0$, in the sense of the following theorem, which is essentially \cite[Theorem~10.9]{ahlhof}.

\begin{theorem}\label{rcg dense}
For every $\rho\in\functionals$ there exists a random coalescing geodesic $G$ with Busemann function asymptotically linear to $\rho$.
\end{theorem}

We recall that the set $\functionals$ (introduced in Theorem \ref{AH2}) contains all tangent functionals.

\subsection{``Ordering'' and convergence of infinite geodesics and its relation to Busemann functions\label{sec:ordering}} 

Given a half-plane $H := \{z \in \R^2:\, z \cdot \rho > 0 \}$ for some nonzero $\rho \in \R^2$, we say the path $\pi = (x_i)_{i=0}^\infty$ \emph{eventually moves into} $H$ if
\begin{equation}
\label{eq:moveintodef}\text{for each $z \in \Z^2$, the set } \pi \setminus [z + H] \text{ is finite}.
\end{equation}
Similarly, a family $(\pi_\alpha)_\alpha$ of paths is said to \emph{uniformly move into $H$} if
\begin{equation}
\label{eq:unifmoveintodef} 
\text{for each $z \in \Z^2$, the set } \bigcup_{\alpha} \pi_\alpha \setminus [z + H] \text{ is finite.}
\end{equation}
Lemma~\ref{lem:whatrho} implies that every infinite geodesic $\gamma$ whose Busemann function is asymptotically linear to $\rho$ eventually moves into a half-plane described by $\rho$:
	\begin{equation}
	\label{eq:topasshalf}
	\text{a.s., for each $\rho, g$ as in \eqref{Busemann shape}, $g$ eventually moves into $\{x:\, \rho \cdot x > 0\}$.}
	\end{equation}
	Our event $\mathcal{X}$ from Definition~\ref{defin:calx} implies that each geodesic has asymptotically linear Busemann function and so moves into an appropriate half-plane as in \eqref{eq:topasshalf}.
	
Of course, there could be uncountably many elements of $\functionals$; it will be useful to consider instead a finite family of half-planes:
\begin{equation}
\label{eq:Hidef}
H_i:=\{x\in\R^2:(\cos(i\pi/4),\sin(i\pi/4))\cdot x> 0\},\quad i=0,1,\ldots,7.
\end{equation}
	 Because $\arc(\rho)$ has length at most $\pi/4$ for a supporting $\rho$, \eqref{eq:topasshalf} implies
	 \begin{equation}
	 \label{eq:topasshalf2}
	 	\text{on $\mathcal{X}$, for each infinite geodesic $g$, $\exists i$ such that $g$ eventually moves into $H_i$.}
	 \end{equation}

We will notate the boundaries of these half-planes in the usual way:
\begin{equation}
\label{eq:partialH}
\partial H_i =\{x\in\R^2:(\cos(i\pi/4),\sin(i\pi/4))\cdot x= 0\}\ .    
\end{equation}
As above, we sometimes write $H_i$ when we mean $H_i \cap \Z^2$ (or $\partial H_i$ for $\partial H_i \cap \Z^2$) when there is no risk of confusion.
We note that if $j = i + 4 \mod 8$ then $H_i = -H_j$, and in particular $\R^2$ is the disjoint union of $H_i$, $H_j$, and $\partial H_i = \partial H_j$. For this reason, we index the half-planes cyclically, setting $H_i := H_{i \text{ mod } 8}$ for general integers $i$. These half-planes enjoy some useful properties: most notably, that a path of the graph $\Z^2$ which begins in the set $\Z^2 \setminus H_i$ and enters $H_i$ must exit $\Z^2 \setminus H_i$ at a vertex of $\Z^2 \cap \partial H_i$.

Consider two disjoint infinite geodesics which eventually move into a common $H_i$. Their terminal segments after their last entry into $H_i$ have a fixed counterclockwise ordering. Properties of this ordering will be important in the arguments that follow, so we formalize it carefully here. We define the ordering on the event $\mathcal{X}$ and describe some fundamental results relating the ordering and Busemann functions; this lemma is the main focus of this subsection.  For simplicity of notation, we describe the ordering procedure in detail only in the case of the half-plane $H_0$; the other cases are similar. This will recur often (especially in the proofs of Section \ref{sec:thehighways}); in general, we prefer to work with $H_0$ and $H_4$ when possible for reasons of notational simplicity.

\begin{definition}\label{defin:ordering}
	Consider an outcome $\omega \in \mathcal{X}$; suppose $g, g'$ are two infinite geodesics that eventually move into $H_0$.  The last intersection of $g$ (resp.~$g'$) with $\partial H_0 + n \be_1$ will be denoted by $v_n$ (resp.~$v_n'$) when it is defined --- in other words, when $g \cap [\partial H_0 + n\be_1] \neq \varnothing$. We write  $g \prec g'$ in $H_0$, and say that {\bf $g$ is ccw (counterclockwise) of $g'$} in $H_0$, if we have $v_n \cdot \be_2 \geq v_n' \cdot \be_2$ for all sufficiently large $n$. (Since $g,g'$ eventually move into $H_0$, $v_n, v_n'$ are defined for large $n$.)
	Moreover, a sequence $(g_m)_{m\ge1}$ of infinite geodesics that eventually move into $H_0$ is called {\bf ccw increasing} if $g_{m'} \prec g_m$ for all $1 \leq m < m'$.
\end{definition}

Recall, from~\eqref{eq:topasshalf2}, that on $\mathcal{X}$ every geodesic evetually moves into some $H_i$.
The ccw ordering relation $\prec$ is almost surely a total ordering (modulo coalescence), as the following lemma guarantees:
\begin{lemma}\label{lem:ordergood}
	Fix an outcome $\omega \in \mathcal{X}$. For each pair $g, g'$ of infinite geodesics which eventually move into $H_0$,
	\begin{enumerate}[label = {(\alph*)} ]
		\item \label{it:ordertrue} Either $g \prec g'$, $g' \prec g$, or  $g$ and $g'$ coalesce.
		\item \label{it:twoplanes} The ordering relation $\prec$ is consistent: if $g$ and $g'$ also eventually move into some $H_i$ with $i \neq 0$, then
$g \prec g'$ in $H_0$ if and only if $g \prec g'$ in $H_i$.
		\item \label{it:consistent} If $g \prec g'$, then $\rho_g \cdot \be_2 \geq \rho_{g'} \cdot \be_2$.
	\end{enumerate}
\end{lemma}
Of course, analogous statements hold for geodesics directed in $H_i$ for $i \neq 0$, with appropriate replacements for the vector $\be_2 \in \partial H_0$. For $H_1$, we take $(-1,1)$, so that the ordering defined by $\prec$ agrees with the intuitive notion of paths being clockwise to one another; for other $H_i$'s we take a rotation of either $\be_2$ (if $i$ is even) or $(-1,1)$ (if $i$ is odd). 

Item \eqref{it:consistent} of the lemma guarantees in particular that if $\rho_g \neq \rho_{g'}$, then their ordering in $H_i$ is determined by and compatible with the natural clockwise ordering on supporting functionals.


\begin{proof}
	We begin by proving \eqref{it:ordertrue}. We fix an outcome in $\mathcal{X}$ and geodesics $g, g'$ as in the statement of the lemma. Suppose $g$ and $g'$ do not coalesce, and that $v_m \cdot \be_2 > v_m' \cdot \be_2$ for some $m$. By translating and considering terminal segments of $g$ and $g'$, we may assume that $g$ and $g'$ are disjoint and each intersect $\partial H_0$ at exactly one vertex (respectively $v_0$ and $v_0'$), with $v_0 \cdot \be_2 > v_0' \cdot \be_2$.  We form a doubly infinite simple curve $\mathcal{C}$ from the union of $g$, the edge $\{v_0, v_0 - \be_1\}$, and the set $\{x:\, x \cdot \be_1 = -1, \,  x \cdot \be_2 \geq v_0 \cdot \be_2\}$.  Then by the Jordan curve theorem, $\mathbb{R}^2 \setminus \mathcal{C}$ (considered as a topological subspace of $\mathbb{R}^2$) has two connected components $D_1$ and $D_2$; the set $D_1$ contains $k \be_2$ for all $k$ large (and indeed $n \be_1 + k \be_2$ for each $n \geq 0$ and each $k$ large relative to $n$), and $D_2$ contains both $g'$ and the vertices $- k \be_2$ for all $k$  large.
	
	Suppose (for the sake of contradiction) that for some $N > 0$, we had $v_N \cdot \be_2 < v_N' \cdot \be_2$.  We write $\widetilde g$ and $\widetilde g'$ for the terminal segments of $g$ and $g'$ beginning at $v_N$ and $v_N'$. As in the last paragraph, $\widetilde g'$ splits $[H_0 \cup \partial H_0] + N \be_1 $ into two components. One of these components, $\widetilde D_2$, contains $(N + n) \be_1 + k \be_2$ for each  $n \geq 0$ and all $k$ large relative to $n$; the other, $\widetilde D_1$, contains $\widetilde g$. Since $|g \setminus \widetilde g| < \infty$, we can choose a value of $n \geq 0$ such that $(g \setminus \widetilde g) \cap [\partial H_0 + (N + n) \be_1] = \varnothing$. For this fixed $n$ and for each large value of $k$, we build a path from $(N + n) \be_1 + k \be_2$ to $-k \be_2$ by concatenating
	\begin{itemize}
		\item The vertical line segment from $(N + n) \be_1 + k \be_2$ to its first intersection $w$ with $\widetilde g'$,
		\item The segment of $g'$ from $w$ to $v'_0$,
		\item The vertical line segment from $v_0'$ to $-k \be_2$.
	\end{itemize}
	Since the first vertical segment remains in $\widetilde D_2$ and contains no portion of $g \setminus \widetilde g$, this vertical segment does not intersect $g$. The segment of $g'$ in the second item avoids $g$ since $g$ and $g'$ are disjoint. The final vertical segment also clearly avoids $g$. Thus, for $k$ sufficiently large, the entire path so described connects a vertex of $D_1$ with a vertex of $D_2$ without intersecting $g$, a contradiction; this completes the proof of item \eqref{it:ordertrue}.
	
	We use the same basic setup as above to prove \eqref{it:twoplanes}, assuming that $g$ and $g'$ intersect $\partial H_0$ only at $v_0$ and $v_0'$. For definiteness, we assume that $g$ and $g'$ eventually move into not only $H_0$ but also $H_3$.  For some $t$ larger than $v_0 \cdot \be_2$ and $v_0' \cdot \be_2$, we have $g \cap [\partial H_3  + t \be_2] \neq \varnothing$ and $g' \cap [\partial H_3  + t \be_2] \neq \varnothing$. For the sake of contradiction, we suppose $g \prec g'$ in $H_0$ (i.e., $v_0 \cdot \be_2 > v_0' \cdot \be_2$) but that $g' \prec g$ in $H_3$. Thus, letting $w$ (resp.~$w'$) be the last intersection of $g$ (resp.~$g'$) with  $[\partial H_3  + t \be_2]$, we have $(w'-w) \cdot (-\be_1 + \be_2) \geq 0$.

 Constructing the same doubly infinite curve $\mathcal{C}$ as above, we have again that $v_0' \in D_2$. We will find the required contradiction by showing that $w' \in D_1$, an impossibility since $g'$ cannot cross $\mathcal{C}$. We construct a path $\mathcal{P}$ from $w'$ to $v_0' - \be_1/2$ by concatenating the straight line segments
	\begin{itemize}
	\item from $w'$ to $\zeta := w' - \be_2/2$;
	\item from $\zeta$ along the ray $\{\zeta- s (\be_1 + \be_2):\, s \geq 0\}$ until its first intersection with $\partial H_0 - \be_1/2$;
	\item from the endpoint of the previous segment along $\partial H_0 - \be_1/2$ to the point $v_0' - \be_1/2$.
	\end{itemize}
	See Figure~\ref{fig:18} for a depiction of these paths.
	We extend $\mathcal{P}$ to a doubly infinite curve $\mathcal{C}'$ by concatenating the portion of $g'$ after $w'$ on one end, and the vertical segment $\{v_0' - \be_1/2 - s \be_2: \, s \geq 0\}$ on the other.  Both $v_0$ and $w$ lie on the same side of  $\mathbb{R}^2 \setminus \mathcal{C}'$.
	Now, $g$ can only intersect $\mathcal{P}$ along the segment $\{\zeta - s (\be_1 + \be_2):\, s \geq 0\}$; since such intersections must represent incursions into the other component of $\mathbb{R}^2 \setminus \mathcal{C}'$, there must be an even number of such intersections. 
	
	\begin{figure}
  \begin{minipage}[c]{0.3\textwidth}
\hspace*{2cm}\includegraphics[scale=0.5]{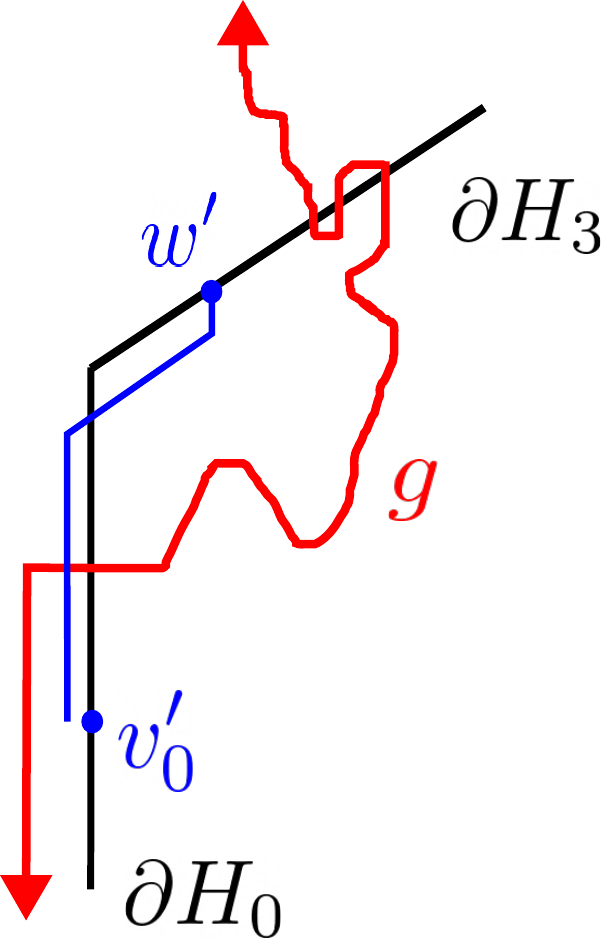}
  \end{minipage}\hfill
  \begin{minipage}[c]{0.6\textwidth}
    \caption{The argument for item \eqref{it:twoplanes} of Lemma~\ref{lem:ordergood}. The path $\mathcal{C}$ is colored red, and the path $\mathcal{P}$ is colored blue. The portion of $\mathcal{C}$ to the right of $\partial H_0$ is identical to $g$ (labeled in the figure).
    } \label{fig:18}
  \end{minipage}
\end{figure}
	
	In particular, $\mathcal{P}$ intersects $\mathcal{C}$ an odd number of times: an even number (possibly zero) of intersections with $g$, and a single intersection with $\mathcal{C}$ at $v_0 - \be_1/2$. Since it ends in  $D_2$, it follows that $w' \in D_1$, furnishing the necessary contradiction.
	

	We move to proving item \eqref{it:consistent}; since by \ref{it:terminalseg} of Lemma \ref{lem:busemanpropps} replacing a geodesic by a terminal segment does not change its Busemann function, we may continue to assume that $g$ and $g'$ intersect $\partial H_0$ at a single vertex and are ordered as in the proof of item \eqref{it:ordertrue}. In fact, for simplicity of notation, we shall assume that $v_0 = 0$.
	We will first simplify the problem by making some harmless assumptions about the linear functionals $\rho_g, \rho_{g'}$.
	
	Suppose $g\prec g'$. We shall assume that $\rho_g \cdot \be_2 > 0$; the proof of the remaining case is analogous. If $\rho_{g'} \cdot \be_2 \leq 0$, then item \eqref{it:consistent} is proved; we therefore may additionally assume $\rho_{g'} \cdot \be_2 > 0$. Regarding $\be_1$-coordinates, we spend the remainder of the paragraph arguing that we may assume $\rho_g \cdot \be_1 > 0$ or $\rho_{g'} \cdot \be_1 > 0$. Clearly $\rho_g \cdot \be_1, \, \rho_{g'} \cdot \be_1 \geq 0$ --- otherwise, $g$ and $g'$ would not eventually move into $H_0$ on $\mathcal{X}$. If $\rho_g \cdot \be_1 = 0$ and $\rho_{g'} \cdot \be_1 = 0$, then since $\rho_g, \rho_{g'}$ are supporting functionals of $\ball$ with positive $\be_2$-coordinate, we have $\rho_g = \rho_{g'}$ corresponds to a supporting line touching $\partial \ball$ at $\be_2 / \mu(\be_2)$. We can therefore assume at least one of $\rho_g \cdot \be_1, \, \rho_{g'} \cdot \be_1 $ is greater than zero; the argument is similar in either case, so we assume $\rho_{g'} \cdot \be_1 > 0$ in what remains.

	Because both $\rho_g$ and $\rho_{g'}$ have positive $\be_2$-coordinate, property \ref{it:calx4} of Definition \ref{defin:calx} guarantees that, for all $k$ larger than some $\omega$-dependent $K \geq 0$, $B_{g'}(v_0', k \be_2) > 0$. In particular, for $k > K$ fixed, we have $T(v_0', v_n') > T(k \be_2, v_n')$ for all large $n$.
	Similarly, because $\rho_{g'} \cdot \be_1 > 0$,  we see that $B_{g'}(v_0', - \ell \be_1) < 0$ for all $\ell$ larger than some $\omega$-dependent $L \geq 0$.
	By \ref{it:buseprelim} of Lemma \ref{lem:busemanpropps}, this implies
	\begin{equation}
	\label{eq:Bmono}
	T(v_0', v_n') < T(- \ell \be_1, v_n') \quad \text{ for all $n \geq 1$ and $\ell > L$}.
	\end{equation}
	By \eqref{eq:Bmono} and the observations immediately preceding, we see that for each $k > K$, we have  $- \ell \be_1 \notin \geo(k \be_2, v_n')$ for each $\ell > L$ and each large $n$.

	
	We use the above to argue the intuitively clear fact that 
	\begin{equation}
	\label{eq:withinL}
	\text{$\geo(k \be_2, v_n')$ (with $k > K$ and $n$ large relative to $k$) comes within distance $L$ of $g$.}
	\end{equation}
		\begin{figure}
 \includegraphics[scale=0.5]{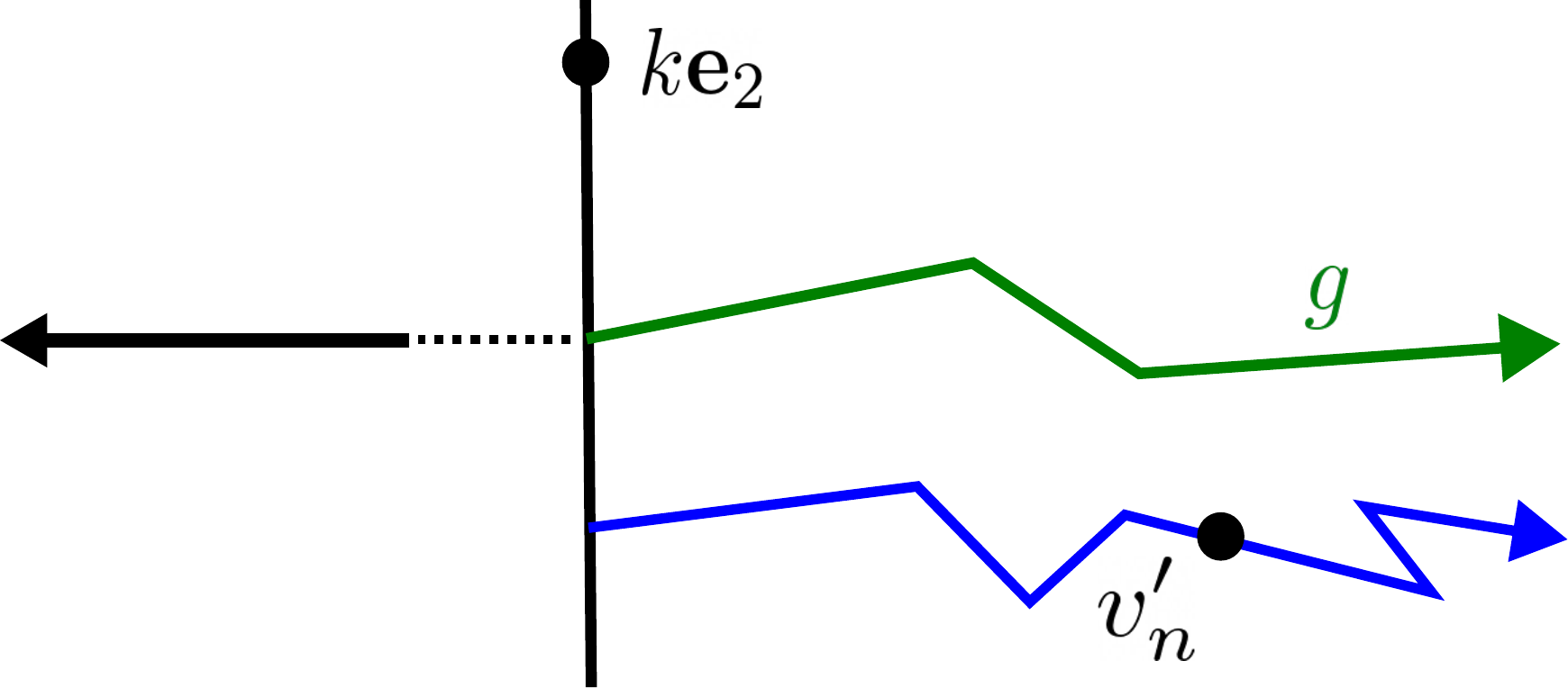}
    \caption{The situation described in \eqref{eq:withinL}. The path $\pi$ consists of the concatenation of $g$ (drawn in green), the dashed segment (from $0$ to $- L \be_1$), and the black ray (from $- L \be_1$ onwards). The blue ray is the geodesic $g'$. The geodesic from $k \be_2$ to $v_n'$ must intersect the union of $g$ and the dashed segment.
    } \label{fig:19}
\end{figure}

    Indeed, building a doubly infinite simple curve $\pi$ by concatenating $\{- r \be_1: \, r \geq 0\}$ and $g$, we see by arguments similar to those earlier in the proof that $\geo(k \be_2, v_n')$ must cross $\pi$. Since it cannot do this at a vertex $- \ell \be_1$ for $\ell > L$, it must either intersect $g$ or touch a vertex of the form $-\ell \be_1$ for $0 \leq \ell \leq L$, proving \eqref{eq:withinL}. See Figure~\ref{fig:19}.
	
	Applying \eqref{eq:withinL} and subadditivity of $T$, we see there exists an $\omega$-dependent constant $A$ such that
	\begin{equation}\label{eq:almostthrough}
	T(k \be_2, v_n') \geq \inf_{w \in g} [T(k\be_2, w) + T(w,  v_n')] - A\ ,
	\end{equation}
	uniformly in $k > K$ and $n$ large as above. In fact, by item \ref{it:calx3} of Definition \ref{defin:calx}, the infimum above is attained at some $w$.
	
	Applying \eqref{eq:almostthrough} for the infimizing value of $w$, we have
	\begin{align*}
	T(k \be_2, v_n') - T(0, v_n') &\geq T(k \be_2, w) + T(w,  v_n') - T(0, v_n') - A\\
	&\geq T(k \be_2, w) + T(w,  v_n') - T(0, w) - T(w, v_n') - A\\
	&= T(k \be_2, w) - T(0, w) - A  \\
	&\geq B_g(k \be_2, 0) - A\ ,
	\end{align*}
	where in the last line we again used \ref{it:buseprelim} from  Lemma \ref{lem:busemanpropps}.
	Taking $n \to \infty$ next with $k$ fixed, this gives
	\[B_{g'}(k\be_2, 0) \geq B_g(k \be_2, 0) - A\ . \]
	Since  the above inequality holds with fixed $A$ for all $k > K$, it immediately yields
	\[\rho_{g} \cdot \be_2 \geq \rho_{g'} \cdot \be_2\ , \]
	as claimed.
\end{proof}

There is a natural alternate description of the above ordering on $\tree_0$ (on the event $\mathcal{X}$) which will be useful in many arguments.
\begin{definition} \label{defin:regionbetweentwo}
	Consider an outcome in $\mathcal{X}$ and suppose that $g_1\neq g_3$ are elements of $\tree_0$ which eventually move into a common half-plane $H_i$ (from \eqref{eq:Hidef}). This implies that there is some $z \in \Z^2$ such that
	\[(g_1 \cup g_3) \subseteq [z + H_i]\ . \]
	The geodesics $g_1, g_3$ start at the same vertex and (by the occurrence of $\mathcal{X}$) intersect in a single segment; hence, they have a last intersection point $x$. Concatenating the terminal segments of $g_1$ and $g_3$ beginning at $x$, we obtain a doubly infinite simple path, which --- when considered as a plane curve --- divides the plane into two components. The union of this doubly infinite path with the component contained entirely in $z + H_i$ is called the {\bf region between $g_1$ and $g_3$} in $H_i$.  (It is immediate that this definition is independent of the choice of $z$ as above.)
	
	With some abuse of notation, we say that an infinite geodesic $g_2$ {\bf lies in the region between $g_1$ and $g_3$} in $H_i$ if all but finitely many vertices of $g_2$ lie in this region. Similarly, we say that a finite geodesic $\geo(0, y)$ lies in the region between $g_1$ and $g_3$ in $H_i$ if $\geo(0,y)$ is contained in the union of this region with the geodesics $g_1$ and $g_3$.
\end{definition}

\begin{lemma}
	\label{lem:regionbetween}
	Consider an outcome of $\mathcal{X}$ in which $g_1, g_2, g_3$ are three distinct elements of $\tree_0$ eventually moving into a common half-plane $H_i$. Then $g_1 \prec g_2 \prec g_3$ if and only if $g_1\prec g_3$ and $g_2$ lies in the region between $g_1$ and $g_3$ in $H_i$.
\end{lemma}
\begin{proof}
	Consider an outcome $\omega \in \mathcal{X}$ and suppose $g_1 \prec g_2 \prec g_3$ eventually move into $H_i$ (for simplicity of notation, we take $H_i = H_0$ as usual). We translate $\omega$ by a multiple of $\be_1$ so that $g_1 \cap H_0, g_2 \cap H_0,$ and $g_3 \cap H_0$ are all distinct; we write $\tilde g_j$ for the segment of $g_j$ starting with its last intersection with $\partial H_0$ for $j = 1, \,2,\, 3$. Write $x \in \partial H_0$ (resp.~$y$, $z$) for the starting points of $\tilde g_1$ (resp.~$\tilde g_2$, $\tilde g_3$). Consider the path $\pi$ starting at $y$ consisting of the union of the half-edge  $\{y,\,y + \be_1/2\}$ and the ray $\{y + \be_1/2 + t \be_2, \, 0 \leq t < \infty\}$.
	
	Points on the curve $\pi$ corresponding to large $t$ all lie outside of the region between $g_1$ and $g_3$.  We will show that $y$ lies in the region between $g_1$ and $g_3$ by showing that $\pi$ crosses $g_1 \cup g_3$ an odd number of times (note that such intersections occur in isolated points). 
	
	Indeed, since $\widetilde g_1$ intersects $\partial H_0$ in exactly one place (at $x$), the curve $\pi$ intersects $\widetilde g_1$ at exactly one point (at $x + \be_1/2$). It is possible that $\pi$ intersects the boundary of the region between $g_1$ and $g_3$ elsewhere, but such intersections must come in pairs, corresponding to excursions of $g_1$ or $g_3$ into $H_0 \setminus \partial H_0$ before their terminal segment. More specifically, forming a doubly infinite path $\pi'$ from $\widetilde g_1$, $\widetilde g_2$, and the segment of $\pi$ connecting $\widetilde g_1$ and $\widetilde g_2$, again $\pi'$ breaks the plane up into two components, one of which is entirely contained in  $H_0$. If $g_3$ (or the portion of $g_1$ before $\tilde g_1$) enters this component by intersecting $\pi$, it must subsequently exit again to reach $x$ or $z$. Because $\mathcal{X}$ occurs, this exit must occur via another intersection of $\pi$; otherwise there would be multiple geodesics between some pair of vertices. 
	
	We thus see that 
	\[|\pi\cap (g_1 \setminus \widetilde g_1)| + |\pi \cap g_3| \quad \text{is even}. \]
   Adding in the single intersection with $\tilde g_1$, we see that $\pi$ has an odd number of  intersections with the boundary of the region between $g_1$ and $g_3$  in $H_0$, and so $y$ lies in this region. This completes one implication of the lemma. 
	
	To show the converse, by relabeling and reflecting if necessary, it suffices to take the setup above, with $g_1 \prec g_2 \prec g_3$, and show that $g_1$ is not in the region between $g_2$ and $g_3$. The argument is similar to the preceding (with different choices of paths), so we describe it only briefly. Recall that $x$ is the last intersection of $g_1$ with $\partial H_0$. We construct a new path $\check \pi$ consisting of the half-edge from $x$ to $x + \be_1 /2$ and then a vertical ray. This path does not intersect $\tilde g_2$ or $\tilde g_3$. Though it may enter the region between $g_2$ and $g_3$, each entrance is paired with a corresponding exit: the segment of $g_2$ or $g_3$ which $\check \pi$ intersects must both enter and exit $H_0$ by crossing $\check \pi$. Hence, considering the two components of the plane made by removal of the portions of $g_2$ and $g_3$ after they separate, we see $\pi'$ begins and ends in the same component, completing the proof.
\end{proof}

The ordering of infinite geodesics can be extended to compare infinite and finite geodesics. On the event $\mathcal{X}$, because there exists a unique finite geodesic between each pair of vertices of $\Z^2$, any finite geodesic from $0$ is ``trapped'' by the geodesics of $\tree_0$. More formally, on $\mathcal{X}$, suppose $g_1, g_2 \in \tree_0$ last intersect at the vertex $x$, and consider the geodesic $\geo(0, y)$. 
If $\geo(0, y)$ ever strictly enters one of the two regions of the plane formed by removing the segments of $g_1 \cup g_2$ from $x$ onward, it cannot then enter the other component without re-intersecting $g_1 \cup g_2$, in contradiction to uniqueness of geodesics. This is the basis of the following lemma.

\begin{lemma}\label{lem:geobetween}
	We fix an outcome $\omega \in \mathcal{X}$ and geodesics $g_1 \prec g_2$ eventually moving into $H_i$, and we denote the region between $g_1$ and $g_2$ in $H_i$ by $U$. For each $y \in \Z^2$,
	if $\geo(0,y)$ contains a vertex $z$ strictly lying in $U$  --- in other words, a vertex $z \in U \setminus(g_1 \cup g_2)$ --- then $\geo(0, y)$ lies in the region between $g_1$ and $g_2$ in $H_i$. In other words, recalling our convention from Definition \ref{defin:regionbetweentwo}, $\geo(0, y) \subseteq U \cup g_1 \cup g_2$. Moreover, all vertices of $\geo(0,y)$ from $z$ onward lie in $U \setminus(g_1 \cup g_2)$.
\end{lemma}
The preceding lemma follows by an argument similar to that used to prove Lemma \ref{lem:regionbetween}; we omit the proof. The next result tells us that ``extremal'' infinite geodesics exist: given a family of infinite geodesics $(g_\alpha)$ which uniformly move into $H_i$, there is a ``counterclockwisemost limiting geodesic'' which eventually moves into $H_i$ (as long as these geodesics are uniformly confined in $H_i$).


\begin{lemma}\label{lem:extremal}
	Fix an outcome $\omega \in \mathcal{X}$. Suppose $(\gamma_\alpha)$ is a collection of elements of $\tree_0$ which uniformly moves into $H_i$ in the sense of \eqref{eq:unifmoveintodef}.
	Then there exists an extremal ccw limit $\gamma$ of $(\gamma_\alpha)$ in the following sense:
	\begin{enumerate}[label = {(\alph*)} ]
		\item \label{it:1stgam} There is a sequence $(\alpha_k)_{k=1}^\infty$ of values of $\alpha$ such that $\gamma_{\alpha_k}$ converges to $\gamma$ in the sense of \eqref{eq:convdef} (hence $\gamma \in \tree_0$), and
		\item \label{it:2ndgam} For each $\alpha,$ we have $ \gamma \prec \gamma_\alpha$ in $H_i$.
	\end{enumerate}
	Suppose that for each $\omega \in \mathcal{X}$ we have a family $(\gamma_\alpha)$ as above. If for each edge $e \in \edges^2$ the set of outcomes $\{e \in \gamma_\alpha \text{ for some $\alpha$}\}$ is measurable, then $\gamma$ is a random variable $\gamma: \Omega_1 \to \Omega_2$.
\end{lemma}
\begin{proof}
	Fix such an outcome $\omega \in \mathcal{X}$ and geodesics $(\gamma_\alpha)$ as in the statement of the lemma. As usual, for simplicity of notation, we take $H_0 = H_i$; the argument is essentially the same for the other half-planes of \eqref{eq:Hidef}.
	For each $k \geq 1$, each $\gamma_\alpha$ has a last intersection with $H_0 + k \be_1$; write $v_k^{(\alpha)}$ for the vertex at which this last intersection occurs.
	Because $(\gamma_\alpha)$ uniformly moves into $H_0$, we have
	\begin{equation}
	\label{eq:tocontravk}
	\text{ for each $k \geq 1$,}\quad \sup_{\alpha} v_k^{(\alpha)} \cdot \be_2 < \infty\ .
	\end{equation}
	
	
	We can thus make the following definition:
	\begin{equation}
	\text{for each $k \geq 1$,}\quad\zeta_k \text{ is the vertex of the form  } v_k^{(\alpha)} \text{ which maximizes } v_k^{(\alpha)} \cdot \be_2.
	\label{eq:extremalzeta}
	\end{equation}
		We note a ``consistency'' property of the $\zeta_k$'s; namely,
	\begin{equation}
	\label{eq:zetaconst}
	\text{for each $0 < k < m$, } \quad \zeta_k \in \geo(0, \zeta_m).
	\end{equation}
	Indeed, suppose $\gamma_{\alpha_m}$ (resp.~$\gamma_{\alpha_k}$) is a geodesic in $(\gamma_\alpha)$ such that $\zeta_m = v_m^{(\alpha_m)}$  is its last intersection with $\partial H_0 + m \be_1$ (resp.~$\zeta_k = v_k^{(\alpha_k)}$ its last intersection with $\partial H_0 + k \be_1$), and suppose that the last intersection of $\gamma_m$ with $\partial H_0 + k \be_1$ were not $\zeta_k$. Then since $\zeta_k$ is extremal in the sense of \eqref{eq:extremalzeta}, the last intersection of $\gamma_m$ with $\partial H_0 + k\be_1$ must have strictly smaller $\be_2$-coordinate than $\zeta_k$. But then by Lemma \ref{lem:ordergood}, we would have $\gamma_k \prec \gamma_m$, and in particular the last intersection of $\gamma_k$ with $\partial H_0 + m \be_1$ would have strictly larger $\be_2$-coordinate than $\zeta_m$. This contradicts  the definition of $\zeta_m$ and shows \eqref{eq:zetaconst}.
	
	We now consider the sequence $(\geo(0, \zeta_k))$ and set $\gamma$ to be any subsequential limit (we will indeed see that there is a unique limit) of this sequence as $k \to \infty$. By \eqref{eq:zetaconst}, we see that $\zeta_k \in \gamma$ for each $k$, which shows \eqref{it:1stgam} from the statement of the lemma: indeed, $\gamma$ and $\gamma_{\alpha_k}$ each contain $\geo(0, \zeta_k)$ as their initial segment. Moreover, since $\zeta_k$ is the last intersection of  $\gamma_{\alpha_m}$ with $\partial H_0 + k \be_1$ for each $m \geq k$, and since $\gamma$ and $\gamma_{\alpha_m}$ both contain $\geo(0, \zeta_m)$ as an initial segment, we see that $\zeta_k$ is also the last intersection of $\gamma$ with $\partial H_0 + k \be_1$.  In particular, applying Lemma \ref{lem:ordergood}, we see $\gamma \prec \gamma_\alpha$  for each $\gamma_\alpha$.
	
	The measurability result follows  from the way we exhibit $\gamma$ as a subsequential limit. From our assumption that $\{e \in \cup_\alpha \gamma_\alpha\}$ is measurable, it follows immediately that each event of the form $\{\zeta_k = z\}$, for $z \in \Z^2$, is measurable. Since $\geo(x,y): \Omega_1 \to \Omega_2$ is a measurable mapping for deterministic $x$ and $y$, it follows that $\geo(0, \zeta_m)$ is measurable for each $m$. Since $\geo(0, \zeta_m) \to \gamma$, we see $\gamma$ is also measurable.
\end{proof}

\subsection{Measures on non-crossing geodesics}\label{sec:measureson}
We describe first the goal of this section in somewhat loose terms, then spend the bulk of the section making this precise. Suppose we could sample a family of geodesics $(\gamma_x)_{x \in \Z^2}$ (where each $\gamma_x \in \tree_x$) in a translation-invariant way. By Lemma \ref{lem:ordergood}, the $\gamma_x$'s whose Busemann functions are asymptotically linear to functionals with larger $\be_2$-coordinate must be ``more $\be_2$-directed'' than those whose functionals have smaller $\be_2$-coordinate. In particular, these geodesics must somehow ``sort themselves'' so that their ordering is consistent with their Busemann functions. It seems plausible that if there are enough distinct linear functionals exhibited by the $\gamma_x$'s, then this ``sorting'' will require some of the $\gamma_x$'s to cross.

On the other extreme, if all the $\gamma_x$'s have Busemann functions asymptotically linear to the same functional $\rho$, there are two possibilities: either all the $\gamma_x$'s coalesce or some do not. Theorem \ref{AH3} suggests that coalescence is typical: for some of these geodesics not to coalesce, the functional $\rho$ must be chosen in an exceptional set for the weights $\omega$ for many vertices, which seems difficult to achieve. To summarize this and the previous paragraph: we should expect in most cases that our $\gamma_x$'s should tend to either cross or to coalesce.

The above reasoning has some obvious gaps, but is not so far from the truth, as we shall see. To make it precise, we need some new definitions.
\begin{definition} \label{defin:shiftinvar}
  A translation-invariant probability measure $\nu$ on $\Omega_1\times\Omega_2$ is called a {\bf shift-invariant measure on non-crossing geodesics} if its marginal on $\Omega_1$ is $\P$ and its marginal on $\Omega_2$ does not put all mass on the empty configuration (i.e., $\P(\eta(x,y) = 0 \text{ for all } (x,y) \in \dedges^2) < 1$), and if for $\nu$-almost every $(\omega,\eta)\in\Omega_1\times\Omega_2$ the (directed) graph encoded by $\eta$ has the following properties: 
\begin{itemize}
\item every site has either out-degree 1 or both in- and out-degree zero;
  \item there are no (directed) cycles in this graph (combined with the previous item, this guarantees all nontrivial directed paths are infinite and self-avoiding);
\item all infinite directed paths in $\eta$ are geodesics in the weight configuration $\omega$;
\item any two  infinite directed paths are either disjoint or coalesce.
\end{itemize}
Suppose $\eta \in \Omega_2$ is sampled from a shift-invariant measure on non-crossing geodesics as above. If $g$ and $g'$ are two infinite paths in the graph encoded by $\eta$, we write $g \sim g'$ if $g$ and $g'$ coalesce; this defines an equivalence relation, whose equivalence classes are called {\bf coalescence classes}.
\end{definition}
We note that the final item in the definition of shift-invariant measures on non-crossing geodesics follows from the statement about out-degrees, and so is  redundant to specify --- we do so only to emphasize this statement.

The following theorem makes precise the reasoning with which we began this section, and will be central in the proof of Theorem~\ref{thm:highways}. The theorem was proved in~\cite{ahlhof}. We provide here an alternative proof, which shows how the theorem follows from Theorems~\ref{AH1}-\ref{AH3}.

\begin{theorem}\label{thm:non-crossing}
Every shift-invariant measure on non-crossing geodesics is supported on families of geodesics with at most four coalescence classes.
\end{theorem}

\begin{proof}
We start by simplifying the measure $\nu$. Since the marginal of $\nu$ on $\Omega_1$ is $\P$, we have that $\nu(\omega \in \mathcal{X}) = 1$, where as usual $\mathcal{X}$ refers to the event in Definition \ref{defin:calx}. Since $\nu$-a.s.~all paths $\gamma$ in the graph encoded by $\eta \in \Omega_2$ are infinite geodesics, items \ref{it:calx4} and \ref{it:calx5} in the definition of $\mathcal{X}$ ensure that directed paths in graphs sampled from $\nu$ have asymptotically linear Busemann functions:
\begin{equation}
  \label{eq:noncrosshi}
   \begin{minipage}{0.9\textwidth}
\begin{center}
  for $\nu\text{-a.e. } (\omega, \eta)$, each infinite path $\gamma$ of $\eta$ is a geodesic and has Busemann function asymptotic to some supporting functional $\rho_\gamma$.
  \end{center}
  \end{minipage}
\end{equation}

For each $i$, we define a mapping $F_i: \Omega_2 \to \Omega_2$ which deletes all directed edges except those which are part of a (directed) path $\gamma$  in $\eta$ whose linear functional $\rho_\gamma$ is supporting only at points $z \in \partial \ball$ with $z \cdot (\cos(i \pi/4), \sin(i \pi/4)) > 0$. In other words, for each directed edge~$(x,y) \in \dedges^2$,
	\[F_i(\eta)(x,y) = \begin{cases}1\ , \quad &\text{the unique infinite path $\gamma$ in $\eta$ beginning with $(x,y)$ }\\
	&\text{has } \rho_\gamma \text{ such that } \arc(\rho_\gamma) \subseteq H_i;\\
	0 \quad &\text{otherwise.} \end{cases}\]
    Of course, if $\eta(x,y) = 0$, then $(x,y)$ is in no path in the graph encoded by $\eta$, so $F_i(\eta)(x,y) = 0$.  The function $F_i$ is manifestly Borel measurable; we can thus define $\nu_i$ to be the measure on $\Omega_1 \times \Omega_2$ which is the pushforward of $\nu$ by the mapping $(\omega, \eta) \mapsto (\omega, F_i(\eta))$. We  note that the mapping $\textup{id}\times F_i$ is translation-covariant in the following sense: for each $z \in \Z^2$,
    \[\big(\omega,F_i(\sigma_z \eta)\big) = \sigma_z\big(\omega, F_i(\eta)\big)\ , \]
    where we recall from \eqref{eq:sigmazdef} the definition of the translation operator $\sigma_z$. From this it follows that $\nu_i$ is invariant under translations: if $A \subseteq \Omega_1 \times \Omega_2$ is an arbitrary event,
    \begin{align*}
      \nu_i(\sigma_z^{-1} A) = \nu\big((\textup{id}\times F_i)^{-1}(\sigma_z^{-1} A)\big) = \nu\big( \sigma_z^{-1}(\textup{id}\times F_i)^{-1}(A)\big) = \nu\big((\textup{id}\times F_i)^{-1}(A)\big) = \nu_i(A)\ .
    \end{align*}
    In fact,
    \begin{equation*}
    \begin{gathered}
      \text{if}\quad\nu(\text{there is a path $\gamma \in \eta$ with $ \arc(\rho_\gamma) \subseteq H_i$}) > 0 \ ,\\
      \text{then } \nu_i \text{ is a shift-invariant measure on non-crossing geodesics}.
      \end{gathered}
    \end{equation*}
    Indeed, we have shown translation invariance, and $\nu_i$ inherits the other properties of Definition \ref{defin:shiftinvar} immediately from $\nu$ --- as long as it does not put all its mass on the empty configuration (and this could only happen if no paths $\gamma$ of $\eta$ have $\rho_\gamma$ as in the last display).
    
    For the remainder of the proof we fix an $i$ such that $\nu_i$ is a shift-invariant measure on non-crossing geodesics, as above. We will show that
    \begin{equation}
      \label{eq:iclasses}
      \nu_i(\text{$\eta$ has more than one coalescence class}) = 0\ .
    \end{equation}
    Given \eqref{eq:iclasses}, the theorem follows for reasons we explain next. First note that if $\ball$ is not a diamond (i.e.\ the ball in the $\ell^1$ norm), then each $\rho \in \functionals$ must satisfy $\arc(\rho) \subseteq H_i$ for some $i \in \{0, 2, 4, 6\}$. On the other hand, if $\ball$ is a diamond, then each $\rho \in \functionals$ must satisfy $\arc(\rho) \subseteq H_i$ for some $i \in \{1, 3, 5, 7\}$. Hence, identifying $\eta$ with the graph it encodes for abbreviation and applying \eqref{eq:noncrosshi}, for some set $I$ of size four we have
    \begin{align*}
    \nu(\eta \text{ has $> 4$ coalescence classes})
     \leq \sum_{i\in I} \nu_{i}(\text{there exist paths $g_1 \not \sim g_2$ of $\eta$})\ ,
    \end{align*}
    which by \eqref{eq:iclasses} is zero.
    It hence suffices to show that \eqref{eq:iclasses} holds.

    We show that for a.e.~$(\omega,\eta)$, the supporting functional $\rho_\gamma$ cannot depend on the choice of $\gamma$ in $\eta$:
    \begin{equation}\label{eq:onebuse}
      \nu_i(\text{there are infinite paths $g_1, g_2$ in $\eta$ with } \rho_{g_1} \neq \rho_{g_2}) = 0\ .
    \end{equation}
    Without loss of generality, we fix the case $i = 0$ to simplify notation. We suppose that \eqref{eq:onebuse} were false for $i = 0$; by convexity of $\ball$, this implies that we can find
    $g_1, g_2$ in $\eta$ such that $\rho_{g_1} \cdot \be_2 \neq \rho_{g_2} \cdot \be_2$. On $\mathcal{X}$, both $g_1$ and $g_2$ must eventually move into $H_0$.
    In particular, if \eqref{eq:onebuse} were false,
        \begin{equation}\label{eq:threebuse}
      \nu_0(\text{there are $g_1, g_2$ in $\eta$ eventually moving into $H_0$ with } \rho_{g_1} \cdot \be_2 > \rho_{g_2} \cdot \be_2) > 0\ .
    \end{equation}

    By the translation-invariance of $\nu_0$, we can furthermore guarantee that the paths $g_1, g_2$ in \eqref{eq:threebuse} intersect $\partial H_0$; since these paths (a.s.) eventually move into $H_0$, we can (by taking terminal segments) ensure they touch $\partial H_0$ only once, at their initial vertex. Formally,
    \begin{equation}
      \label{eq:fourbuse}
      \nu_0(\text{there are $g_1, g_2$ as in \eqref{eq:threebuse} intersecting $\partial H_0$ only at their initial vertices}) > 0\ .
      \end{equation}

      We wish to slightly refine the event in \eqref{eq:fourbuse} to guarantee that the path $g_1$ begins below $g_2$. Let us write $A_{k, \ell}$ for the event in $\eqref{eq:fourbuse}$ with initial vertices $k \be_2$, $\ell \be_2$:
      \[A_{k, \ell} = \{\text{there are paths $g_1, g_2$ as in \eqref{eq:threebuse} intersecting $\partial H_0$ only at (respectively) $k \be_2$, $\ell \be_2$} \}. \]
By the ergodic theorem, for a.e.~outcome $(\omega, \eta)$ in the event from \eqref{eq:fourbuse}, there exist infinitely many positive integers $k, \ell$ such that $(\omega, \eta) \in A_{k, \ell}$.
Because there are only countably many choices of such $k, \ell$, we can choose some $k < \ell$ such that $A_{k, \ell}$ occurs with positive probability; recalling that this conclusion followed from \eqref{eq:threebuse}, we see that
\[\text{If \eqref{eq:threebuse} is true, then there exist $k < \ell$ with } \nu_0(A_{k, \ell}) > 0\ . \]

We now argue that in fact $\nu_0(A_{k, \ell}) = 0$ whenever $k < \ell$; by the last display, this will show \eqref{eq:threebuse} is false. Indeed, on the event $A_{k, \ell} \cap \mathcal{X}$ (which has positive probability if $A_{k, \ell}$ does), by Item \ref{it:consistent} of Lemma \ref{lem:ordergood} we have $g_1 \succ g_2$ in the ordering on $H_0$ (since the starting point of $g_1$ lies below that of $g_2$, and these paths cannot coalesce, having distinct Busemann functiions). On the other hand, by Item \eqref{it:consistent} of Lemma \ref{lem:ordergood}, we have $g_1 \prec g_2$, since $\rho_{g_1} \cdot \be_2 > \rho_{g_2} \cdot \be_2$. This shows that $A_{k, \ell} \cap \mathcal{X}$ has contradictory properties and hence is empty, showing that \eqref{eq:threebuse} must be false. This shows that \eqref{eq:onebuse} holds.

    
On the event
\begin{align*}
R :=  \mathcal{X} &\cap \{\text{the graph encoded by $\eta$ is nonempty and all paths in it are geodesics} \} \\
  &\cap \{\text{there do not exist infinite paths $g_1, g_2$ in $\eta$ with $\rho_{g_1} \neq \rho_{g_2}$ }\}
\end{align*}
(which has positive $\nu_i$-probability), we define $\varrho = \varrho(\omega, \eta)$ be the common $\rho_g$ of each infinite path $g$ in the graph encoded by $\eta$ (off this event, we set $\varrho = 0$).
The event $R$ is invariant under lattice translations, and on $R$, the value of $\varrho$ is clearly translation-invariant (since its value does not depend on the particular infinite path chosen in $\eta$). Hence,
\begin{equation}
  \label{eq:rhoinvar}
  \varrho = \varrho \circ \sigma_z \quad \text{ for each lattice translation $\sigma_z$}\ .
\end{equation}

If we knew that $\varrho$ were almost surely equal to some nonrandom constant $\widehat \varrho$ on the event $R$, we could immediately conclude \eqref{eq:iclasses}, using Theorem \ref{AH3}: if the geodesics in the graph encoded by $\eta$ did not coalesce, we would find two geodesics whose Busemann functions are asymptotic to the same linear functional $\widehat \varrho$, a probability zero event.
However, our assumptions are not enough to guarantee that $\varrho$ is $\nu_i$-a.s.~constant. To overcome this,  we show that $\varrho$ is independent of the edge configuration $\omega$. Formally:
\begin{equation}
  \label{eq:rhoindep}
  \text{For arbitrary Borel $U \subseteq \R^2$ and  $V \subseteq \Omega_1$,}\quad \nu_i(\{\varrho \in U\} \cap \{\omega \in V\}) = \nu_i(\varrho \in U) \P(V)\ .
\end{equation}
We show \eqref{eq:rhoindep}; writing $\E_{\nu_i}$ for expectation with respect to $\nu_i$, we have
\begin{align}\label{eq:ergodrho}
  \nu_i(\{\varrho \in U\} \cap \{\omega \in V\}) = \frac{1}{(2m + 1)^2} \E_{\nu_i} \left[\sum_{z \in [-m,m]^2} \mathbf{1}_{\{\sigma_z \omega \in V\}} \mathbf{1}_{\varrho \circ \sigma_z \in U} \right]
\end{align}
for arbitrary fixed $U,V$ as in \eqref{eq:rhoindep}. Applying \eqref{eq:rhoinvar}, we write the expectation from \eqref{eq:ergodrho} as
\begin{equation}
  \label{eq:ergodrho2}
   \E_{\nu_i} \left[\mathbf{1}_{\varrho\in U} \left( \frac{1}{(2m + 1)^2}\sum_{z \in [-m,m]^2} \mathbf{1}_{\{\sigma_z \omega \in V\}} \right) \right]\ .
 \end{equation}
 The quantity in parentheses is an ergodic average of a measurable function on $\Omega_1$; since the marginal of $\nu_i$ on $\Omega_1$ is the ergodic measure $\P$, we have
 \[ \lim_{m \to \infty} \frac{1}{(2m + 1)^2}\sum_{z \in [-m,m]^2} \mathbf{1}_{\{\sigma_z \omega \in V\}} = \P(\omega \in V)\ , \quad \nu_i\text{-a.s.} \]

 Using the last display and the dominated convergence theorem, we see
 \[\eqref{eq:ergodrho2} \longrightarrow \E_{\nu_i} \left[\mathbf{1}_{\varrho\in U} \P(\omega \in V) \right]\quad \text{as $m \to \infty$} ; \]
 inserting this back into \eqref{eq:ergodrho}, we have established \eqref{eq:rhoindep}. It remains to use \eqref{eq:rhoindep} to show that $\nu_i$-a.s.~all paths in the graph encoded by $\eta$ coalesce, thus establishing \eqref{eq:iclasses}. We define the function $h: \R^2 \times \Omega_1 \to \R$ without any explicit reference to $\eta$:
 \[h(\rho, \omega) = \mathbf{1}_{\{\exists \text{non-coalescing infinite geodesics $g_1, g_2$ with $\rho_{g_1} = \rho = \rho_{g_2}$} \} }(\omega)\ ;\]
 if $(\omega, \eta) \in R$ exhibits multiple coalescence classes, then $h(\varrho(\omega, \eta), \omega) = 1$. Thus, \eqref{eq:iclasses} will follow once we establish that $\E_{\nu_i}[h(\varrho(\omega, \eta), \omega)] = 0$.

 Writing $\E_\varrho$ for expectation with respect to the distribution of $\varrho$, by the independence shown in \eqref{eq:rhoindep}, we can write
 \begin{align*}
   \E_{\nu_i}[h(\varrho(\omega, \eta), \omega)] = \E_{\varrho}[ \E[ h(\varrho, \omega)]] = 0\ ,
 \end{align*}
 where the inner expectation (over $\omega$ for fixed $\varrho$) is zero by Theorem \ref{AH3}. This shows \eqref{eq:iclasses} and completes the proof of the theorem.
\end{proof}

\section{The highways and byways problem}
\label{sec:thehighways}

 In this section we prove Theorem~\ref{thm:highways}.
 We shall aim for a contradiction, and thus assume that for some diverging sequence $(v_k)_{k\ge1}$ we have uniformly in $k$ that
 \begin{equation}\label{assumption}
 0 < \delta := \lim_{k \to \infty} \P\big(v_k\in\tree_0\big) =  \lim_{k \to \infty} \P\big(0\in\tree_{v_k}\big)\ ,
 \end{equation}
 where the existence of the limit follows by taking a subsequence (if necessary) and the final equality follows by the invariance of $\P$ under translation and reflection under {\bf A1} or {\bf A2}. We will generally prefer to take the perspective of the last quantity from \eqref{assumption}, working with the events $\{0 \in \tree_{v_k}\}$.

The proof will roughly amount to showing that~\eqref{assumption} will imply the existence of a sequence of families of finite disjoint geodesics. The important step will be to show that these geodesics grow longer as we move forward through our sequence, while the density remains stable. This will allow us to take limits to obtain a family of infinite non-crossing and non-coalescing geodesics. This may be formalized by encoding the finite geodesics in a larger probability space, producing a shift-invariant measure on non-crossing  geodesics. Graphs sampled from this measure will turn out to have infinitely many coalescence classes with positive probability, contradicting Theorem~\ref{thm:non-crossing} and thus disproving \eqref{assumption}.

The argument is similar in spirit to the argument used in~\cite{ahlhof} to solve the `midpoint problem' from~\cite{benkalsch03}, but differs in the details. The two main differences between the two proofs are both related to the fact that the basic assumption in~\eqref{assumption} is expressed in terms of infinite geodesics instead on finite paths as in~\cite{ahlhof}. On one hand this facilitates the analysis in some aspects, since there will be fewer endpoints of finite segments to control. On the other hand it will force us to consider two different constructions of measures on non-crossing geodesics (in Sections~\ref{sec:finseg} and~\ref{sec:finseg2}): one based on an infinite collection of finite paths (as in~\cite{ahlhof}) and the other based on a finite number of infinite paths. The latter gives some evidence of the flexibility of the method.

 Before beginning the main part of our proof, we state some elementary consequences of our assumption \eqref{assumption}: first, by compactness of $S^1$ and taking a further subsequence of $(v_k)$, we may assume that $(v_k / |v_k|)$ converges. By the rotation and reflection symmetries of $\P$, we can further assume that the limit lies in a sector of central angle $\pi/4$:
 \begin{equation}
   \label{eq:vktov}
  \lim_{k \to \infty}v_k/|v_k| = v \quad \text{for some $v \in [0, \pi/4] \subseteq S^1$}. 
 \end{equation}
 We henceforth assume our sequence $(v_k)$ satisfies \eqref{eq:vktov}.

 It will sometimes be useful to consider the set of all $v_k$ for which $0 \in \tree_{v_k}$. By Fatou's lemma, \eqref{assumption} implies that there are infinitely many such $v_k$ with positive probability:
 \begin{equation}
   \begin{gathered}
\label{assumption3}
\text{Defining } \mathcal{D} := \{\exists \text{ a subsequence $(v_{k_\ell})$ of } (v_k) \text{ with } 0 \in \tree_{v_{k_\ell}} \text{ for each $\ell$} \}\ ,\\
\text{the assumption \eqref{assumption} implies $\P(\mathcal{D}) \geq \delta$.}
\end{gathered}
\end{equation}
On the event $\mathcal{D}$ from \eqref{assumption3},
 \begin{equation}
  \label{eq:Ikdef}
  (I(k))_{k = 1, 2 \ldots} \text{ is the subsequence of $(1, 2, \ldots)$ consisting of those $\ell$ such that $0 \in \tree_{v_\ell}$.}
 \end{equation}
We note that after some initial work, we will again replace the sequence $(v_k)$ by a subsequence (at \eqref{eq:goodprob} below), along which a certain ``good event'' has nonvanishing probability, and take a further subsequence at \eqref{eq:vkthin}. We emphasize here that this restriction to a subsequence will not invalidate any of the results preceding \eqref{eq:goodprob}.

The organization of the remainder of this section is as follows. In Section~\ref{sec:nondiff}, we rule out the possibility that $v$ corresponds to a direction of differentiability of $\partial \ball$ (Lemma~\ref{corners}) and argue that for large $k$, $\geo(0, v_{I(k)})$ is ``well-localized'' near its endpoints (Lemmas~\ref{lem:hpfromvk} and \ref{lem:hpfromvk2}). In Section~\ref{sec:finseg}, we show that under \eqref{assumption}, if additionally geodesics of the form $\geo(x,x+v_k)$ tend not to intersect (see \eqref{eq:raul0} for a formal statement), one can construct a measure on non-crossing geodesics which exhibits infinitely many coalescence classes (the latter fact, the culmination of the construction, is Lemma~\ref{cornelis}). In Section \ref{sec:finseg2}, we consider the case that \eqref{eq:raul0} does not hold and provide an alternative construction of a measure on non-crossing geodesics which exhibits infinitely many coalescence classes. Finally, in Section~\ref{sec:highwaysatlast}, we pull together the arguments of the preceding sections and prove Theorem~\ref{thm:highways}.
 

\subsection{Showing non-differentiability, and first consequences thereof}\label{sec:nondiff}

We first assume, in addition to \eqref{assumption}, that ``$\partial\ball$ is differentiable at $v$'' --- more formally, that
\begin{equation}
  \label{eq:diffblev}
\text{$v / \mu(v)$ is a point of $\partial \ball$ at which there is a (unique) tangent functional to $\partial \ball$,}
\end{equation}
where $v$ was introduced in \eqref{eq:vktov}. In this section, we will show that assuming both \eqref{assumption} and \eqref{eq:diffblev} simultaneously leads to a contradiction, proving a special case of Theorem \ref{thm:highways}. We deal with the more difficult case of non-differentiability at $v$ in Sections \ref{sec:finseg} -- \ref{sec:finseg2}.

We begin by introducing two families $(G_n^+), \, (G_n^-)$ of particularly useful random coalescing geodesics, which will be invoked repeatedly in the arguments of this subsection. The first important feature of these families is that large multiples of $v$ lie in the region between $G_n^+$ and $G_n^-$, and so $\geo(0, v_k)$ lies in this region (in the sense of Definition \ref{defin:regionbetweentwo}) for large $k$. The second is that these families get as close to the side of $\partial \ball$ containing $v$ as possible: there is no geodesic $g$ lying between $G_n^+$ and $G_n^-$ for all $n$ that does not have $v \in \arc(\rho_g)$. Special care is needed at first when discussing ``the region between'' $G_n^+$ and $G_n^-$, since it is not clear that these must eventually move into a common $H_i$. These issues are resolved once we have proved Lemmas \ref{lem:hpfromvk} and \ref{lem:hpfromvk2}; we can then work with geodesics directed in single choice of $H_i$ and rely on the ordering of Definition \ref{defin:ordering} (and subsequent characterizations thereof).

We first need a simple result showing that there exist enough geodesics to allow our construction of $(G_n^+)$ and $(G_n^-)$.
\begin{prop}\label{prop:diamondprop}
  Under \eqref{assumption}, we have $|\functionals| = \infty$. In particular, there exist infinitely many $\rho \in \functionals$ such that both $\rho \cdot \be_1 \geq 0$ and $\rho \cdot \be_2 \geq 0$.

\end{prop}
\begin{proof}

  The second claim of the proposition follows from the first by the symmetry properties of $\P$ under {\bf A1} and {\bf A2}. We prove the first claim, assuming $|\functionals| < \infty$ in order to  derive a contradiction.
  
We recall that under \eqref{assumption}, the observation \eqref{assumption3} holds: with positive probability, $0 \in \tree_{v_{I(k)}}$ for some  sequence $(I(k))$ with $I(k) \to \infty$. On this event,  taking a subsequential limit of $(\geo(0, v_{I(k)}))_k$ in the sense of \eqref{eq:convdef}, we can produce some infinite geodesic $g \in \tree_0$. Since $\functionals$ is finite, Theorem \ref{rcg dense}, Theorem \ref{AH1}, and Theorem \ref{AH2} give the existence of some random coalescing geodesic $G$ such that
\[\P\left(|I(k)| = \infty\text{ and some subsequence of $(\geo(0, v_{I(k)}))$ converges to $G$ } \right) > 0\ . \]
But this contradicts Lemma \ref{backwards}, and so we have proved $|\functionals| = \infty$.
\end{proof}
We can now identify appropriate sequences $(\rho^+_n), (\rho_n^-)$ of elements of $\functionals$, whose corresponding random coalescing geodesics will play the roles of $(G_n^+), (G_n^-)$ alluded to above.
\begin{definition}\label{defin:Gplus}
We let $\functionals_v^+ \subseteq \functionals$ (with $v$ as in \eqref{eq:vktov}) be the set  of $\rho \in \functionals$ such that $\arc(\rho) \subseteq [0, \pi]$ with $\inf \arc(\rho) > v$.
We choose a maximal counterclockwise  decreasing sequence $(G_n^+)$ of random coalescing geodesics whose functionals $\rho_n^+ = \rho_{G_n^+}$ lie in $\functionals_v^+$. That is, $\inf \arc(\rho_m^+) \geq \inf \arc(\rho_n^+)$ for $m \leq n$; moreover, there is no random coalescing geodesic $G$ with $\rho_G \in \functionals_v^+$ such that $\inf \arc(\rho_n^+) > \sup \arc(\rho_G) > \inf \arc(\rho) > v$ for each $n$ or such that $G_n \prec G$ in some $H_i$ for all large $n$.


Informally, this is a counterclockwise decreasing sequence of geodesics whose functionals are not supporting at $v$, but which get as close to being supporting at $v$ as possible.
We define $\functionals_v^-$,  considering arcs now as subsets of $[-\pi/2, \pi/4]$, to be the set of elements $\rho \in \functionals$ such that $\sup \arc(\rho) < v$.
The functionals and random coalescing geodesics $\rho_n^-$ and $G_n^-$ are a maximal counterclockwise increasing sequence: we choose $\rho_n^-$'s so that $\sup \arc(\rho_m^-) \leq \sup \arc(\rho_n^-) < v$ for $m \leq n$ and so there is no $G$ with $\rho_G \in \functionals_v^-$ which lies counterclockwise of all $G_n^-$ for large $n$. 
\end{definition}

We note that it is possible that the sequence $(\rho_n^+)$ (similarly $(\rho_n^-)$) be constant or eventually constant. For instance, if $\ball$ is a square with sides aligned with the axes and $v  = 0$, it is \emph{a priori} possible that $\functionals$ consists of four elements, the ones associated to the four sides of $\partial \ball$ whose existence is guaranteed by Theorem~\ref{DH1}.  In this case, for each $n$ the functional $\rho_n^+$ will take the same value, namely an appropriate multiple of $\be_2$.

The statement of the next lemma involves the random subsequence of vertices $(v_{I(k)})$ from \eqref{eq:Ikdef}, defined on the event $\mathcal{D}$ from \eqref{assumption3}. 

  \begin{lemma}\label{corners}
    Assume \eqref{assumption} holds, and recall that $v_k/|v_k|\to v$. Then:
    \begin{enumerate}
    \item \label{it:cornerv} $\ball$ is not differentiable ``at $v$'' (i.e., at $v / \mu(v)$).
    \item \label{it:cornerv2} There is a  subevent  $\mathcal{D}_1 \subseteq \mathcal{D}$ with $\P(\mathcal{D}_1) = \P(\mathcal{D})$ on which any subsequential limit $\gamma \in \tree_0$ of $(\geo(0, v_{I(k)}))_k$ has $\arc(\rho_\gamma) = \{v\}$. In particular, there exist functionals $\rho \in \functionals$ with $\arc(\rho) = \{v\}$, and indeed infinitely many such functionals.
    \end{enumerate}
\end{lemma}
\begin{proof}
	
	We begin by proving item \eqref{it:cornerv}; we now assume to the contrary that $v / \mu(v) \in \partial \ball$ is a point of differentiability, with corresponding tangent functional $\varpi$, and derive a contradiction. Let $G$ denote the random coalescing geodesic with $\rho_G = \varpi$ (whose existence follows from Theorems \ref{DH2} and \ref{rcg dense} and from Proposition~\ref{rcg}). We first argue that 
\begin{equation}
\label{eq:existshirho}
\exists i \in \{0, 1\} \text{ such that $G$, and $G_n^+, G_n^-$ for all large $n$, eventually move into $H_i$.}
\end{equation}
This follows from the convexity and symmetry of $\ball$, using the fact that $v \in [0, \pi/4]$.

Indeed, suppose first that $\partial \ball$ admits more than four (and hence at least eight) tangent functionals.  In this case, Proposition~\ref{prop:diamondprop} guarantees there is a supporting functional
$\varpi'$ in $\functionals_v^+$ with $v, \pi/2 \notin \arc(\varpi')$. Reflecting $\varpi'$ about the $\be_1$-axis produces another functional $\varpi'' \in \functionals_v^-$ with $-\pi/2 \notin \arc(\varpi'')$, and then the maximality from Definition \ref{defin:Gplus} guarantees that $G_n^+$ and $G_n^-$ eventually move into $H_0$ for all large $n$. 

Otherwise, $\ball$ is either a ``diamond'' ($\ell^1$ ball) or ``square'' ($\ell^\infty$ ball).  In the former case, $\arc(\varpi) = [0, \pi/2]$ and, since $|\functionals| = \infty$ by Proposition \ref{prop:diamondprop}, we have that there exist infinitely many supporting functionals to $\partial \ball$ at $\be_1 / \mu(\be_1)$ and $\be_2 / \mu(\be_2)$ (in particular, $v \neq 0$). Choosing $\varpi'$ and $\varpi''$ such that $\arc(\varpi') = \{\pi/2\}$, $\arc(\varpi'') = \{0\}$, we see that $G_n^+$ and $G_n^-$ must eventually move into $H_1$ for all large $n$, proving \eqref{eq:existshirho}. In the latter case, by Proposition \ref{prop:diamondprop} there must then exist infinitely many elements $\rho$ of $\functionals$ with $\arc(\rho) = \{\pi/4\}$ and $\{-\pi/4\}$, and then the argument proceeds similarly to before.

We fix $H_i$ as in \eqref{eq:existshirho} and fix large $n$ such that $G_n^+$ and $G_n^-$ eventually move into $H_i$ --- as usual, for definiteness of notation, we take $H_i = H_0$. 
We consider a fixed outcome $\omega \in \mathcal{D} \cap \mathcal{X}$.
Since $v \notin \arc(\rho_n^+) \cup \arc(\rho_n^-)$, for any fixed $n$ the geodesics $\geo(0, v_k)$ must all lie in the region between $G_n^+$ and $G_n^-$ in $H_0$ (in the sense of Definition \ref{defin:regionbetweentwo}) for large $k$. In particular, any subsequential limit $\gamma$ of $(\geo(0, v_{I(k)}))_k$ must eventually move into $H_0$ and indeed lie in the region in $H_0$ between $G_n^+$ and $G_n^-$, by Lemma \ref{lem:geobetween}. 

Item~\ref{it:calx5} of Definition~\ref{defin:calx} then implies that, treating arcs as subsets of $[-\pi/2, \pi]$,
\[ \sup \arc(\rho_n^-) \leq \inf \arc(\rho_\gamma) \leq \sup \arc(\rho_\gamma) \leq \inf\arc(\rho_{n}^+)\]
with $\rho_\gamma \neq \rho_n^+, \rho_n^-$.
 Since $n$ was large but arbitrary, 
\begin{equation}
\label{eq:limnodiff0}
\text{$\rho_\gamma$ is actually a supporting functional to $\ball$ at $v/\mu(v)$;}
\end{equation}
 in other words, $\rho_\gamma = \varpi$, the unique such supporting functional.

But by Theorem \ref{AH2}, $G$ is a.s.~the unique element of $\tree_0$ having $\rho_G = \varpi$. So the above has actually shown that
\[\text{on a full-measure subevent of $\mathcal{D}$, a subsequence of $(\geo(0, v_{I(k)}))$ converges to $G$.}  \]
Since $\P(\mathcal{D}) > 0$, this contradicts Lemma \ref{backwards} and therefore our assumption that $\partial \ball$ is differentiable at $v$, showing the first claim \eqref{it:cornerv} of the proposition.

We now turn to the second claim of the proposition. We will first show that \eqref{eq:limnodiff0} still holds, even though we no longer assume that $\partial \ball$ is differentiable at $v / \mu(v)$:
\begin{equation}
\label{eq:limnodiff}
\text{on $\mathcal{X}$, each subsequential limit $\gamma$ of $(\geo(0, v_{I(k)}))_k$ has $v \in \arc(\rho_\gamma)$}
\end{equation}
or in other words, $\rho_\gamma$ is a supporting functional at $v$.
Since $\partial \ball$ is not differentiable at the point $v / \mu(v)$, this point is an extreme point of $\ball$. If there exists another point $w \in [0, \pi/2)$ such that $w / \mu(w)$ is an extreme point of $\ball$ (which is true in particular if $v \in (0, \pi/4)$), then the analogue of \eqref{eq:existshirho} with $G$ replaced by $\gamma$ holds by an argument nearly identical to the one used previously. Then the argument used to establish \eqref{eq:limnodiff0} again establishes \eqref{eq:limnodiff}. 

Proving \eqref{eq:limnodiff} is slightly more complicated in the case that no $w$ as above exists --- in other words, $\ball$ is a square or diamond, with $v$ as one of its corners. The argument is nearly identical in both cases, so for definiteness of notation we assume 
\begin{equation}
\label{eq:assumediamond}
\text{$\ball$ is a multiple of the $\ell^1$ unit ball and $0 = v$.}
\end{equation}
In this case, on $\mathcal{X}$, every $\gamma$ arising as as subsequential limit of $\geo(0, v_{I(k)})$ must have either $\arc(\rho_\gamma) = [j \pi/2, (j+1)\pi/2]$ (so $\rho(\gamma)$ is a supporting functional along a ``side'' of the diamond) or $\arc(\rho(\gamma) = \{j \pi/2\}$ for some $j \in \{0, 1 ,2, 3\}$. There are at most four elements $\rho \in \functionals$ with $\arc(\rho_\gamma)$ of the form $[j \pi/2, (j+1)\pi/2]$; applying Lemma \ref{backwards}, we see that a.s.~no subsequential limit $\gamma$ can have one of these four functionals as its $\rho_\gamma$. Thus, we will have shown \eqref{eq:limnodiff} (in case \eqref{eq:assumediamond}) once we establish
\begin{equation}
\label{eq:noother}
\text{a.s., no limit $\gamma$ of $(\geo(0, v_{I(k)})$ can have $\arc(\rho_\gamma) = \{j \pi/2\}$ for $j = 1, 2, 3$}.
\end{equation}

If \eqref{eq:noother} were false, then applying Proposition \ref{prop:diamondprop} and Lemma \ref{backwards}, we see that there would exist random coalescing geodesics $G', G''$ with $\arc(\rho_{G'}) = \arc(\rho_{G''}) = \{j \pi/2\}$ (in particular, eventually moving into $H_{2j}$) such that
\[\P(\exists \text{ a subsequential limit $\gamma \neq G', G''$ of $(\geo(0, v_{I(k)}))$
with $G' \prec \gamma \prec G''$ in $H_{2j}$}) > 0. \]
But the vertices $v_k$ do not lie in the region between such $G'$ and $G''$ for large $k$, and so the last display contradicts Lemma \ref{lem:geobetween}. This contradiction establishes \eqref{eq:noother} and hence \eqref{eq:limnodiff}.

It remains to use \eqref{eq:limnodiff} to conclude the second claim of the proposition. Another application of Lemma \ref{backwards} shows that for any fixed supporting functional $\rho$ (or finite set of supporting functionals) at $v$, there is a.s.~no limit of $(\geo(0, v_{I(k)}))$ whose Busemann function is asymptotically linear to $\rho$. The proof concludes by noting that (by convexity) there can be at most two supporting functionals $\rho$ at $v$ with $\arc(\rho) \neq \{v\}$.
	\end{proof}

\subsubsection{Control of $0 \ni \gamma \in \tree_{v_k}$}
In the previous lemma, we showed that subsequential limits of the geodesics $\geo(0, v_{I(k)})$ are directed toward $v$. From this, we will be able to show an appropriate version of the statement ``$\geo(0, v_k)$ eventually moves into some $H_i$ uniformly for large $k$;''  see Lemma \ref{lem:hpfromvk2} below. First, we provide an analogue ``from the perspective of $v_{I(k)}$'' by controlling the behavior of $\geo(0, v_{I(k)})$ near its endpoint $v_{I(k)}$. Because we cannot apply Lemma \ref{backwards} (since $0 \in \tree_{v_k}$ does not necessarily mean $v_k \in \tree_0$), the conclusions we can draw here are slightly weaker, and in particular we will have less control over the particular $H_i$ our geodesics will move into.

 To shorten the statement of this lemma (and the event appearing therein), we introduce the following notation:
	\begin{equation}
	\label{eq:Xixdef}
	\text{For each $x \in \Z^2$ and $k \geq 1$,}\quad	\Xi_x(k) = \Xi_x := \{\gamma \in \tree_x: \, \exists \gamma' \in \tree_{x + v_k} \text{ with } \gamma \subseteq \gamma' \}.
	\end{equation}
  \begin{lemma}\label{lem:hpfromvk}
    There is a choice of half-plane $H_i\ni v$, such that the following holds for all large $L$:

    	    \[\limsup_{k \to \infty}\, \P\left(
    	    \begin{array}{c}
    	   \geo(0, v_k) \subseteq v_k + H_{i+4} + [-L, L]^2, \\
    	     0 \in \tree_{v_k}, \,  \text{$\Xi_0(k)$ uniformly moves into $H_{i+4}$},\\
    	     \text{and indeed }\bigcup_{\gamma \in \Xi_0(k)} \gamma \subseteq H_{i+4} + [-L, L]^2, 
    	    \end{array}
    	    \right) \geq \delta/2\ , \]
    	    where we recall that ``uniformly moves into $H_i$'' was defined at \eqref{eq:unifmoveintodef}.
    \end{lemma}
    We have written the half-plane in the form $H_{i+4}$ for notational convenience, since its rotation $H_i$ (for the same value of $i$) will frequently be used; see e.g.~Lemma \ref{lem:hpfromvk2}.
\begin{proof}
	We extend our $\Xi_x(k)$ notation from \eqref{eq:Xixdef} to negative values of $k$ by setting $v_{-k} = -v_k$ for $k \geq 1$.
  By symmetry, it suffices to show the analogue of the lemma from the perspective of $0$:
    \begin{equation}\label{eq:from0confine}
    \limsup_{k \to \infty}\, \P\left(
  \begin{array}{c}
  \geo(0, v_k) \subseteq (H_i + [-L, L]^2), \\
  v_k \in \tree_{0}, \,   \text{ $\Xi_{v_k}(-k)$ uniformly moves into $H_i$},\\
   \text{and indeed } \bigcup_{\gamma \in \Xi_{v_k}(-k)} \gamma \subseteq v_k + H_i + [-L, L]^2
  \end{array}
  \right) \geq \delta/4 \end{equation}
      for all $L$ larger than some ($k$-independent) constant.
The result of Lemma \ref{corners} implies  (along with Theorem \ref{rcg dense}) that there is some $\rho \in \functionals$ with $\arc(\rho) = \{v\}$ and an associated random coalescing geodesic $G_*$ with $\rho_{G_*} = \rho$ almost surely. Intuitively, \eqref{eq:from0confine} is true because the geodesic $\geo(0, v_k)$ must either lie in the region between $G_*$ and $G_n^+$ or the region between $G_*$ and $G_n^-$. However, because it is not guaranteed that all three of these geodesics lie in a common $H_i$, we have to take care in defining these regions.

  We for the moment fix $n \geq 1$ large enough that
\begin{equation}\label{eq:stari}
  \text{$\arc(G_n^+)$ and $v$ are both contained in a common (open) $H_i$};
\end{equation}
this is possible since there exist analogues of $G_*$ and $\rho$ associated to rotations and reflections of $v$. Enlarging $n$, we can ensure \eqref{eq:stari} also holds when we replace $G_n^+$ by $G_n^-$ (possibly also replacing $H_i$ with a different half-plane, to be denoted by $H_j$). 
For $H_i$ as in \eqref{eq:stari},
\begin{equation}
  \label{eq:minusv}
  \text{the point $-v$  lies outside of the closed set $[H_i \cup \partial H_i]$,}
\end{equation}
with an analogous statement holding for $H_j$.

By Lemma \ref{backwards}, the set $\{k \geq 1: \, v_k \in G_*\}$ is finite almost surely. 
We fix an outcome $\omega \in \mathcal{X}$ such that finitely many $v_k$ lie in $G_*$. We will show that in $\omega$, 
 for large $k$,
 \begin{equation}
   \begin{split}
   \label{eq:allkbet}
   \text{ $v_k$ lies strictly between $G_*$ and $G_n^+$ in $H_i$ or strictly between $G_*$ and $G_n^-$ in $H_j$.}
 \end{split}
\end{equation}
(meaning that $v_k$ lies in the region between these two geodesics as in Definition \ref{defin:regionbetweentwo}, but that $v_k \notin G_n^+ \cup G_*$.)
   We begin by defining an appropriate version of the ``region $C$ between $G_n^+$ and $G_n^-$'' and arguing a weaker version of \eqref{eq:allkbet}: that $v_k$ must lie in $C$. Let $\zeta$ denote the last vertex common to both $G_n^+$ and $G_n^-$. As in Definition \ref{defin:regionbetweentwo}, the segments of $G_n^+$ and $G_n^-$ from $\zeta$ onwards form a doubly infinite simple path $P$ (considered as a plane curve). The Jordan curve theorem shows that $P$ divides $\R^2$ into two connected components. Because 
   \begin{equation}
       \label{eq:vnotindir}
       v \notin \dir(G_n^+) \cup \dir(G_n^-)
   \end{equation} (and the latter is a closed set), the vertices $v_k$ (for large $k$) and all large multiples of $v$ must lie in the same component $C$ of $\mathbb{R}^2 \setminus P$. It is easy to see from \eqref{eq:allkbet} and \eqref{eq:vnotindir} that every continuous curve connecting a large multiple of $v$ to $-v$ must intersect $P$. This shows the promised ``weaker version of \eqref{eq:allkbet}.'' 

   We now show \eqref{eq:allkbet} using a topological argument similar to many others we have made; again for definiteness, for the remainder of this paragraph, we assume $G_n^-$ and $G_*$ both eventually move into $H_0$ (in other words, $H_0 = H_j$). We write $C^+$ for the region between $G_n^+$ and $G_*$ (in $H_i$) and $C^-$ for the region between $G_n^-$ and $G_*$ (in $H_j = H_0$). Let us take $k$ large enough that the portions of the  geodesics $G_n^+$, $G_n^-$, and $G_*$ lying in $H_0 + (v_k \cdot \be_1) \be_1$ are all distinct. We also assume that 
   the ray
  \[\pi := \left\{v_k + \frac{1}{2} [\be_1 + \be_2] - t \be_2: \, t \geq 0\right\}\]
  does not intersect $G_n^+$, which also holds for large $k$, since the elements of $\dir(G_n^+)$ have larger $\be_2$-coordinate than $v$. The starting point of $\pi$ is $v_k$, which has some membership status with respect to $C, \, C^+,\, C^{-}$. For $t$ large, the point $v_k + \frac{1}{2} [\be_1 + \be_2] - t \be_2$ lies outside of $C^+$, $C^-$, and $C$. 
  
  The path $\pi$ can cross only $G_*$ and $G_n^-$. Each time $\pi$ crosses $G_n^-$, it enters or exits  $C_-$ and also $C$; in particular, there must be an odd number of such crossings, since $\pi$ ends up outside $C$. Similarly, each crossing of $G_*$ corresponds to an entrance or exit of $C_+$ and also $C_-$. If the number of crossings of $G_*$ is odd, then $\pi$ crosses from $C_+$ to its complement, so $v_k \in C_+$; if the number is even, then $\pi$ crosses from $C_-$ to its complement, so $v_k \in C_-$. This proves \eqref{eq:allkbet} holds for $\omega$, for all large $k$.
  
 Thus it follows from Lemma \ref{lem:geobetween} that, in $\omega$, the geodesic $\geo(0,v_k)$ lies in $C_+$ or $C_-$ (possibly with some vertices belonging to the initial portions of $G_n^+$, $G_n^-$, $G_*$ before these three geodesics branch). Since $\omega$ was an arbitrary element of a probability one event,  by \eqref{assumption} we have  
 \begin{equation}
   \label{eq:maxregion}
   \limsup_{k \to \infty} \max_{U = C^+, C^-} \{ \P( \mathcal{X} \cap \{v_k \in \tree_{0}, \, \geo(0, v_k)  \text{ ends in } U \})  \} \geq \delta / 2 \ .
 \end{equation}
 
We assume that the maximal probability in \eqref{eq:maxregion} is attained when $U = C^+$ (the other case is similar) and consider an outcome in the corresponding event. In this outcome, for each $\gamma \in \tree_0$ with $v_k \in \gamma$, the initial segment $\geo(0, v_k)$ of $\gamma$ terminates at a point lying strictly within $C^+$ --- hence, by Lemma \ref{lem:geobetween} again, $\gamma$ itself lies between $G_n^+$ and $G_*$ in $H_i$, and in particular eventually moves into $H_i$.

Because $C^+$ contains only finitely many vertices of $\Z^2 \setminus H_i$, the claim of \eqref{eq:from0confine} involving $\gamma$ holds:
on the event in \eqref{eq:maxregion}, with $U = C^+$,
\begin{equation}
  \label{eq:almostfrom0}
  \text{any $\gamma \in \tree_0$ with $v_k \in \gamma$ is contained in $H_i \cup [-L, L]^2$ for some large random $L$}.
\end{equation}
Of course, an analogous statement holds in the case $U = C^-$, with $H_i$ replaced by $H_j$.
Thus, in the case $U = C^+$, we can choose an $L$ large such that 
\begin{equation}
    \label{eq:almostXi}
    \limsup_{k \to \infty}\, \P\left(
  \begin{array}{c}
  \geo(0, v_k) \subseteq (H_i + [-L, L]^2), \\
  v_k \in \tree_{0}, \,   \text{ $\Xi_{v_k}(-k)$ uniformly moves into $H_i$}
  \end{array}
  \right) \geq 3\delta/8\ .
\end{equation}

The event from \eqref{eq:maxregion} also allows us to guarantee the remaining condition on  $\Xi_{v_k}(-k)$ from  \eqref{eq:from0confine} holds. Indeed, let us again assume $U = C^+$ for definiteness. Any $\gamma \in \Xi_{v_k}(-k)$ is a subsegment of some $\gamma' \in \tree_0$, which must (by the preceding discussion) lie in the region between $G_n^+$ and $G_*$ in $H_i$. Since $G_n^+$ and $G_*$ coalesce with $G_n^+(v_k)$ and $G_*(v_k)$, we also have $G_*(v_k) \prec \gamma \prec G_n^+(v_k)$ in $H_i$. We can choose a deterministic large $L$ such that
\[\P\left(\begin{array}{c}\text{ the region between $G_*(v_k)$ and $G_n^+(v_k)$ in $H_i$ is contained}\\ \text{in $[v_k + H_i] \cup [v_k + [-L,L]^2]$}
\end{array}
\right) > 1-\delta / 8 \]
uniformly in $k$.

Pulling together this observation with \eqref{eq:almostXi} shows we can choose $L$ large such that \eqref{eq:from0confine} holds, completing the proof of the lemma.
  \end{proof}

  Lemma \ref{lem:hpfromvk} allows us to localize $\geo(0, v_k)$ near its endpoint $v_k$ on $\{0 \in \tree_{v_k}\}$. The following lemma is the counterpart for the other endpoint, $0$. As alluded to previously, we are able to prove a stronger result for this endpoint thanks to Lemma \ref{backwards}. In particular, though we cannot guarantee that  Lemma \ref{lem:hpfromvk} would hold with any particular value of $i$, the following lemma guarantees that $\geo(0, v_{I(k)})$ is contained in an enlargement of $H_i$ for the same value of $i$:

\begin{lemma}\label{lem:hpfromvk2}Assuming \eqref{assumption} holds,
  \[\limsup_{L \to \infty} \limsup_{k \to \infty} \P(0 \in \tree_{v_k}, \, \geo(0, v_k) \not \subseteq H_i \cup [-L, L]^2) = 0\ ,  \]
  where $i$ has the same value as in the statement of Lemma \ref{lem:hpfromvk}.
			\end{lemma}
			\begin{proof}
				We fix an arbitrary $\varepsilon > 0$. We will establish the claim of the lemma by making a choice of positive integers $L,K$ such that 
				\begin{equation}
				\label{eq:hanks}
					\P(0 \in \tree_{v_k}, \, \geo(0, v_k)  \not \subseteq H_i \cup [-L, L]^2) \leq \varepsilon \text{ for all $k > K$.}
					\end{equation}
				We will focus on the event $\mathcal{D}$ defined in \eqref{assumption3}; we first show that we can disregard contributions to \eqref{eq:hanks} from the complementary event $\mathcal{D}^c$.
				We upper bound the probability in \eqref{eq:hanks} by
				\begin{equation}
				\label{eq:hanks2}
				\P(0 \in \tree_{v_k}, \,\geo(0, v_k)  \not \subseteq H_i \cup [-L, L]^2, \, \mathcal{D}) + \P(0 \in \tree_{v_k}, |I| < \infty)
				\end{equation}
				(where we have discarded the zero-probability event $\mathcal{X}^c$ as well). The latter probability converges to $0$ as $k \to \infty$, so we can choose a $K < \infty$ such that $\P(0 \in \tree_{v_k}, |I| < \infty) < \varepsilon / 3$ for all $k > K$. It remains to control the first term of \eqref{eq:hanks2}.

				As a first step, we recall the event $\mathcal{D}_1$ from the statement of Lemma \ref{corners}; since $\P(\mathcal{D} \setminus \mathcal{D}_1) = 0$, we can replace $\mathcal{D}$ by $\mathcal{D}_1$ in the first term of \eqref{eq:hanks2}.
                                As in \eqref{eq:stari}, we choose a $G_n^+$ and a half-plane of the form $H_j$ such that both $G_n^+$ and any geodesic $\gamma$ with $\dir(\gamma) = \{v\}$ eventually move into $H_j$ (
                                 we write $H_j$ to distinguish this half-plane from the special half-plane $H_i$ appearing in the statement of the lemma).
                                 
                                The geodesic $G_n^+$ will serve to uniformly bound the set of subsequential limits of $(\geo(0, v_{I(k)}))$, allowing us to find an extremal subsequential limit. Formally, we can choose some subsequential limit $g$, which (because $\arc(\rho_g) = \{v\}$ on $\mathcal{D}_1$) satisfies $G_n^+ \prec g$ in $H_j$. The family 
                                \[\{\gamma:\,\text{$\gamma$ is a subsequential limit of $(\geo(0,v_{I(k)}))$ with $\gamma \prec g$} \} \]
                                satisfies the assumption of  Lemma~\ref{lem:extremal}. Thus, for each outcome in $\mathcal{D}_1$, we may define such a random subsequential limit $g_+$ which satisfies $g_+ \prec \gamma'$ in $H_j$ for any subsequential limit $\gamma'$ of $(\geo(0, v_{I(k)}))$.
                                
                                 Because on $\mathcal{D}_1$ each subsequential limit $\gamma$ of $(\geo(0, v_{I(k)}))$ also eventually moves into $H_i$, we have by \eqref{it:twoplanes} of Lemma~\ref{lem:ordergood} that $g_+ \prec \gamma$ in $H_i$ as well. 
                                 A similar argument yields a subsequential limit $g_-$ such that each $\gamma \prec g_-$ in $H_i$.
                                  For each fixed random coalescing geodesic $G$, by Lemma \ref{backwards}, we have $g_+, g_- \neq G$ for almost every outcome in $\mathcal{D}_1$. Thus, we can choose distinct random coalescing geodesics $G_1, G_2$ with $\arc(\rho_{G_1}) = \arc(\rho_{G_2}) = \{v\}$ such that for all $k  > K$,
				\begin{equation}
				\label{eq:hanks3}
				\P(\mathcal{D}_1, G_1 \prec g_+ \prec g_- \prec G_2 \text{ in $H_i$, with $g_+ \neq G_1$ and $g_- \neq G_2$}) > \P(\mathcal{D}_1) - \varepsilon/3\ .
				\end{equation}
				
				On the event in \eqref{eq:hanks3}, Lemma \ref{lem:geobetween} tells us that for large $k$, the elements of the sequence $(\geo(0, v_{I(k)}))$ lie in the region between $G_1$ and $G_2$ in $H_i$ (with some initial vertices shared with $G_1$ and $G_2$, as usual).  In particular, if we  bound the first term of \eqref{eq:hanks2} by
                                \begin{equation}\label{eq:hanks4}
                                \begin{gathered}
      				  \P(0 \in \tree_{v_k}, \,   \mathcal{D}_1, \geo(0, v_k) \text{ is not between $G_1$ and $G_2$ in $H_i$}) \\
                                  + \P(\mathcal{D}_1,\, \text{$G_1$ or $G_2$ intersects $\Z^2 \setminus H_i$ outside of $[-L, L]^2$})\ , 
                                  \end{gathered}
                                \end{equation}
                                then we see the first term of \eqref{eq:hanks4} is smaller than $\varepsilon /3$ uniformly in $L$ for $k$ sufficiently large. On the other hand, the second term of \eqref{eq:hanks4} can be made smaller than $\varepsilon / 3$ uniformly in $k$ large by taking $L$ sufficiently large, by the fact that $G_1$ and $G_2$ eventually move into $H_i$. In other words, returning to \eqref{eq:hanks}, we see
                                \[\P(0 \in \tree_{v_k}, \, \geo(0, v_k)  \not \subseteq H_i \cup [-L, L]^2) < \varepsilon \]
                                for a fixed large $L = L(\varepsilon)$, uniformly in large $k$.
				\end{proof}

\subsection{Constructing a non-crossing family: the first case. \label{sec:finseg}}

      In the previous section, we proved several lemmas allowing us to ``localize'' the geodesics $\geo(0, v_{I(k)})$ within half-planes. The reasons for the particular form of Lemmas \ref{lem:hpfromvk}, \ref{lem:hpfromvk2} will become clearer as we proceed; for now, it is worth briefly recalling our chief goals. We will show that assumption \eqref{assumption} leads to a contradiction by showing that under that assumption, one can construct a shift-invariant measure on non-crossing geodesics which assigns positive probability to graphs with infinitely many coalescence classes. This will contradict Theorem \ref{thm:non-crossing} and prove that assumption \eqref{assumption} is false. In fact, the ``non-crossing'' portion of the construction is not obvious, and this will force us to consider two different possible constructions --- this section is devoted to the first, and Section \ref{sec:finseg2} is devoted to the other. Which construction we use depends on whether \eqref{raul seixas} below holds or not; we discuss the technical problem in more detail at \eqref{raul seixas}.

      As a first step, we need to produce a suitable shift-invariant family of geodesics: we build these from shifted analogues of $\geo(0, v_{I(k)})$. More accurately, we will choose geodesics $\geo(x, x+ v_k)$ for $x$ such that the event $\good_k(x)$ of Definition \ref{defin:goodevent} below occurs. 
\begin{definition}
  \label{defin:goodevent}
  We define $\good_k(x) = \good_{k,L}(x)$ to be the subevent of $\mathcal{X}$ on which (in addition to the defining conditions of $\mathcal{X}$ from Definition \ref{defin:calx}) the following hold (with the same choice of $H_i$ as in Lemmas \ref{lem:hpfromvk} and \ref{lem:hpfromvk2} and with $\Xi_x(k)$ as in \eqref{eq:Xixdef}):
\begin{enumerate}[label = {[\alph*]}]
\item $x \in \tree_{x + v_k}$;
\item $\geo(x,x +  v_k) \subseteq [x + (H_i \cup [-L/4,L/4]^2)]$, and $\geo(x, x + v_k)  \subseteq x + v_k + (H_{i+4} + [-L/4, L/4]^2)$;\label{it:good2}
  \item For each $\gamma \in \Xi_x(k)$, we have $\gamma \subseteq (x + H_{i+4} + [-L/4, L/4]^2])$,  and furthermore the family $\Xi_x(k)$ uniformly moves into $H_{i+4}$.\label{it:good3}
\end{enumerate}
We omit the argument in the case $x = 0$, writing $\good_k = \good_{k,L}$ instead of $\good_k(0) = \good_{k,L}(0)$.
\end{definition}
We fix an $L_0$ uniform in $k$ such that
\begin{equation}
  \label{eq:goodprob}
  \text{for each $L \geq L_0$, we have} \quad \P(\good_{k,L}) > \delta / 8\quad \text{for all large $k$}\ .
\end{equation}
We note that those Lemmas \ref{lem:hpfromvk} and \ref{lem:hpfromvk2} guarantee that such a choice of  $L_0$ can be made.

\paragraph{\bf A convention about half-planes}
Much of the subsequent work of the paper follows from consequences of \eqref{eq:goodprob} --- indeed, we will show that properties of $\good_k$ lead to a contradiction which shows \eqref{assumption} is false. As in many of our proofs above, to avoid cumbersome notation, we assume that
\begin{equation}
  \label{eq:assumehalf}
\text{the half-plane $H_i$ invoked in the definition of $\good_k$ is $H_0$.}  
\end{equation}
Assumption \eqref{eq:assumehalf} will remain in force for the remainder of this section. As usual, the adaptations in what follows required to treat the case of a different $H_i$ are very minor but notationally cumbersome.


We will use $\good_{k, L}$ to build an event $\good_{k,L}^*(x)$ which will allow us, under the assumption \eqref{assumption}, to create a family of nonintersecting geodesics which remain in a common $H_i$ (working with $H_0$ for definiteness as in \eqref{eq:assumehalf}). In fact, we cannot guarantee this event has positive probability, and this is what forces us to consider two different possible constructions --- the construction in this section will suit our needs when \eqref{raul seixas} holds, and the construction in Section \ref{sec:finseg2} will be used otherwise.

We note that for each fixed value of $L \geq L_0$, we can find infinitely many values of  $\ell \geq L$ and associated $\delta' = \delta'(L)> 0$ such that
\begin{equation}
 \label{eq:raul0}
  \limsup_{k \to \infty} \prob(\good_{k,L} \cap \good_{k,L}(\ell \be_2)) > \delta'\ .
\end{equation}
We assume for the remainder of this section that the above can be strengthened to include a non-intersection condition; namely, there exists an $L \geq L_0$ and a $\delta' > 0$ such that
\begin{equation}
  \label{raul seixas}
  \limsup_{k \to \infty} \prob(\good_{k,L} \cap \good_{k,L}(L \be_2) \cap \{\geo(0, v_k) \cap \geo(L \be_2, L \be_2 + v_k) = \varnothing\}) > \delta'\ .
\end{equation}
We denote the event appearing in \eqref{raul seixas} by $\good_{k,L}^*$. This setting in which \eqref{raul seixas} holds is the ``first case'' alluded to in the title of this section.

Similarly, for a general vertex $x \in \Z^2$, we define
\begin{align}
  \label{eq:goodstar}
  \good_{k,L}^*(x) := &\good_{k,L}(x) \cap \good_{k,L}(x + L \be_2)\\
  \nonumber\cap &\{\geo(x, x + v_k) \cap \geo(x+ L \be_2, x+ L \be_2 + v_k)  = \varnothing\}\ .
\end{align}
The value of $L$ will now be fixed for the remainder of this construction so that \eqref{raul seixas} holds. We tend to omit explicit reference to $L$ in notation going forward, writing $\good_k^*(x)$.

The non-intersection of $\geo(0, v_k)$ and $\geo(L \be_2, L \be_2 + v_k)$ appearing in the event $\good_k^*$ in fact guarantees that $\geo(0, v_k)$ is ``shielded'' from other geodesics originating above $\geo(L \be_2, L \be_2 + v_k)$. This fact, which is the content of the following lemma, is a main reason for including condition \eqref{it:good2} in the definition of $\good_k$ (Definition \ref{defin:goodevent}) and hence a main reason for our work in Lemmas \ref{lem:hpfromvk} and \ref{lem:hpfromvk2}.
\begin{lemma}
  \label{lem:crosslocal}
  Suppose $k$ is large enough that $v_k \in H_0$.
  If $\good_k^* \cap \good_k^*(j L \be_2)$ occurs for some $j \geq 1$, then $\geo(0, v_k) \cap \geo(j L \be_2, j L \be_2 + v_k) = \varnothing$.
\end{lemma}
In fact, the proof will establish the conclusion of the lemma under the weaker assumption that $\good_k^* \cap \good_k(j L \be_2)$ occurs. See Figure~\ref{fig:34} for a depiction of the relevant aspects of the event $\good_k^* \cap \good_k^*(j L \be_2)$.

	\begin{figure}\label{fig:34}
 \includegraphics[scale=0.5]{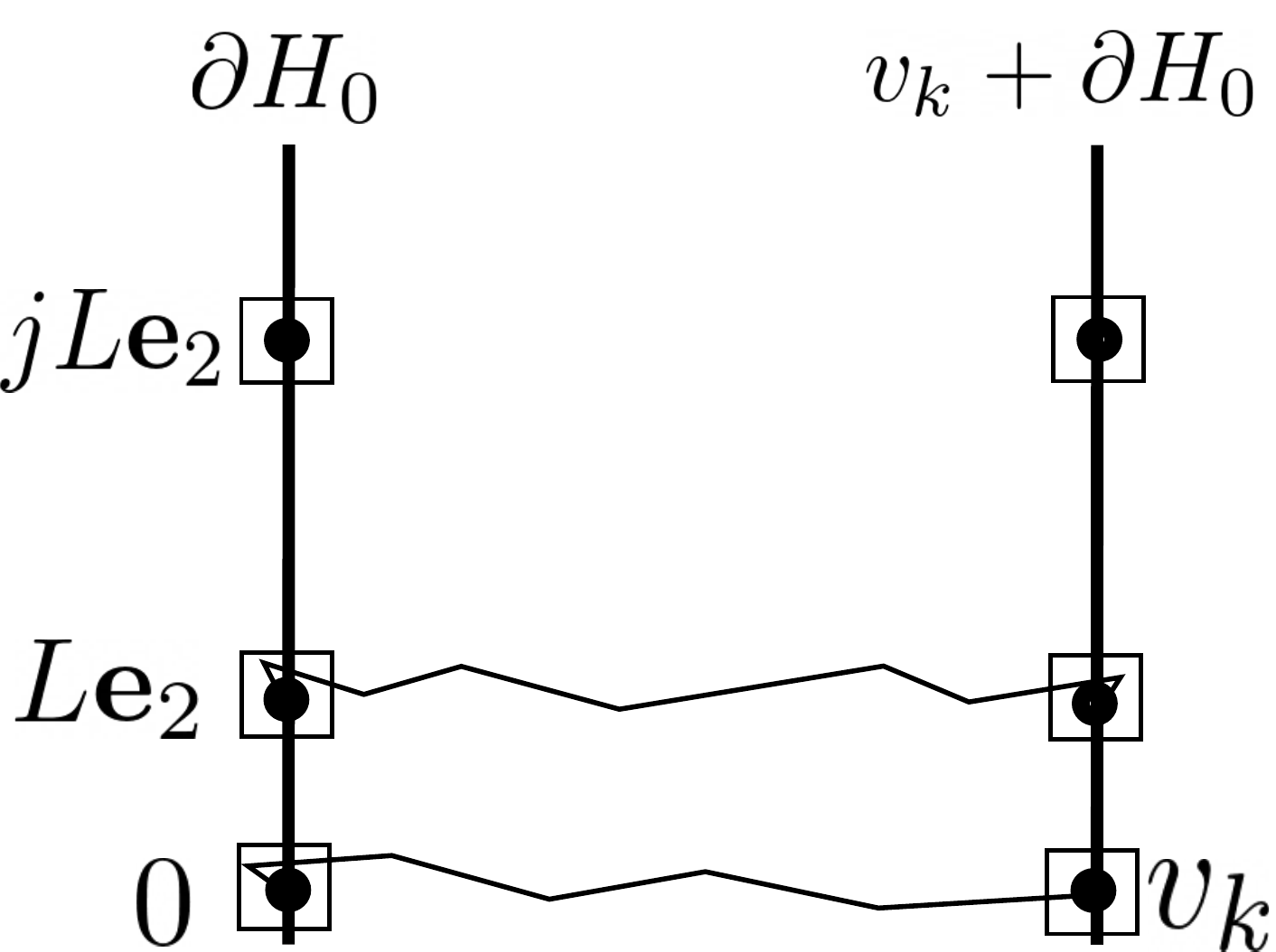}
    \caption{The relevant portions of the geometry of geodesics in the event $\good_k^* \cap \good_k^*(j L \be_2)$ from Lemma~\ref{lem:crosslocal}. The geodesics $\geo(0, v_k)$ and $\geo(L \be_2, L \be_2 + v_k)$ cannot intersect, and must remain confined in the region between $\partial H_0$ and $v_k + \partial H_0$ except within distance $L/4$ of their endpoints, the squares surrounding these points on the figure. The geodesic $\geo(j L \be_2, j L \be_2 + v_k)$ is similarly confined.  Thus, the geodesic $\geo(L \be_2, L \be_2 + v_k)$ forms a barrier preventing $\geo(0, v_k)$ and $\geo(j L \be_2, j L \be_2 + v_k)$ from intersecting.
    } \label{fig:19}
\end{figure}

\begin{proof}
  Suppose that $\good_k^* \cap \good_k^*(j L \be_2)$ occurs; if $j = 1$, the claim of the lemma is obvious. Otherwise, since $\good_k^* \subseteq \mathcal{X}$, the geodesic $\geo(L \be_2, L \be_2 + v_k)$ is well-defined and has a last intersection with $\Z^2 \setminus H_0$; we let $x_1 \in \partial H_0$ denote this last intersection, after which the geodesic remains entirely within $H_0$. After $x_1$, the geodesic touches $v_k + \partial H_0$  for the first time at some vertex $x_2$. Let $\pi$ be the segment of $\geo(L \be_2, L \be_2 + v_k)$ from $x_1$ until $x_2$. Note that $x_1 \in L \be_2 + [-L/4, L/4]^2$ and $x_2 \in L \be_2 + v_k + [-L/4, L/4]^2$, since $\good_k^*(y) \subseteq \good_{k}(y)$ for each $y$  (recalling \ref{it:good2} of Definition \ref{defin:goodevent}).

  We now construct a doubly-infinite simple curve $P$; as usual, $\mathbb{R}^2 \setminus P$ consists of two connected components. We will argue that $\geo(0, v_k)$ will lie strictly in one of these components, denoted $C_1$, and that $\geo(j L \be_2, j L \be_2 + v_k)$ lies in $C_2 \cup P$, where $C_2$ is the other connected component of $\mathbb{R}^2 \setminus P$.  This clearly implies the claim of the lemma. Because we have made several similar constructions already, we sketch this argument, giving most attention to the unique features of the setup here.

  Our curve $P$ consists of the following pieces, traversed in order:
  \begin{enumerate}[label = (\roman*)]
  \item The ray $\{-\ell \be_2: \ell \geq L/2\}$;
  \item the portion of the boundary of $[-L/2, L/2]^2$ lying outside $H_0$;
  \item The segment $\{\ell \be_2: \, L/2 \leq \ell \leq x_1 \cdot \be_2\};$
  \item The geodesic $\pi$;
  \item The segment $\{x_2 - \ell \be_2: \, 0 \leq \ell \leq (x_2 - v_k) \cdot \be_2 - L/2\};$
  \item the portion of the boundary of $v_k + [-L/2, L/2]^2$ lying in $v_k + H_0$;
  \item the segment $\{v_k - \ell \be_2: \, \ell \geq L/2\}$.
  \end{enumerate}
  We write $C_1$ for the component of $\mathbb{R}^2 \setminus P$ containing $0$, and $C_2$ for the other component. Then $\geo(0, v_k)$ starts in $C_1$. It cannot touch $\pi$ because, on $\good_k^*$, we have $\geo(0, v_k) \cap \geo(L e_2, L e_2 + v_k) = \varnothing$. It cannot touch the other pieces making up $P$ because, on $\good_k$, the geodesic $\geo(0, v_k)$ remains strictly in the union of $H_0 \cap [H_4 + v_k]$, $[-L/4, L/4]^2$, and $v_k + [-L/4, L/4]^2$. In particular, $\geo(0, v_k)$ is contained in $C_1$.

  Clearly $\geo(j L \be_2, j L \be_2 + v_k)$ begins in $C_2$. Similarly to the preceding paragraph, since $\good_{k}(j L \be_2)$ occurs, we see that this geodesic could not intersect $P \setminus \pi$. It is possible that $\geo(j L \be_2, j L e_2 + v_k)$ intersects $\pi$, but it could not then enter $C_1$. Indeed, if it entered $C_1$, it would (since $j L \be_2 + v_k \in C_2$) subsequently have to exit, and this exit would have to happen via a vertex of $\pi$. This would imply nonuniqueness of geodesics between some pair of vertices of $\pi$, which is impossible (since $\good_k \subseteq \mathcal{X}$).
\end{proof}

Using the preceding lemma, we can define a family of non-intersecting geodesics lying on any vertical line as follows. For each $u \in \Z^2$, we set
\begin{equation}
  \label{eq:Idef}
  I_{k}(u) = I_{k, L}(u) = \{-u + j L \be_2: \, j\in \Z,\, \good_{k,L}^*(-u + j L \be_2) \text{ occurs}\}\ .
  \end{equation}

We emphasize the most important aspects of this definition. First, for $x \in I_k(u)$, the geodesic $\geo(x, x + v_k)$ stays in the strip $[x + [H_0\cup \partial H_0]] \cap [x + v_k + [H_4 \cup \partial H_4]]$ except near its endpoints; second,
\begin{equation}
  \label{eq:Inoncross}
  \text{if $x, y \in I_k(u)$, then } \geo(x, x + v_k) \cap \geo(y, y +  v_k) = \varnothing\ ,
\end{equation}
where \eqref{eq:Inoncross} is an immediate consequence of Lemma \ref{lem:crosslocal}.

For fixed $u$, we set
\begin{equation}
\label{eq:Fkdefin}
\mathfrak{F}_k(u) = \{\geo(x, x+v_k):\, x \in I_k(u)\},\text{ oriented from $x$ to $x+v_k$}.
\end{equation}
In this notation, \eqref{eq:Inoncross} says that the geodesics of $\mathfrak{F}_k(u)$ are disjoint (do not intersect). The distribution of the family $\mathfrak{F}_k(u)$ is manifestly invariant by translations under multiples of $L \be_2$; we note this here for future reference:
\begin{equation}
  \label{eq:frakftrans}
  \mathfrak{F}_k(u) \text{ has the same distribution as } \mathfrak{F}_k(u) + L \be_2\ .
\end{equation}

As we alluded to in the beginning of the section, we would like to take $k \to \infty$ to produce a limiting family of infinite geodesics which do not coalesce, then use this to derive a contradiction. We will require translation invariance to make this argument, and for this reason introduced the parameter $u$ above. By averaging over values of $u$, we can take a distributional limit $\nu$ of the resulting geodesic family; we will show that the non-coalescence of these geodesics contradicts Theorem \ref{thm:non-crossing} above.

We enlarge our probability space  $\Omega_1$: for each $k \geq 1$, we define a new space
\begin{equation}
  \label{eq:omega1prime}
  \Omega^{(k)}_1 = \Omega_1 \times \left [ \{1,2, \ldots v_k \cdot \be_1 \} \times \{1, 2, \,\ldots, L\} \right ] \quad \text{ with measure } \P^{(k)} = \P \times\text{Unif}^{(k)}\ ,
\end{equation}
where $\text{Unif}^{(k)}$ is the uniform probability measure. In other words, the new coordinate appearing in $\Omega_1^{(k)}$, denoted by $U_k$, is a random variable uniform on $\{1, 2, \ldots, v_k \cdot \be_1\} \times \{1, \ldots, L\}$, and this variable is independent of the family $(\omega_e)_{e \in \Ec^2}$ of edge weights. We will use these to randomize the starting vertices of our family of geodesics: rather than considering $\mathfrak{F}_k(u)$ for fixed deterministic $u$, we will work with $\mathfrak{F}_k(U_k)$, and we will consider the distributional properties of this family as $k \to \infty$.

For each $k \geq 1$, realization of $U_k$, and edge-weight configuration $\omega \in \Omega_1$, we may represent the family $\mathfrak{F}_k(U_k)$ by a directed graph --- i.e., by a configuration in $\Omega_2$. To do this, we consider each $\geo(x, x + v_k) \in \mathfrak{F}_k(u)$ as an oriented path (i.e., a path in $(\Z^2, \dedges^2)$, directed away from $x$ and toward $x + v_k$). We then  define the mapping 
\[\eta_{k}(\omega, u) = \left(\eta_{k}(v,w)(\omega, u): \, (v,w) \in \dedges^2\right) \in \{0, 1\}^{\dedges^2} = \Omega_2\ ,\]
where $\eta_{k}(v,w)(\omega, u)$  is the indicator of the event that the edge $(v,w)$ is traversed (from $v$ to $w$) by some geodesic in $\mathfrak{F}_k(u)$:
\[  \eta_{k}(v,w)(\omega, u)=
  \begin{cases}
    1 \quad &\text{if $(v,w)$ is in some geodesic of $\mathfrak{F}_k(u)$};\\
    0 \quad &\text{otherwise.}
  \end{cases}\]
  
Similarly to \eqref{eq:psidh}, we define the map $\Psi_{k}:\Omega_1^{(k)} \to\Omega_1\times\Omega_2$ by setting
\begin{equation}
  \label{eq:psidisjoint}
  \Psi_k(\omega) = (\omega, \eta_k(\omega, U_k))\ .\end{equation}
The measure $\nu_{k},$ defined on (the Borel sigma-algebra of) $\Omega_1 \times \Omega_2$ is then the pushforward
\[\nu_k = \P^{(k)} \circ \Psi_k^{-1}\ . \]
This can be compared to the definition appearing in \eqref{DH average}.

It is elementary that $(\nu_k)_k$ is a tight sequence. We can thus, by Prokhorov's theorem, set $\nu$ to be any subsequential (weak) limit of $(\nu_k)_k$:
\begin{equation}
  \label{eq:nudef}\nu = \lim_{m \to \infty} \nu_{k_m} \quad \text{for some sequence $k_m \to \infty$}\ .\end{equation}
We recall that we write a typical sample point in $\Omega_1 \times \Omega_2$ as $(\omega, \eta)$. As we will see below (in Lemma \ref{maluco beleza}), $\nu$ is supported on $(\omega, \eta) \in \Omega_1 \times \Omega_2$ such that $\eta$ encodes geodesics in the weight configuration $\omega$: each directed path in the graph encoded by $\eta$ will turn out to be a geodesic for $(\omega_e)$.
The measure $\nu$ can be thought of as a distribution on limiting geodesics of the form $\geo(x, x+v_k)$, observed far away from their endpoints. In the remainder of this section, we center our analysis on the measure $\nu$, which will be shown to have contradictory properties under the assumption \eqref{raul seixas}.

\subsubsection{Deriving a contradiction under \eqref{raul seixas}}
In this portion of Section \ref{sec:finseg}, we conclude the analysis of our construction under the additional assumption \eqref{raul seixas}, by showing that the measure $\nu$ from \eqref{eq:nudef} contradicts our earlier results. Specifically, we will show  that $\nu$ is translation-invariant (Lemma \ref{sebastian}), supported on non-crossing geodesics (Lemma \ref{maluco beleza}), and in fact assigns positive probability to a set of configurations $\eta \in \Omega_2$ having infinitely many coalescence classes (Lemma \ref{cornelis}). This is in contradiction to Theorem \ref{thm:non-crossing}, and completes the proof of the contradiction under \eqref{raul seixas}.

Each of the properties enjoyed by $\nu$ follows from related properties of the measures $\nu_{k}$ (often via properties of the family $\mathfrak{F}_k(U_k)$, introduced at \eqref{eq:Fkdefin}) --- and as a result, many  analogues of the below lemmas hold for $\nu_k$ as well.

First, we note that the marginal of $\nu$ on $\Omega_1$ coincides with $\P$:
\begin{equation}
  \label{eq:weightmarginal}
  \nu(A \times \Omega_2) = \P(A) \quad \text{for any Borel } A \subseteq \Omega_1\ .
\end{equation}
The observation \eqref{eq:weightmarginal} follows immediately from the definition, since the analogue of \eqref{eq:weightmarginal} holds for each $\nu_k$.
We next prove that $\nu$ is invariant under lattice translations. We recall the definitions of the translation operators $\sigma_z$ from \eqref{eq:sigmazdef}.  

\begin{lemma}\label{sebastian}
For any $z \in \Z^2$, the measure $\nu$ is invariant under $\sigma_z$.
\end{lemma}

\begin{proof}
		Let $A$ be an arbitrary cylinder event in $\Omega_1 \times \Omega_2$:
		\[A = \{t_{e_i} \in \mathbf{B}_i, \, \eta(\bar f_j) = 1, \, \eta(\bar g_\ell) = 0,\, \, i = 1\, \ldots, i_0, \, j = 1, \ldots, j_0,\, \ell = 1, \ldots, \ell_0\} \]
		where each $e_i \in \Ec^2$ (resp.~$\bar f_j, \bar g_\ell \in \dedges^2$), and each $\mathbf{B}_i$ is a real Borel set. We will show that
		\begin{equation}
		\label{eq:shiftcyl}
		\lim_{m \to \infty} \nu_{k_m}(\sigma_{\be} A) = \lim_{m \to \infty} \nu_{k_m}(A)\quad \text{for each $\be \in \{\pm \be_1, \pm \be_2\}$},
		\end{equation}
	where the sequence $(\nu_{k_m})_m$ was chosen to converge to $\nu$ at  \eqref{eq:nudef}. Since such cylinder events generate the Borel sigma-algebra on $\Omega_1 \times \Omega_2$, and since each $\sigma_z$ can be written as a repeated composition of such $\sigma_{\be}$, the lemma follows. Since $\sigma_{-\be_i} = \sigma^{-1}_{\be_i}$ for $i  = 1,2$, it suffices to consider the cases $\be = \be_1, \be_2$.
	
		We denote by $\text{supp}(U_k)$ the support of the distribution of $U_k$ --- that is, $\{1, \ldots, v_k \cdot \be_1\} \times\{1, \ldots, L\}$.
		We write the translation operation explicitly:
		\begin{align}
		\nu_{k_m}(\sigma_{\be} A)&= (\P \times \text{Unif}^{(k_m)})\left( \forall i, j, \ell, \, t_{e_i - \be} \in \mathbf{B}_i,\, \eta_{k_m}(\bar f_j -  \be) = 1,\, \eta_{k_m}(\bar g_\ell - \be) = 0 \right)\nonumber\\
		&= \frac{1}{|\text{Supp}(U_{k_m})|} \sum_{u \in \mathrm{Supp}(U_{k_m})} \P \left(\forall i, j, \ell, \,  t_{e_i - \be} \in \mathbf{B}_i,\, [\bar f_j -  \be] \in \mathfrak{F}_{k_m}(u),\,[\bar g_j -  \be] \notin \mathfrak{F}_{k_m}(u) \right)\label{eq:expandprp}\\
		&=  \frac{1}{|\text{Supp}(U_{k_m})|} \sum_{u \in \mathrm{Supp}(U_{k_m})} \P \left(\forall i, j, \ell, \,  t_{e_i} \in \mathbf{B}_i,\, \bar f_j \in \mathfrak{F}_{k_m}(u+\be),\,\bar g_j \notin \mathfrak{F}_{k_m}(u+ \be) \right)\nonumber\\
		&= \frac{1}{|\text{Supp}(U_{k_m})|} \sum_{w \in [\be + \mathrm{Supp}(U_{k_m})]}  \P \left(\forall i, j, \ell, \,  t_{e_i} \in \mathbf{B}_i,\, \bar f_j \in \mathfrak{F}_{k_m}(w),\,\bar g_j \notin \mathfrak{F}_{k_m}(w) \right)\label{eq:expandprp2}\ .
		\end{align}
		By a similar expansion to the above, we have
		\begin{equation}
		\label{eq:shifttwo}
		\nu_{k_m}(A) = \frac{1}{|\text{Supp}(U_{k_m})|} \sum_{u \in \mathrm{Supp}(U_{k_m})}  \P \left(\forall i, j, \ell, \,  t_{e_i} \in \mathbf{B}_i,\, \bar f_j \in \mathfrak{F}_{k_m}(u),\,\bar g_j \notin \mathfrak{F}_{k_m}(u) \right)\ .
		\end{equation}
		
		We consider first the case $\be = \be_1$. Subtracting \eqref{eq:expandprp2} from \eqref{eq:shifttwo}, we see
		\begin{align*}
		|\nu_{k_m}(\sigma_{\be_1} A) - \nu_{k_m}( A)| &\leq |\text{Supp}(U_{k_m}) \triangle [\be_1 + \text{Supp}(U_{k_m})]|/L (v_{k_m} \cdot \be_1)\\
		&= 2 / (v_{k_m} \cdot \be_1).
		\end{align*}
		Since $v \in H_0$, we have $v_{k_m} \cdot \be_1 \to \infty$, establishing \eqref{eq:shiftcyl} for $\be = \be_1$.
		The case $\be = \be_2$ is simpler. Applying \eqref{eq:frakftrans}, we see the expressions in \eqref{eq:expandprp2} and \eqref{eq:shifttwo} are actually equal for each $m$.
		 Taking a limit in $m$, \eqref{eq:shiftcyl} follows for $\be = \be_2$, completing the proof.
\end{proof}

We recall that our aim is to show that $\nu$ is a  shift-invariant measure on non-crossing geodesics having infinitely many coalescence classes, in contradiction to Theorem \ref{thm:non-crossing}. We have proven shift-invariance; we next prove the bulleted portion from Definition~\ref{defin:shiftinvar}.
\begin{lemma}\label{maluco beleza}
Assume that~\eqref{raul seixas} holds. Then, for $\nu$-almost every $(\omega,\eta)\in\Omega_1\times\Omega_2$ we have
\begin{enumerate}[label = (\alph*)]
\item \label{it:raul1} every site has either out-degree 1 or both in- and out-degree 0 in $\eta$;
\item \label{it:raul2} every directed path in the graph encoded by $\eta$ is a (finite or infinite) geodesic.
    \item \label{it:raul3} there are no undirected cycles (and hence no directed cycles) in the graph encoded by $\eta$.
\end{enumerate}
\end{lemma}
Item \ref{it:raul3} says that the undirected graph with vertex set $\Z^2$, whose edge set is
\[\{ \{x,y\}: (x,y) \text{ and/or } (y,x) \text{ is in the directed graph encoded by $\eta$}\},\]
 has no cycles.
\begin{proof}
  We note that, for each $x \in \Z^2$, $\{x \text{ has out-degree $\geq 2$ in $\eta$} \}$ is a cylinder event in $\Omega_2$ (and hence both open and closed). Thus, for the sequence $(\nu_{k_m})$ from \eqref{eq:nudef},
  \begin{equation}\label{eq:outdegtwo}
    \lim_{m \to \infty} \nu_{k_m}\left( x \text{ has out-degree $\geq 2$ in $\eta$}\right) = \nu\left( x \text{ has out-degree $\geq 2$ in $\eta$} \right)\ .
  \end{equation}
  It therefore suffices to show that the probability on the left-hand side of \eqref{eq:outdegtwo} is zero for each $x$ and $m$. Since $\nu_{k_m}$ was defined as a pushforward of the mapping $\Psi_k$ from \eqref{eq:psidisjoint}, we need to show that each site a.s.~has out-degree at most one in $\eta_k(\omega, U_k)$ (defined just before \eqref{eq:psidisjoint}). Because $\eta_k$ is constructed from the edges of geodesics of $\mathfrak{F}_k (U_k)$, it suffices to show that for each $u$, the union of geodesics of $\mathfrak{F}_k(u)$ does not contain (at least) two directed edges of the form $(x, w_1)$, $(x,w_2)$.

  On the probability one event $\mathcal{X}$ from Definition \ref{defin:calx}, all geodesics are simple paths, so the only way for two such directed edges to be traversed by geodesics of $\mathfrak{F}_k(u)$ would be for them to be in different geodesics: $(x, w_1) \in \gamma_1$, $(x, w_2) \in \gamma_2$ for distinct elements $\gamma_1, \gamma_2 \in \mathfrak{F}_k(u)$. But by \eqref{eq:Inoncross}, this is impossible: a.s., for each pair $\gamma_1, \gamma_2$ of distinct geodesics of $\mathfrak{F}_k(u)$, we have $\gamma_1 \cap \gamma_2 = \varnothing.$ 
  
  We can argue similarly, establishing the remaining portion of \ref{it:raul1} by showing
  \begin{equation}
  \label{eq:badindeg}\lim_{m \to \infty} \nu_{k_m}(\text{$x$ has in-degree $0$ but out-degree $\geq 1$}) = 0. \end{equation}
   If $x$ has in-degree zero in $\eta_k$, then either $x$ does not lie in any element of $\mathfrak{F}_k(U_k)$, or $x$ is an initial point of some element of $\mathfrak{F}_k(U_k)$. In the former case, clearly $x$ also has out-degree $0$. In the latter case, we must have $x \cdot \be_1  = -U_k \cdot \be_1$, and this has probability at most $1 /v_k \cdot \be_1$.
  This shows \eqref{eq:badindeg} and completes the proof of item \ref{it:raul1} above. 

  We proceed with the proof of item \ref{it:raul2}.
  Let $\pi$ be a finite directed path with endpoints $x$ and $y$ in $\Z^2$. 
  We define the events 
  \begin{align}
    B_\pi &= \{(\omega, \eta): \pi \text{ is a path in the directed graph encoded by $\eta$}\}\ ,\nonumber\\
    C_\pi&=\{(\omega,\eta):\pi\text{ is a geodesic}\} = \{T(\pi) = T(x,y)  \}\ . \label{eq:Ctoport}
  \end{align}
  We first argue that
  \begin{equation}
    \label{eq:itbprelim}
    \nu_{k}\left(B_\pi^c\cup(B_\pi\cap C_\pi)\right)=1 \quad\text{for each $k$ and each $\pi$ as above.}
  \end{equation}
  To compute the probability in \eqref{eq:itbprelim}, we pull back via the mapping $\Psi_k$. The inverse image $\Psi_k^{-1}(B_\pi^c\cup(B_\pi\cap C_\pi)$ is the event that either $\pi$ does not appear in the graph encoded by $\eta_k(\omega, U_k)$, or does appear in this graph and is a geodesic for $\omega$. By the same reasoning as in the proof of part \ref{it:raul1} above, each vertex and directed edge of $\pi$ can a.s.~appear in at most one geodesic of $\mathfrak{F}_k(U_k)$. Hence, if $\pi$ appears in the graph encoded by $\eta_k$, it must be a subsegment of a single geodesic of $\mathfrak{F}_k(U_k)$, which proves \eqref{eq:itbprelim}. Since the set of finite directed paths in $\Z^2$ is countable, this proves item~\ref{it:raul2} is almost sure for $\nu_{k_m}$ (we use this fact in the proof of property \ref{it:raul3}).
  
 Taking $k = k_m$, we now argue that we can pass to the limit as $m \to \infty$ limit in \eqref{eq:itbprelim}. Since $B_\pi$ and its complement $B_\pi^c$ are cylinder events, they are both open and closed in $\Omega_2$.   Moreover, because $T$ is a continuous function of the edge weights, the latter representation in \eqref{eq:Ctoport} shows that $C_\pi$ is closed.
By the Portmanteau theorem, we conclude that $\nu\big(B_\pi^c\cup(B_\pi\cap C_\pi)\big)=1$, so that for $\nu$-a.e.~$(\omega,\eta)$, $\pi$ is either not a path in (the directed graph encoded by) $\eta$ or is both a path in $\eta$ and a geodesic in $\omega$. Applying countability of finite paths again establishes \ref{it:raul2} is  $\nu$-almost sure.

Finally, we prove item \ref{it:raul3}. We fix  $\pi = (z_1, \ldots, z_r ,z_1)$, an arbitrary finite  cycle,
and let $D_\pi$ denote the event that $\pi$ is a cycle in the (undirected version of the) graph encoded by $\eta$. Then $D_\pi$ is another cylinder event in $\Omega_2$, and hence both open and closed, so by the Portmanteau theorem $\nu(D_\pi) = \lim_m \nu_{k_m}(D_\pi)$.

We now argue that for each $m$, we have $\nu_{k_m}(D_\pi) = 0$. We fix an outcome $(\omega, \eta)$ in $D_\pi$ such that property \eqref{eq:outdegtwo} and property \ref{it:raul2} above hold, and such that $\mathcal{X}$ occurs; we will show that $(\omega, \eta)$ has contradictory properties. Since \eqref{eq:outdegtwo} and \ref{it:raul2} above, as well as $\mathcal{X}$, are $\nu_{k_m}$-almost sure, this shows that $\nu_{k_m}(D_\pi) = 0$.

Without loss of generality, we assume that the unique outgoing (directed) edge from $z_1$ in the graph encoded by $\eta$ is $(z_1, z_2)$. We then we follow along $\pi$ vertex by vertex  until the first time we either  i) return to $z_1$ or ii) reach a $z_\ell$ such that the directed edge $(z_{\ell + 1}, z_\ell)$ appears in the graph encoded by $\eta$. If i) occurs, we have traversed an oriented cycle beginning and ending at $z_1$; since this cycle is a geodesic, this contradicts the occurrence of $\mathcal{X}$. If ii) occurs, as in the proof of item \ref{it:raul2} of this lemma, the edges $(z_{\ell - 1}, z_{\ell})$ and $(z_{\ell + 1}, z_{\ell})$ must appear in exactly one geodesic $\gamma \in \mathfrak{F}_k(U_k)$, with the same orientation. In other words, $\gamma$ crosses the oriented edge $(z_{\ell - 1}, z_{\ell})$ and then the oriented edge $(z_{\ell + 1}, z_{\ell})$ (or vice-versa), in particular visiting $z_\ell$ at least twice. This is again impossible by the occurrence of $\mathcal{X}$.

Thus, $\nu_{k_m}(D_\pi) = 0$, and so $\nu(D_\pi) = 0$. Since $\pi$ was arbitrary, and since the set of finite cycles in $\Z^2$ is countable, item \ref{it:raul3} of the lemma follows.
\end{proof}

\begin{lemma}\label{cornelis}
  Assume that~\eqref{raul seixas} (and \eqref{assumption}) holds. Then
 $$
\nu\big(\text{the graph encoded by } \eta\text{ has infinitely many coalescence classes}\big)>0.
$$
In other words, we can, with uniformly positive probability, find arbitrarily large collections of disjoint infinite oriented paths (which are geodesics by Lemma~\ref{maluco beleza}) in this graph.
In particular, $\nu(\eta(\bar{e}) = 0 \text{ for all $\bar{e}$}) < 1$.
\end{lemma}

\begin{proof}
  Let $M \geq 1$ be a fixed integer, and define the event (on $\Omega_2$)
  \[A = A(M)= \left\{\text{the directed graph encoded by $\eta$ has at least $M$ coalescence classes}\right\}\ . \]
  As before, this means that we can find at least $M$ disjoint infinite oriented paths in the graph encoded by $\eta$.
  We will show that there is a $c > 0$ such that
  \begin{equation}
    \label{eq:nucoalM}
    \nu\left(\bigcap_{M=1}^\infty A(M)\right)  = \lim_{M \to \infty} \nu(A(M))  \geq c\ ,
  \end{equation}
   and the result follows. As usual, we will show \eqref{eq:nucoalM} by showing an analogue for $\nu_{k_m}$ and then taking $m \to \infty$. Graphs sampled from $\nu_{k_m}$ almost surely contain no infinite paths, however, and this requires us to choose an appropriate finite approximation of the event $A(M)$.

  Let $a > 0$ be a large integer depending only on $M$, to be fixed precisely by \eqref{eq:nucoalM2};
  we recall that $L$ was fixed as a constant in \eqref{raul seixas}. We define the subevent $A_a = A_a(M) \subseteq A(M)$ by
  \begin{align*}
    A_a = \left\{\exists\geq M \text{ disjoint directed infinite paths in $\eta$ from vertices of $\{0\} \times [0, aL]$} \right\}\ .
  \end{align*}
  We note that
  \begin{equation*}
      A_a = \bigcap_{R=aL}^\infty \left\{ \begin{array}{c}
      \text{there are at least } M \text{ paths in $\eta$ from vertices of $\{0\} \times [0, aL]$}\\ 
      \text{to vertices of }\Z^2 \setminus [-R, R]^2 \text{ which do not intersect} \end{array} \right\}
      =: \bigcap_{R=aL}^\infty A_{a,R}\ ,
\end{equation*}
where we have introduced the label $A_{a,R}$ for the event appearing in the $R$th term of the intersection in the middle expression.

In particular, $\nu(A_a) = \lim_{R \to \infty} \nu(A_{a,R}).$  Each $A_{a,R}$ is a cylinder event of $\Omega_2$, and so $\nu(A_{a,R}) = \lim_{m \to \infty} \nu_{k_m}(A_{a,R})$. The bound \eqref{eq:nucoalM} will thus follow once we have proven
\begin{equation}
    \label{eq:nucoalM2}
   \text{for each $M \geq 1$, there is an $a \geq 1$ such that}\quad \limsup_{R \to \infty} \limsup_{m \to \infty} \nu_{k_m}(A_{a,R})  \geq c
  \end{equation}
  for some uniform (in $M$) constant $c > 0$. We devote the rest of the proof to showing \eqref{eq:nucoalM2} with the choice $a = \lceil  3M / \delta' \rceil$, where $\delta'$ is as in \eqref{raul seixas}. We fix this value of $a$ in what follows.

  Because the $\Omega_2$-marginal of the measures $\nu_{k_m}$ is a pushforward of the mapping $\eta_k$ defined above \eqref{eq:psidisjoint}, the bound \eqref{eq:nucoalM2} will follow as soon as we show  that for each large value of $R$,
  \begin{equation}
    \begin{split}
    \label{eq:nucoalM3}
    &\text{for each $k$ and for each $u$ such that $1 \leq u\cdot \be_1 < v_k \cdot \be_1 - R,$ } \\
    &\P(\text{at least $M$ geodesics of $\mathfrak{F}_k(u)$ intersect $\{0\} \times [0, aL]$} ) \geq c\ ,
    \end{split}
  \end{equation}
  Indeed, the paths of $\mathfrak{F}_k(u)$ begin at vertices $x$ for which $\good_{k,L}^*(x)$ occurs (for our fixed choice of $L$), and hence are disjoint by Lemma \ref{lem:crosslocal}. For each $x \in I_k(u)$ (i.e., each initial vertex of an element of $\mathfrak{F}_k(u)$), we have $(x + v_k) \cdot \be_1 = v_k \cdot \be_1 - u \cdot \be_1$, and so $x + v_k \notin [-R, R]^2$ when $u$ is as in \eqref{eq:nucoalM3} --- hence, when $R > aL$, each such geodesic visits the complement of $[-R,R]^2$ after intersecting $\{0\} \times [0, aL]$.  Averaging over $u = U_k$ in \eqref{eq:nucoalM3} and observing that $\P^{(k)}(U_k \cdot \be_1  \geq v_k \cdot \be_1 - R) \to 0$ as $k \to \infty$ shows \eqref{eq:nucoalM2} and establishes the lemma.  We thus devote the remainder of the proof to showing \eqref{eq:nucoalM3}. We emphasize here the order in which parameters have been chosen, in anticipation of taking limits. $M$ and $a$ have been fixed; we will derive a bound on the probability in \eqref{eq:nucoalM3} depending on $R$ and $k$, where $R$ is large but fixed relative to $k \gg R$.

  We write, for each $\ell \in \Z$ (with fixed $u$, $k$, and $M$ as above)
  \begin{equation}
  \label{eq:Xelldef}
  X_\ell = \mathbf{1}_{\{\text{$\geq M$ geodesics of $\mathfrak{F}_k(u)$ intersect $\{0\} \times [aL\ell, aL (\ell + 1)]$}\}}\ ,
  \end{equation}
  so that the probability on the left-hand side of \eqref{eq:nucoalM3} is just $\E[X_0]$. Because the distribution of the family $\mathfrak{F}_k(u)$ is invariant under translations by $L \be_2$ (recalling \eqref{eq:frakftrans}), it follows that
  \begin{equation}
    \label{eq:Xell}
  \E[X_\ell] = \E[X_0] \quad \text{for each } \ell \in \Z. \end{equation}
Each geodesic $\gamma \in \mathfrak{F}_k(u)$ must cross the $\be_2$-axis $\{0\} \times \Z$ somewhere, since $\gamma$ begins at a point with negative $\be_1$-coordinate and ends at a point with positive $\be_1$-coordinate. We will use this to argue that many $X_\ell$'s must be positive; hence, by \eqref{eq:Xell} and the definition of $X_0$, \eqref{eq:nucoalM3} must hold.


To show this formally,  we introduce a new parameter $S$ (which will be taken to infinity for fixed $R$ and $k$), indexing the boxes
\begin{equation}
\label{eq:LambdaSdef}
\Lambda_S = [-L -u \cdot \be_1 ,\, L - u \cdot \be_1 +  v_k \cdot \be_1] \times [-a SL , a SL]\ ,\quad S \in \{1, 2, \ldots\}\ .  
\end{equation}
As noted in the previous paragraph, if $x \in I_{k}(u) \cap \Lambda_S$, then $\geo(x, x+v_k)$ intersects the $\be_2$-axis. If the first such intersection occurs outside $\Lambda_S$, then $\geo(x, x+v_k)$ must exit $\Lambda_S$. Definition \ref{defin:goodevent} ensures that such an exit can happen only through the top and bottom of $\Lambda_S$:
\[[-L -u\cdot \be_1 ,\, L - u\cdot \be_1 +  v_k \cdot \be_1] \times \{-aSL\} \text{ and } [-L -u \cdot \be_1 ,\, L - u \cdot \be_1 +  v_k \cdot \be_1] \times \{aSL\}\ . \]
Because the geodesics in $\mathfrak{F}_k(u)$ are disjoint, distinct elements of $\mathfrak{F}_k(u)$ must exit $\Lambda_S$ at distinct points (or not exit $\Lambda_S$ at all). In particular,
\begin{equation}
  \label{eq:topexit}
  \text{at most $C_1:= 4(L+1) + 2v_k \cdot \be_1$ geodesics of $\mathfrak{F}_k(u)$ can exit $\Lambda_S$}
\end{equation}
 where it is important only that the constant $C_1$ not depend on $a$, $u$, $R$, or $S$.

We now account for elements of $\mathfrak{F}_k(u)$ which do not exit $\Lambda_S$. If $X_\ell = 0$, then at most $M-1$ geodesics in $\mathfrak{F}_k(u)$ can intersect the segment $\{0\} \times [a L \ell, a L (\ell+1)]$; if $X_\ell = 1$, at most $a L + 1$ geodesics can intersect this segment. Thus, the number of geodesics of $\mathfrak{F}_k(u)$ which cross $\{0\} \times \Z$ before leaving $\Lambda_S$ is bounded above by
\begin{equation}
  \label{eq:sideexit}
  (2 S + 1)(M-1) + (a L + 1) \sum_{\ell = -S}^S X_\ell\ .
\end{equation}

Pulling together \eqref{eq:topexit}, \eqref{eq:sideexit}, and our observations about elements of $\mathfrak{F}_k(u)$ having to cross the $\be_2$-axis, we see
\begin{align*}
  |\{\gamma \in I_k(u):\, \gamma \cap \Lambda_S \neq \varnothing\}| \leq C_1+ 3MS + (a L + 1) \sum_{\ell = -S}^S X_\ell\quad \text{a.s.,}
\end{align*}
uniformly in $a$ and $S$.
Taking expectations and recalling that $\P([-u + j L \be_2] \in I_k(u)) = \P(\good_{k}^*)$, the above yields
\[(2a S+1) \prob(\good_k^*) \leq C_1 + 3MS + ( a L + 1) (2 S + 1) \E[X_0]\ , \]
where we have used \eqref{eq:Xell} to compare $\E[X_\ell]$ to $\E[X_0]$.
Last, we apply \eqref{raul seixas} to lower-bound the left-hand side, uniformly in $a$ and $S$:
\begin{equation}
  \label{eq:takealarge}(2a S+1) \delta' \leq C_1 + 3MS + (a L + 1) (2 S + 1) \E[X_0]\ ,
\end{equation}
where we recall that $\delta'$ is independent of $k$.

Dividing both sides of \eqref{eq:takealarge} by $aS$, we first take $S \to \infty$:
\begin{equation*}
  2\delta' \leq 3M/a + (2L + 1/a) \E[X_0]\ ,
\end{equation*}
and the $k$-dependent term $C_1$ has vanished. 
Using the choice $a = \lceil 3M/\delta' \rceil$ (from directly below \eqref{eq:nucoalM2}), the last display implies
\[\E[X_0] \geq \delta'/(4L + 2) \]
 uniformly in $k$ and in $u < v_k \cdot \be_1 - R$. This establishes \eqref{eq:nucoalM3} and completes the proof.\end{proof}

\subsection{Constructing a non-crossing family: the second case} \label{sec:finseg2}

We have shown that \eqref{assumption} leads to a contradiction under the additional assumption \eqref{raul seixas}. In this section, we show that \eqref{assumption} also leads to a contradiction when \eqref{raul seixas} does not hold, by using a modified version of the construction from Section \ref{sec:finseg}. This will be used to complete the proof of Theorem \ref{thm:highways} in Section \ref{sec:highwaysatlast}, where we pull together the present section with our earlier arguments.

In this section as well, \emph{we continue to enforce the convention from \eqref{eq:assumehalf}} that the half-plane $H_i$ appearing in the definition of $\good_{k,L}$ is $H_0$; this is purely for notational convenience. Unlike in Section \ref{sec:finseg2}, we assume that \eqref{raul seixas} does not hold; that is, in addition to \eqref{assumption}, we assume that, for some $L \geq L_0$ (where $L_0$ is from \eqref{eq:goodprob}) and for each $\ell \geq L$,
\begin{equation}
  \label{eq:raul2}
\liminf_{k \to \infty} \prob(\geo(0, v_k) \cap \geo(\ell \be_2, \ell \be_2 + v_k) \neq \varnothing \mid \good_{k,L} \cap \good_{k,L}(\ell \be_2)) = 1\ .
\end{equation}
The assumption \eqref{eq:raul2} will be in force for the remainder of this section, and the symbol $L$ will henceforth refer to the value appearing in \eqref{eq:raul2}. We write $\good^{\circ}_{k,L}(x, y)$ for the event appearing in \eqref{eq:raul2}, but with $0$ and $\ell \be_2$ replaced with arbitrary initial vertices. That is,
\begin{align}
  \good_{k,L}^\circ(x, y) = &\good_{k,L}(x) \cap \good_{k,L}(y)  \nonumber \\
  \cap &\{\geo(x, x+ v_k) \cap \geo(y, y + v_k) \neq \varnothing\}\ . \label{eq:goodcirc}
\end{align}


On the event $\good_{k,L}(x)$, the set
\[\Xi_x = \Xi_x(k) = \{\gamma \in \tree_x: \, \exists \gamma' \in \tree_{x + v_k} \text{ with } \gamma \subseteq \gamma' \}, \]
from \eqref{eq:Xixdef}, is nonempty. By item \ref{it:good3} of Definition \ref{defin:goodevent}, on $\good_{k,L}(x)$, we even have
\begin{equation}\label{eq:backboxer}
\begin{gathered}
  \text{$\Xi_x$ uniformly moves into $H_{i+4}$, and for all $\gamma \in \Xi_x$,}\\
  \text{we have $\gamma \subseteq H_{i+4} \cup [-L/4, L/4]^2$.}
  \end{gathered}
  \end{equation}
  It is also true that, for any $x \in \Z^2$, any $k \geq 1$, and any edge $e \in \edges^2$, we have
  \begin{align}
  \label{eq:backboxer2}
  \left\{e \in \bigcup_{\gamma \in \Xi_x}\gamma\right\} = \bigcap_{m=1}^{\infty} \left\{\begin{array}{c}e \text{ appears after } x \text{ in some geodesic}\\ \text{from $x + v_k$ of length at least $m$} \end{array}\right\} \text{ is measurable,}
  \end{align}
  where the equality follows by taking subsequential limits of finite geodesics.
This allows us to make the following definition:
\begin{definition} \label{defin:bigGamma}
  For each $x$ such that $\good_{k,L}(x)$ occurs, we define the measurable map $\Gamma_k(x):\, \Omega_1 \to \Omega_2$ to be the counterclockwisemost element of $\Xi_x$ in the half-plane $H_{i+4}$ (if $\good_{k,L}(x)$ does not occur, we set $\Gamma_k(x) = \varnothing$). By Lemma \ref{lem:extremal} (whose hypothesis holds by \eqref{eq:backboxer} and \eqref{eq:backboxer2}), $\Gamma_k(x)$ is well-defined, and indeed $\Gamma_k(x) \in \Xi_x$ (since $\tree_{x + v_k}$ is closed under geodesic limits). In particular, $\Gamma_k(x)$ lies entirely in the union of $x +H_{i+4}$ and $x + [-L/4,L/4]^2$.
\end{definition}

The construction from here on will parallel that of Section \ref{sec:finseg}, using the geodesics $\Gamma_k(x)$ (themselves already infinite) rather than the geodesics $\geo(0, v_{I(k)})$.
We begin by replacing $(v_k)$ by a subsequence for a final time, in order to simplify notation. Indeed, since \eqref{eq:raul2} holds, we can assume that
\begin{equation}
  \label{eq:vkthin}
  \min_{-4k \leq \ell_1, \ell_2 \leq 4k} \P(\good_{k,L}^\circ(\ell_1 L \be_2, \ell_2 L \be_2) \mid \good_{k,L}(\ell_1 L \be_2)  \cap \good_{k,L}(\ell_2 L \be_2)) \geq 1 - 1/k^3\ .
  \end{equation}
As in \eqref{eq:Fkdefin}, we define a family of relevant geodesics: for each $k$ and vertex $u \in \Z^2$, we let
\begin{equation}
\label{eq:Gfrakk}
\mathfrak{G}_k(u) = \{\Gamma_k(u + \ell  L \be_2 ):\, -4k \leq \ell \leq 4k\}\ .\end{equation}
This definition is in some ways simpler than that of \eqref{eq:Fkdefin} --- as we will see, the nonintersection of the $\Gamma_k$'s is in a sense automatic from \eqref{eq:vkthin}:
\begin{prop}\label{prop:nointG}
	For each $u \in \Z^2$, define the event  $\mathcal{N}_k(u) \subseteq \Omega_1$ as follows:
	\[\mathcal{N}_k(u) = \{ \text{there exist $\gamma_1 \neq \gamma_2$ in $\mathfrak{G}_k(u)$ with $\gamma_1 \cap \gamma_2 \neq \varnothing$} \}\ . \]
		We then have $\lim_{k \to \infty}\P(\mathcal{N}_k(u)) = 0$, and indeed
		\[\limsup_{k \to \infty}\E[|\mathfrak{G}_k(u)| \mathbf{1}_{\mathcal{N}_k(u)}] = \limsup_{k \to \infty}\E[|\mathfrak{G}_k(0)| \mathbf{1}_{\mathcal{N}_k(0)}] < \infty\ .\]
	\end{prop}
	\begin{proof}
		We first note the following: for $\ell_1 \neq \ell_2$
		\begin{equation}\text{on $\good_{k,L}^\circ(\ell_1 L \be_2, \ell_2 L \be_2),$ we have $\Gamma_k(\ell_1L \be_2) \cap \Gamma_k(\ell_2 L \be_2) = \varnothing. $}\label{eq:goodisbetter} \end{equation}
		Indeed, if \eqref{eq:goodisbetter} did not hold, by translation invariance we could find some integer $\ell \neq 0$ and some $\omega \in \Omega_1$ on which $\geo(0, v_k) \cap \geo(\ell L \be_2, \ell L \be_2 + v_k) \ni x$ for some vertex $x$ and $\Gamma_k(0) \cap \Gamma_k(\ell L \be_2) \ni y$ for some vertex $y$. Since $\omega \in \good_{k, L}(0)$, we have by items~\ref{it:good2} and~\ref{it:good3} of Definition \ref{defin:goodevent} that $\ell L \be_2 \notin \geo(0, v_k) \cup \Gamma_k(0)$. Similarly, we have that $0 \notin \geo(\ell L \be_2, \ell L \be_2 + v_k) \cup \Gamma_k(\ell L \be_2)$. 
		
		Since $0$ is the only vertex shared by $\geo(0, v_k)$ and $\Gamma_k(0)$ (with a similar statement for $\ell L \be_2$), we see that $x \neq y$. Now we see that there are two distinct geodesics from $x$ to $y$: one consists of the segment of $\geo(0, v_k)$ from $x$ to $0$ appended to a portion of $\Gamma_k(0)$, and the other consists of a segment of $\geo(\ell L \be_2, \ell L \be_2 + v_k)$ appended to $\Gamma_k(\ell L \be_2)$. (These are distinct because one geodesic contains $0$ and the other does not). Since $\omega \in \good_{k,L}^\circ(0) \subseteq \mathcal{X}$, this nonuniqueness of geodesics leads to a contradiction, showing that \eqref{eq:goodisbetter} must in fact hold.

		The first part of the proposition now follows by \eqref{eq:goodisbetter}, \eqref{eq:vkthin}, and a union bound:
		\begin{align*}
		\sum_{\ell_1, \ell_2 = -4k}^{4k} \P(\Gamma_k(u + \ell_1 L \be_2) \cap \Gamma_k(u + \ell_2 L \be_2) \neq \varnothing) \leq (8k+1)^2 \times k^{-3}  \leq 128 k^{-1} \ ,
		\end{align*}
		proving the claim about $\P(\mathcal{N}_k(u))$. The other claim follows from the above estimate and the fact that $|\mathfrak{G}_k(u)| \leq 8k+1$ holds almost surely.
		\end{proof}

We now enlarge $\Omega_1$ slightly, analogously to \eqref{eq:omega1prime}. Recalling that the value of $L$ was fixed below \eqref{eq:raul2}, we set for each $k \geq 1$
\begin{equation}
  \label{eq:omega1prime2}
  \Omega^{(k)}_1 = \Omega_1 \times [\{0, 1, \ldots, \lfloor \delta k / 32 \rfloor \} \times \{-kL, \ldots, kL\}] \quad \text{ with measure } \P^{(k)} = \P \times \text{Unif}^{(k)} \ ,
\end{equation}
where  $\text{Unif}^{(k)}$ is uniform probability measure on the set $\{0, 1, \ldots, \lfloor \delta k / 32 \rfloor\} \times \{-kL, \ldots, kL\}$. In other words, $\Omega_1$ is enlarged to include a random variable independent of $(\omega_e)$ and  uniformly distributed on the vertices of $[0, \lfloor \delta k / 32 \rfloor] \times [-kL, kL]$. We denote a typical element of this space by $(\omega, U_k)$. Our definitions are slightly more complicated than those of Section \ref{sec:finseg}: our family of geodesics $\mathfrak{G}_k(u)$ lacks the translation-invariance appearing in \eqref{eq:frakftrans}, and so requires somewhat ``more randomization'' when taking $k \to \infty$ to recover this invariance.

We recall the definition of $\Omega_2$ from \eqref{eq:allomegas}; we will as before define a sequence of mappings $\Omega_1^{(k)} \to \Omega_1 \times \Omega_2$.  The building blocks of these mappings will be the functions $\eta_k$, defined on $\Omega_1^{(k)}$ for each $k$, which take the following form:
\[\eta_{k}(\omega, u) = \left(\eta_{k}(v,w)(\omega, u): \, (v,w) \in \dedges^2\right) \in \{0, 1\}^{\dedges^2} = \Omega_2\ ,\] 
where  $\eta_{k}(v,w)(\omega, u)$  is the indicator of the event that the edge $(v,w)$ is traversed by some geodesic in $\mathfrak{G}_k(u):$
\[  \eta_{k}(v,w)(\omega, u)=
  \begin{cases}
    1 \quad &\text{if $(v,w)$ is in some geodesic $g \in \mathfrak{G}_k(u)$};\\
    0 \quad &\text{otherwise.}
  \end{cases}\]

Similarly to \eqref{eq:psidisjoint}, we define the map $\Phi_k:\Omega_1^{(k)} \to\Omega_1\times\Omega_2$ by setting
\begin{equation}
  \label{eq:psidisjoint2}
  \Phi_k(\omega) = (\omega, \eta_k(\omega, U_k))\ .\end{equation}
The measure $\upsilon_{k},$ defined on (the Borel sigma-algebra of) $\Omega_1 \times \Omega_2$ is then the pushforward
\[\upsilon_k = \P \circ (\Phi_k)^{-1}\ . \]
This is similar to the definition of $\nu_k$ above, though in addition to the change of randomization noted above, now $\upsilon_k$ is supported on configurations with infinite paths before even taking the $k \to \infty$ limit.
As before, $(\upsilon_k)_k$ is a tight sequence, and so we can set $\upsilon$ to be any subsequential weak limit:
\begin{equation}
  \label{eq:upsdef}\upsilon = \lim_{m \to \infty} \upsilon_{k_m} \quad \text{for some sequence $k_m \to \infty$}\ .\end{equation}
We will see that $\upsilon$ is supported on configurations of non-crossing geodesics. We first establish some basic properties of $\upsilon$, analogous to those of $\nu$, which follow by very similar arguments.

The analogue of \eqref{eq:weightmarginal} is again immediate, so we state it without proof:
\begin{equation}
\label{eq:weightmarginal2}
\upsilon(A \times \Omega_2) = \P(A) \quad \text{for any Borel } A \subseteq \Omega_1\ .
\end{equation}
We next establish translation-invariance of the measure $\upsilon$: 
\begin{lemma}\label{lem:upstrans}
	For any $z \in \Z^2$, the measure $\upsilon$ is invariant under $\sigma_z$.
	\end{lemma}
The argument is substantially similar to that used to show Lemma \ref{sebastian}, with minor notational complications. We therefore omit the proof.	
The next lemma is the analogue for $\upsilon$ of Lemma \ref{maluco beleza}.

\begin{lemma}\label{lem:maluco2}
	Assume that~\eqref{eq:raul2} holds. Then, for $\upsilon$-almost every $(\omega,\eta)\in\Omega_1\times\Omega_2$ we have
	\begin{enumerate}[label = (\alph*)]
		\item \label{it:paul1} every site has either out-degree 1 or both in- and out-degree 0 in $\eta$;
		\item \label{it:paul2} every directed path in the graph encoded by $\eta$ is a (finite or infinite) geodesic.
		\item \label{it:paul3} there are no undirected cycles (and hence no directed cycles) in the graph encoded by $\eta$.
	\end{enumerate}
\end{lemma}
As in the case of $\nu$, these reflect properties inherited from $\upsilon_k$. 
\begin{proof}
	We begin by proving \ref{it:paul1}, first showing that $\upsilon$-almost surely no site has out-degree at least two. As in the proof of \ref{it:raul1} of Lemma \ref{maluco beleza}, this follows once we show
	\begin{align}\label{eq:paul1}
	\lim_{m \to \infty} \P \times \text{Unif}^{(k)}(\text{there are distinct $\gamma_1, \gamma_2 \in \mathfrak{G}_k(U_k)$ with } \gamma_1 \cap \gamma_2 \neq \varnothing) = 0\ .  
	 \end{align}
	 Conditioning on $U_k = u$, the conditional probability of the event appearing in \eqref{eq:paul1} is exactly the probability of $\mathcal{N}_k(u)$ from Proposition \ref{prop:nointG}; by that proposition, the conditional probability tends to zero uniformly in $u$. Applying this in \eqref{eq:paul1}, the result follows.
	 
	 Again following the strategy of the proof of Lemma \ref{maluco beleza}, we show that
	 \begin{equation}
	 \label{eq:upsdeg}
	 \lim_{m \to \infty} \upsilon_{k_m}(0 \text{ has in-degree $0$ and out-degree at least $1$}) = 0
	 \end{equation}
	 which (since the event in \eqref{eq:upsdeg} is a cylinder event, and since $\upsilon$ is translation-invariant by Lemma~\ref{lem:upstrans}) completes the proof of~\ref{it:paul1}. We recall that the marginal of $\upsilon_{k_m}$ on $\Omega_2$ is the distribution of the set of edges appearing in geodesics from $\mathfrak{G}_k(U_k)$. 
	 
	 In particular, the probability appearing in \eqref{eq:upsdeg} is at most
	 \[\P^{(k)}(0 \text{ is the initial vertex of an element of $\mathfrak{G}_k(U_k)$}) = 0. \]
	 Here the equality follows from the fact that if $x$ is the initial vertex of an element of $\mathfrak{G}_k(u)$, we have $x \cdot \be_1 = u \cdot \be_1$ and the fact that $U_k \cdot \be_1 \geq 1$ almost surely. This concludes the proof of \ref{eq:upsdeg} and thus of \ref{it:paul1}.
	 
	 The proofs of \ref{it:paul2} and \ref{it:paul3} are virtually identical to that of \ref{it:raul2} and \ref{it:raul3} of Lemma \ref{maluco beleza}, so we omit them.
\end{proof}

We now arrive at the ultimate goal of the construction in the second case: showing that samples from $\upsilon$ admit more coalescence classes than is allowed by Theorem \ref{thm:non-crossing}. This is the point of the following lemma, which is analogous to Lemma \ref{cornelis}.
\begin{lemma}\label{lem:cornelis2}
	Assume that~\eqref{eq:raul2} (and \eqref{assumption}) hold. Then
	$$
	\upsilon\big(\text{the graph encoded by } \eta\text{ has infinitely many coalescence classes}\big)>0.
	$$
	In particular, $\upsilon(\eta(\bar{e}) = 0 \text{ for all $\bar{e}$}) < 1$.
\end{lemma}

\begin{proof}
	By essentially identical manipulations to those used in the proof of Lemma \ref{cornelis}, we arrive at an analogue of \eqref{eq:nucoalM2}. We thus see that the lemma will be proved once we show the following:
	\begin{equation}
	\begin{gathered}
	\label{eq:upcoalM3}
	\text{ for each $M \geq 1$, if $a > 128M/\delta$, we have for all $k$ large }\\
	\P \times \text{Unif}^{(k)}(\text{at least $M$ geodesics of $\mathfrak{G}_k(U_k)$ intersect $\{0\} \times [0, aL]$} ) \geq \delta / (128 L).\ 
	\end{gathered}
	\end{equation}
	(We recall $L$ was fixed below \eqref{eq:raul2}.) 
	 Unlike in the case of \eqref{eq:nucoalM3}, though, we no longer try to get a $\P$-probability bound which is uniform in all $u$. Roughly speaking, we do this to avoid complications related to the lack of translation-invariance of the measures $\nu_k$.

	We introduce the following boxes, which are different from (but play a similar role to) the boxes $\Lambda_S$ from \eqref{eq:LambdaSdef}:
	\[\Lambda_{k}' = [0,\, \delta k/32 + L] \times [-kL, kL]\ .  \]
	Here as usual, $\delta$ is as in \eqref{assumption}.
  We introduce the collection of geodesics of $\mathfrak{G}_k(u)$ beginning in $\Lambda_k'$:
	\begin{equation}
	\label{eq:Gkdef}
	\begin{gathered}
	G_k(v) :=\{\Gamma_k(v +  \ell L \be_2): \, v + \ell L \be_2 \in \Lambda_k'\} \subseteq \mathfrak{G}_k(v), \text{ for}\\ v \in \{1, \ldots, \lfloor \delta k / 32\rfloor \} \times \{ -2kL, -2kL + L, \ldots, 2kL \}\ . 
	\end{gathered}
	\end{equation}
	
	We have defined $G_k(v)$ for some values of $v$ lying outside of $\mathrm{Supp}(U_k)$; we note also that $\mathfrak{G}_k(v) \setminus G_k(v)$ is potentially large. These facts allow us to accommodate some shifts of $U_k$, and they guarantee the following statement, which recovers a modicum of translation-invariance for $\upsilon_k$:
	\begin{equation}
	    \label{eq:almosttrans}
	    \begin{gathered}
	    \text{For $v \in \{1, \ldots, \lfloor \delta k / 32\rfloor \} \times \{ -2kL, -2kL + L, \ldots, 2kL \}$ and $j = -k, -k + 1, \ldots, k$,}\\
	    \text{$G_k(v)$ has the same distribution under $\P$ as $G_k(v + j L \be_2)$.}
	    \end{gathered}
	\end{equation}
	The fact \eqref{eq:almosttrans} follows immediately from the definitions of $\mathfrak{G}_k(v)$ and $G_k(v)$.

	For each $v$, each element of $G_k(v)$ must exit $\Lambda_k'$.
	On the event $\Omega_1 \setminus \mathcal{N}_k(v)$, elements of $G_k(v)$ intersect the boundary of $\Lambda_k'$ in distinct vertices, allowing us to relate the number of intersections to $|G_k(v)|$. Similarly to the case of $\nu$, we argue that not too many elements of $G_k(v)$ can exit $\Lambda_k'$ through their top, bottom, or right sides. 
	
	Indeed, no elements of $G_k(v)$ may exit through the right side, for any value of $v$:
		\[\text{if $\Gamma_k(v + \ell \be_2) \in G_k(v)$, then } \Gamma_k(v + \ell \be_2) \cap \left( \{\left\lfloor \delta k / 32\right\rfloor + L\} \times [-kL, kL]\right) = \varnothing. \]
    	This is because $\good_k(x))$ occurs whenever $\Gamma_k(x) \neq \varnothing$ (recalling Definition \ref{defin:bigGamma}). On the event $\Omega_1 \setminus \mathcal{N}_k(v)$, because elements of $G_k(v)$ exit $\Lambda_k'$ at distinct points, there cannot be too many exits through the top and bottom sides:
		\begin{equation*}
		\begin{gathered}
		\text{On $\Omega_1 \setminus \mathcal{N}_k(v)$, at most $\delta k / 32 + L+1$} \text{ elements of $G_k(u)$}
		\text{ intersect } [0, \left\lfloor \delta k / 32\right\rfloor + L] \times \{kL\} , \\ \text{ with an analogous statement for the bottom side of $\Lambda_k'$.}
		\end{gathered}
		\end{equation*}
		In particular, at most $\delta k/16 + 2(L+1)$ geodesics may exit $\Lambda_k'$ through the union of its top and bottom side.
		
			Applying the estimates of the last paragraph, we see
		\begin{equation}
		\begin{gathered}
		    \label{eq:leftsideexits}
		    \text{On $\Omega_1 \setminus \mathcal{N}_k(v)$, at least $|G_k(v)| - \delta k/16 - 2(L+1)$ vertices of } \{0\} \times [-kL, kL]\\
		    \text{are first exit points of elements of $G_k(v)$ from $\Lambda_k'$.}
		    \end{gathered}
		\end{equation}
		Here ``first exit point'' means the first vertex on the boundary of $\Lambda_k'$ encountered by a particular element of $G_k(v)$.
		
		We use the above work to lower bound the number of geodesics in $\mathfrak{G}_k(v)$ intersecting the left side of $\Lambda_k'$.
		Recalling our goal \eqref{eq:upcoalM3}, we define
		\[ Q_k :=  \left\{ \text{at least $M$ vertices of $\{0\} \times [0, aL)$ are first exit points of elements of $\mathfrak{G}_k(U_k)$} \right\}, \]
		The desired bound \eqref{eq:upcoalM3} will follow once we show
		\begin{equation}
		\label{eq:fraktoG}
		\P^{(k)}(Q_k) \geq \delta / (128 L) \text{ for each pair $M\geq 1$ and $a > 128M/\delta$, uniformly in all large $k$}.
		\end{equation}
		
		We do this by estimating the expected number of vertices on $[0, aL)$ which are first exit points as in the definition of $Q_k$. Writing $\E^{(k)}$ for expectation with respect to $\P^{(k)},$
		\begin{align}
		(2k/a) &\E^{(k)}[|\{x \in \{0\} \times [0, a L): x \text{ is the first exit point of some } \gamma \in \mathfrak{G}_k(U_k) \}| ]\nonumber\\
		     \geq &\sum_{r = -\lfloor k/a\rfloor - 1}^{\lfloor k/a \rfloor} \E^{(k)}\left[\left|\left\{ 
		        \begin{array}{c}x \in \{0\} \times [(r-1) a,  ra L): \, x \text{ is the first exit} \\\text{point of some } \gamma \in \mathfrak{G}_k(U_k- (r-1) a L \be_2) 
		         \end{array}
		     \right\}\right|\right]\nonumber\\
		    \geq &\sum_{r = -\lfloor k/a\rfloor - 1}^{\lfloor k/a \rfloor} \E^{(k)}\left[\left|\left\{
		        \begin{array}{c}
		        x \in \{0\} \times [(r-1)a,  ra L): \, x \text{ is the first exit}\\ \text{ point of some } \gamma \in G_k(U_k- (r-1) a L \be_2)
		        \end{array}
		    \right\}\right|\right]\nonumber\\
		    =&\sum_{r = -\lfloor k/a\rfloor - 1}^{\lfloor k/a \rfloor} \E^{(k)}\left[\left|\left\{
		        \begin{array}{c}
		        x \in \{0\} \times [(r-1)a,  ra L): \, x \text{ is the first exit} \\\text{point of some } \gamma \in G_k(U_k)
		        \end{array}      
		  \right\}\right|\right] \nonumber\\
		    \geq &\E^{(k)}\left[\left|\left\{\gamma \in G_k(U_k): \gamma \text{ exits $\Lambda_k'$ through its left side } \right\}\right| \mathbf{1}_{\Omega_1 \setminus \mathcal{N}_k(U_k)} \right] - 2 aL\ , \label{eq:lasttelescope}
		\end{align}
		where in the second-to-last line we applied \eqref{eq:almosttrans}.
		
		We lower-bound the expectation in \eqref{eq:lasttelescope} using \eqref{eq:leftsideexits}, yielding
		\[\eqref{eq:lasttelescope} \geq \E^{(k)} \left[ |G_k(U_k)| \mathbf{1}_{\Omega_1 \setminus \mathcal{N}_k(U_k)}\right] - 2 aL - \delta k / 16 - 2(L+1)\ . \]
		Now, using Proposition~\ref{prop:nointG} and \eqref{eq:goodprob} to bound $\E [ |G_k(U_k)| \mathbf{1}_{\Omega_1 \setminus \mathcal{N}_k(U_k)}]$ gives
		\[
		\begin{gathered}
		\eqref{eq:lasttelescope} \geq \delta k / 32 \quad \text{ for all large $k$, so for such $k$ we have} \\
		    \E^{(k)}[|\{x \in \{0\} \times [0, a L): x \text{ is the first exit point of some } \gamma \in \mathfrak{G}_k(U_k) \}| ] \geq \delta a / 64.
		    \end{gathered}
		    \]

		The quantity inside the expectation in the last display is almost surely bounded above by $aL$, since there are only $aL$ vertices in $[0, aL)$. Breaking up this expectation depending on whether $Q_k$ occurs or not, we see
		\[\delta a / 64 \leq (M-1) + a L \P^{(k)}(Q_k)\ ,  \]
		From this, \eqref{eq:fraktoG} immediately follows. As noted at that equation, this suffices to show \eqref{eq:upcoalM3}, which in turn implies the statement of the lemma.
\end{proof}

\subsection{Proof of Theorem~\ref{thm:highways}}\label{sec:highwaysatlast}
We pull together the above arguments to complete the proof of Theorem~\ref{thm:highways}. Assuming that the theorem did not hold (in the form of \eqref{assumption} and its refinement~\eqref{eq:vktov}), we derive a contradiction.

We are led to consider two cases, depending on whether the additional assumption \eqref{raul seixas} holds. In case that \eqref{raul seixas} holds, we construct in \eqref{eq:nudef} a measure $\nu$ on $\Omega_1 \times \Omega_2$. Lemmas~\ref{sebastian} and \ref{maluco beleza} show that $\nu$ is a shift-invariant measure on non-crossing geodesics. In Lemma~\ref{cornelis}, we show that $\nu$ assigns positive probability to the set of configurations having infinitely many coalescence classes. This contradicts Theorem~\ref{thm:non-crossing} and shows that \eqref{assumption} and \eqref{raul seixas} cannot simultaneously hold.

If instead both \eqref{assumption} and the negation of \eqref{raul seixas} hold, we construct a measure $\upsilon$ on $\Omega_1 \times \Omega_2$ in \eqref{eq:upsdef}.  Via Lemmas~\ref{lem:upstrans} and \ref{lem:cornelis2}, we see that this $\upsilon$ is a shift-invariant measure on non-crossing geodesics. Then Lemma\ref{lem:maluco2} shows that this measure assigns positive probability to configurations exhibiting infinitely many coalescence classes, again contradicting Theorem~\ref{thm:non-crossing}. Thus, \eqref{assumption} leads to a contradiction; Theorem~\ref{thm:highways} is proved.
\qed


\bibliographystyle{plain}
\bibliography{ahlberg}

\end{document}